\documentclass{amsart}
\usepackage{amsmath,amsthm,amssymb,amsfonts}
\usepackage{mathdots}
\usepackage{latexsym}
\usepackage[all]{xy}
\usepackage{CJK}
\usepackage{bm}
\usepackage{graphicx}
\usepackage{subfigure}
\usepackage{mathrsfs}
\usepackage{float}
\usepackage{makecell}
\usepackage{multirow}
\usepackage{xcolor}
\usepackage{cite}
\usepackage{hyperref}
\numberwithin{equation}{section}
\newtheorem{theorem}{\bf{Theorem}}[section]
\newtheorem{proposition}[theorem]{\bf{Proposition}}

\newtheorem{corollary}[theorem]{\bf{Corollary}}
\newtheorem{lemma}[theorem]{\bf{Lemma}}
\newtheorem{remark}[theorem]{\bf{Remark}}

\newcommand{\mrgl}{\mathrm{GL}}

\newcommand{\mrm}{\mathrm{M}}
\newcommand{\mrn}{\mathrm{N}}

\newcommand{\mrdet}{\mathrm{det}}

\newcommand{\mrhom}{\mathrm{Hom}}

\newcommand{\mfo}{\mathfrak{o}}
\newcommand{\mfa}{\mathfrak{a}}
\newcommand{\mfb}{\mathfrak{b}}
\newcommand{\mfc}{\mathfrak{c}}

\newcommand{\mcc}{\mathcal{C}}

\newcommand{\bql}{\overline{\mathbb{Q}_{l}}}
\newcommand{\bzl}{\overline{\mathbb{Z}_{l}}}
\newcommand{\bfl}{\overline{\mathbb{F}_{l}}}

\newcommand{\bs}[1]{\boldsymbol{#1}}

\begin{document}
	\title{Supercuspidal representations of $\mathrm{GL}_{n}(F)$ distinguished by a unitary involution}
	\author{Jiandi Zou}
	
	\address{Universit\'e Paris-Saclay, UVSQ, CNRS, Laboratoire de Math\'ematiques de Versailles, 78000, Versailles, France.}	
	\email{jiandi.zou@ens.uvsq.fr}
    
	
	\begin{abstract}
		
		Let $F/F_{0}$ be a quadratic extension of non-archimedean locally compact fields of residue characteristic $p\neq 2$. Let $R$ be an algebraically closed field of characteristic different from $p$. For $\pi$ a supercuspidal representation of $G=\mathrm{GL}_{n}(F)$ over $R$ and $G^{\tau}$ a unitary subgroup of $G$ with respect to $F/F_{0}$, we prove that $\pi$ is distinguished by $G^{\tau}$ if and only if $\pi$ is Galois invariant. When $R=\mathbb{C}$ and $F$ is a $p$-adic field, this result first as a conjecture proposed by Jacquet was proved in 2010's by Feigon-Lapid-Offen by using global methods. Our proof is local and works for both complex representations and $l$-modular representations with $l\neq p$. We further study the dimension of $\mathrm{Hom}_{G^{\tau}}(\pi,1)$ and show that it is at most one.
		
	\end{abstract}
    
    \maketitle
	
	\tableofcontents
	
	\section{Introduction}
	
	Let $F/F_{0}$ be a quadratic extension of $p$-adic fields of residue characteristic $p$, and let $\sigma$ denote its non-trivial automorphism. For $G=\mathrm{GL}_{n}(F)$, let $\varepsilon$ be a \emph{hermitian matrix} in $G$, that is, $\sigma(\,^t\varepsilon)=\varepsilon$ with $\,^{t}$ denoting the transpose of matrices. We define $$\tau_{\varepsilon}(x)=\varepsilon\sigma(\,^{t}x^{-1})\varepsilon^{-1}$$ for any $x\in G$, called a \emph{unitary involution} on $G$. We fix $\tau=\tau_{\varepsilon}$ and we denote by $G^{\tau}$ the subgroup of $G$ consisting of the elements fixed by $\tau$, called the \emph{unitary subgroup} of $G$ with respect to $\tau$. For $\pi$ an irreducible smooth representation of $G$ over $\mathbb{C}$, Jacquet proposed to study the space of $G^{\tau}$-invariant linear forms on $\pi$, that is, the space $$\mathrm{Hom}_{G^{\tau}}(\pi,1).$$
	When the space is non-zero, he called $\pi$ \emph{distinguished by} $G^{\tau}$. For $n=3$ and $\pi$ supercuspidal, he proved in \cite{jacquet2001factorization} by using global argument, that $\pi$ is distinguished by $G^{\tau}$ if and only if $\pi$ is $\sigma$-invariant, that is, $\pi^{\sigma}\cong\pi$, where $\pi^{\sigma}:=\pi\circ\sigma$.
	Moreover he showed that this space is of dimension one as a complex vector space when the condition above is satisfied. Besides in \emph{ibid.}, he also sketched a similar proof when $n=2$ and $\pi$ is supercuspidal, to give the same criterion of being distinguished and the same dimension one theorem. Based on these results, he conjectured that in general, $\pi$ is distinguished by $G^{\tau}$ if and only if $\pi$ is $\sigma$-invariant. Moreover, it is also interesting to determine the dimension of the space of $G^{\tau}$-invariant linear forms which is not necessarily one in general. Under the assumption that $\pi$ is $\sigma$-invariant and supercuspidal, Jacquet further conjectured that the dimension is one.
	
	In addition, an irreducible representation $\pi$ of $G$ is contained in the image of quadratic base change with respect to $F/F_{0}$ if and only if it is $\sigma$-invariant (\cite{arthur1989simple}). Thus for irreducible representations, the conjecture of Jacquet gives a connection between quadratic base change and $G^{\tau}$-distinction.
	
	Besides the special case mentioned above, the following two evidences also support the conjecture. First we consider the analogue of the conjecture in the finite field case. For $\overline{\rho}$ an irreducible complex representation of $\mathrm{GL}_{n}(\mathbb{F}_{q^{2}})$, Gow \cite{gow1984two} proved that $\overline{\rho}$ is distinguished by the unitary subgroup $\mathrm{U}_{n}(\mathbb{F}_{q})$ if and only if $\overline{\rho}$ is isomorphic to its twist under the non-trivial element of $\mathrm{Gal}(\mathbb{F}_{q^{2}}/\mathbb{F}_{q})$. Under this condition, he also showed that the space of $\mathrm{U}_{n}(\mathbb{F}_{q})$-invariant linear forms is of dimension one as a complex vector space. In addition, Shintani \cite{shintani1976two} showed that there is a one-to-one correspondence between the set of irreducible representations of $\mathrm{GL}_{n}(\mathbb{F}_{q})$ and that of Galois invariant irreducible representations of $\mathrm{GL}_{n}(\mathbb{F}_{q^{2}})$, where the correspondence, called the \emph{base change map}, is characterized by a trace identity. Thus
	these two results relate the $\mathrm{U}_{n}(\mathbb{F}_{q})$-distinction to the base change map. Finally, when $\overline{\rho}$ is generic and Galois-invariant, Anandavardhanan and Matringe \cite{anandavardhanan2018test} recently showed that the $\mathrm{U}_{n}(\mathbb{F}_{q})$-average of Bessel function of $\overline{\rho}$ on the Whittaker model as a $\mathrm{U}_{n}(\mathbb{F}_{q})$-invariant linear form is non-zero. Since the space of $\mathrm{U}_{n}(\mathbb{F}_{q})$-invariant linear forms is of dimension one, their result gives us a concrete characterization of the space.
	
	The other evidence for the Jacquet conjecture is its global analogue. We assume $\mathcal{K}/\mathcal{K}_{0}$ to be a quadratic extension of number fields and we denote by $\sigma$ its non-trivial automorphism. We choose $\tau$ to be a unitary involution on $\mathrm{GL}_{n}(\mathcal{K})$, which also gives us an involution on $\mathrm{GL}_{n}(\mathbb{A}_{\mathcal{K}})$, still denoted by $\tau$ by abuse of notation, where $\mathbb{A}_{\mathcal{K}}$ denotes the ring of ad\`eles of $\mathcal{K}$. We denote by $\mathrm{GL}_{n}(\mathcal{K})^{\tau}$ (resp. $\mathrm{GL}_{n}(\mathbb{A}_{\mathcal{K}})^{\tau}$) the unitary subgroup of $\mathrm{GL}_{n}(\mathcal{K})$ (resp. $\mathrm{GL}_{n}(\mathbb{A}_{\mathcal{K}})$) with respect to $\tau$. For $\phi$ a cusp form of $\mathrm{GL}_{n}(\mathbb{A}_{\mathcal{K}})$, we define
	$$\mathcal{P}_{\tau}(\phi)=\int_{\mathrm{GL}_{n}(\mathcal{K})^{\tau}\backslash\mathrm{GL}_{n}(\mathbb{A}_{\mathcal{K}})^{\tau}}\phi(h)dh$$
	to be the \emph{unitary period integral} of $\phi$ with respect to $\tau$. We say that a cuspidal automorphic representation $\Pi$ of $\mathrm{GL}_{n}(\mathbb{A}_{\mathcal{K}})$ is $\mathrm{GL}_{n}(\mathbb{A}_{\mathcal{K}})^{\tau}$-distinguished if there exists a cusp form in the space of $\Pi$ such that $\mathcal{P}_{\tau}(\phi)\neq 0$. In 1990's, Jacquet and Ye began to study the relation between $\mathrm{GL}_{n}(\mathbb{A}_{\mathcal{K}})^{\tau}$-distinction and global base change (see for example \cite{jacquet1996distinguished} when $n=3$). For general $n$, Jacquet \cite{jacquet2005kloosterman} showed that $\Pi$ is contained in the image of the quadratic base change map (or equivalently $\Pi$ is $\sigma$-invariant \cite{arthur1989simple}) with respect to $\mathcal{K}/\mathcal{K}_{0}$ if and only if there exists a unitary involution $\tau$ such that $\Pi$ is $G^{\tau}$-distinguished. This result may be viewed as the global version of the Jacquet conjecture for supercuspidal representations.
	
	In fact, for the special case of the Jacquet conjecture in \cite{jacquet2001factorization}, Jacquet used the global analogue of the same conjecture and the relative trace formula to finish the proof. To say it simple, he first proved the global analogue of the conjecture. Then he used the relative trace formula to write a non-zero unitary period integral as the product of its local components at each place of $\mathcal{K}_{0}$, where each local component characterizes the distinction of the local component of $\Pi$ with respect to the corresponding unitary group over local fields. When $\pi$ is $\sigma$-invariant, he chose $\Pi$ to be a $\sigma$-invariant cuspidal automorphic representation of $\mathrm{GL}_{n}(\mathbb{A}_{\mathcal{K}})$ and $v_{0}$ to be a non-archimedean place of $\mathcal{K}_{0}$ such that $(G^{\tau},\pi)=(\mathrm{GL}_{n}(\mathcal{K}_{v_{0}})^{\tau},\Pi_{v_{0}})$. Then the product decomposition leads to the proof of the ``if" part of the conjecture. The ``only if" part of the conjecture, which will be discussed in section 4, requires the application of a globalization theorem. His method was generalized by Feigon-Lapid-Offen in \cite{feigon2012representations} to general $n$ and more general families of representations. They showed that the Jacquet conjecture works for generic representations of $G$. Moreover they were able to give a lower bound for the dimension of $\mathrm{Hom}_{G^{\tau}}(\pi,1)$ and they further conjectured that the inequality they gave is actually an equality. Finally, Beuzart-Plessis has recently verified the equality conjectured above \cite{beuzart2020multiplicities}. Thus for generic representations of $G$, the Jacquet conjecture was settled.
	
	Instead of using global methods, there are other methods to study this conjecture which are local and algebraic. Hakim-Mao \cite{hakim1998supercuspidal} verified the conjecture when $p\neq 2$ and $\pi$ is supercuspidal of level zero, that is, $\pi$ is supercuspidal such that $\pi^{1+\mathfrak{p}_{F}\mathrm{M}_{n}(\mathfrak{o}_{F})}\neq 0$, where $\mathfrak{o}_{F}$ denotes the ring of integers of $F$ and $\mathfrak{p}_{F}$ denotes its maximal ideal. When $\pi$ is supercuspidal and $F/F_{0}$ is unramified, Prasad \cite{prasad2001conjecture} proved the conjecture by applying the simple type theory developed by Bushnell-Kutzko in \cite{bushnell129admissible}. When $p\neq 2$ and $\pi$ is tame supercuspidal, that is, $\pi$ is a supercuspidal representation arising from the construction of Howe \cite{howe1977tamely}, Hakim-Murnaghan \cite{hakim2002tame} verified the conjecture. 
	
	The discussion above leaves us an open question: \emph{Is there any local and algebraic method that leads to a proof of the Jacquet conjecture which works for all supercuspidal representations of $G$?} First, this will lead to a new proof of the results of Hakim-Mao, Prasad and Hakim-Murnaghan which we mentioned in the last paragraph. Secondly, instead of considering complex representations, we are also willing to study $l$-modular representations with $l\neq p$. One hopes to prove an analogue of the Jacquet conjecture for $l$-modular supercuspidal representations, which will generalize the result of Feigon-Lapid-Offen for supercuspidal representations. Noting that they use global methods in their proof, which strongly relies on the assumption that all the representations are complex. Thus their method doesn't work anymore for $l$-modular representations. Finally, we are willing to consider $F/F_{0}$ to be a quadratic extension of non-archimedean locally compact fields instead of $p$-adic fields. Since the result of Feigon-Lapid-Offen heavily relies on the fact that the characteristic of $F$ equals 0, their method fails when considering non-archimedean locally compact fields of positive characteristic. The aim of this paper is to answer this question.
	
	We say a bit more about $l$-modular representations. The study of smooth $l$-modular representations of $G=\mathrm{GL}_{n}(F)$ was initiated by Vign\'eras \cite{vigneras1996representations}, \cite{vigneras2001correspondance} to extend the local Langlands program to $l$-modular representations. In this spirit, many classical results related to smooth complex representations of $p$-adic groups have been generalized to $l$-modular representations. For example, the local Jacquet-Langlands correspondence related to $l$-modular representations has been studied in detail in \cite{dat2011cas}, \cite{minguez2017correspondance} and \cite{secherre2019towards}. Thus, it is also natural to consider the $l$-modular version of the Jacquet conjecture, which hopes to build up the relation between distinction and an expected $l$-modular version of quadratic base change. This paper is the starting point of the whole project.
	
	To begin with, from now on we assume $F/F_{0}$ to be a quadratic extension of non-archimedean locally compact fields of residue characteristic $p\neq 2$ instead of $p$-adic fields. We fix $R$ an algebraically closed field of characteristic $l\neq p$, allowing that $l=0$. When $l>0$, we say that we are in the $l$-modular case (or modular case for short). Later on, we always consider smooth representations over $R$ and we assume $\pi$ to be a supercuspidal representation of $G$ over $R$. Be aware that when $l\neq 0$, a supercuspidal representation is not the same as a cuspidal representation of $G$, although they are the same when $l=0$ (see for example Vign\'eras \cite{vigneras1996representations}, chapitre II, section 2). Now we state our first main theorem:
	
	\begin{theorem}\label{Thmmain}
		
		For $\pi$ a supercuspidal representation of $G=\mathrm{GL}_{n}(F)$ and $\tau$ a unitary involution, $\pi$ is distinguished by $G^{\tau}$ if and only if $\pi^{\sigma}\cong\pi$.
		
	\end{theorem}
	
	Moreover, we may also calculate the dimension of the space of $G^{\tau}$-invariant linear forms:
	
	\begin{theorem}\label{Thmmult1}
		
		For $\pi$ a $\sigma$-invariant supercuspidal representation of $G$, we have $$\mathrm{dim}_{R}\mathrm{Hom}_{G^{\tau}}(\pi,1)=1.$$
		
	\end{theorem}
	
	One important corollary of Theorem \ref{Thmmain} relates to the $\overline{\mathbb{Q}_{l}}$-lift of a $\sigma$-invariant supercuspidal representation of $G$ over $\overline{\mathbb{F}_{l}}$ when $l>0$, where we denote by $\overline{\mathbb{Q}_{l}}$, $\overline{\mathbb{Z}_{l}}$ and $\overline{\mathbb{F}_{l}}$ the algebraic closure of an $l$-adic field, its ring of integers and the algebraic closure of the finite field of $l$ elements respectively. For $(\widetilde{\pi}, V)$ a smooth irreducible representation of $G$ over $\overline{\mathbb{Q}_{l}}$, we call it \emph{integral} if it admits an integral structure, that is, a $\overline{\mathbb{Z}_{l}}[G]$-submodule $L_{V}$ of $V$ generated by a $\bql$-basis of $V$. For such a representation, the semi-simplification of $L_{V}\otimes_{\overline{\mathbb{Z}_{l}}}\overline{\mathbb{F}_{l}}$ doesn't depend on the choice of $L_{V}$, which we denote by $r_{l}(\widetilde{\pi})$ as a representation of $G$ over $\overline{\mathbb{F}_{l}}$, called the \emph{reduction modulo $l$} of $\pi$ (see \cite{vigneras1996representations} for more details). The following theorem which will be proved at the end of section 8, says that it is always possible to find a $\sigma$-invariant $\overline{\mathbb{Q}_{l}}$-lift for a $\sigma$-invariant supercuspidal representation of $G$ over $\overline{\mathbb{F}_{l}}$.
	
	\begin{theorem}\label{Thmlift}
		
		For $\pi$ a $\sigma$-invariant supercuspidal representation of $G$ over $\overline{\mathbb{F}_{l}}$, there exists an integral $\sigma$-invariant supercuspidal representation $\widetilde{\pi}$ of $G$ over $\overline{\mathbb{Q}_{l}}$, such that $r_{l}(\widetilde{\pi})=\pi$.
		
	\end{theorem}
	
	Let us outline the contents of this paper by introducing the strategy of our proof for Theorem \ref{Thmmain} and Theorem \ref{Thmmult1}. In section 2 we introduce our settings and basic knowledge about hermitian matrices and unitary subgroups. Our main tool to prove the theorems will be the simple type theory developed by Bushnell-Kutzko in \cite{bushnell129admissible}, and further generalized by Vign\'eras \cite{vigneras1996representations} and M\'{\i}nguez-S\'echerre \cite{minguez2014types} to the $l$-modular case. In section 3 we will give a detailed introduction to this theory, but here we also recall a little bit for convenience. The idea of simple type theory is to realize any cuspidal representation $\pi$ of $G$ as the compact induction of a finite dimensional representation $\Lambda$ of $\boldsymbol{J}$, which is an open subgroup of $G$ compact modulo its centre. Such a pair $(\boldsymbol{J},\Lambda)$, constructed as in \cite{bushnell129admissible}, is called an \emph{extended maximal simple type}, which we will abbreviate to \emph{simple type} for simplicity. We also mention the following main properties of $(\boldsymbol{J},\Lambda)$:
	
	(1) The group $\boldsymbol{J}$ contains a unique maximal open compact subgroup $J$, which contains a unique maximal normal pro-$p$-subgroup $J^{1}$;
	
	(2) We have $J/J^{1}\cong\mathrm{GL}_{m}(\boldsymbol{l})$, where $E/F$ is a certain field extension of degree $d$ with $\boldsymbol{l}$ denoting the residue field of $E$, and $n=md$;
	
	(3) We may write $\Lambda=\boldsymbol{\kappa}\otimes\boldsymbol{\rho}$, where $\boldsymbol{\kappa}$ and $\boldsymbol{\rho}$ are irreducible representations of $\boldsymbol{J}$ such that the restriction $\boldsymbol{\kappa}|_{J^{1}}=\eta$ is an irreducible representation of $J^{1}$, called a \emph{Heisenberg representation}, and $\boldsymbol{\rho}|_{J}$ is the inflation of a cuspidal representation of $\mathrm{GL}_{m}(\boldsymbol{l})\cong J/J^{1}$.
	
	For a given supercuspidal representation $\pi$ of $G$, our starting point is to prove the ``only if" part of Theorem \ref{Thmmain}. When $R=\mathbb{C}$ and $\mathrm{char}(F)=0$, it is a standard result by using global argument, especially the globalization theorem (\cite{hakim2002globalization}, Theorem 1). When $\mathrm{char}(F)=p>0$, we may keep the original proof except that we need a characteristic $p$ version of the globalization theorem. Fortunately we can use a more general result due to Gan-Lomel\'{\i} \cite{ganglobalization} to get the result we need. Since any supercuspidal representation of $G$ over a characteristic 0 algebraically closed field can be realized as a representation over $\overline{\mathbb{Q}}$ up to twisting by an unramified character, we finish the proof when $\mathrm{char}(R)=0$. When $R=\overline{\mathbb{F}_{l}}$, we consider the projective envelope $P_{\Lambda|_{J}}$ of $\Lambda|_{J}$ and we use the results in \cite{vigneras1996representations} to study its irreducible components and the irreducible components of its $\overline{\mathbb{Q}_{l}}$-lift. Finally we will show that there exists a $\overline{\mathbb{Q}_{l}}$-lift of $\pi$ which is supercuspidal and $G^{\tau}$-distinguished. Thus by using the characteristic 0 case we finish the proof for the ``only if" part for any $R$ under our settings. The details will be presented in section 4.
	
	In section 5, we prove the \emph{$\tau$-selfdual type theorem}, which says that for a unitary involution $\tau$ and a $\sigma$-invariant cuspidal representation $\pi$ of $G$ with a technical condition, we may find a simple type $(\boldsymbol{J},\Lambda)$ contained in $\pi$ such that $\tau(\boldsymbol{J})=\boldsymbol{J}$ and $\Lambda^{\tau}\cong\Lambda^{\vee}$, where $^{\vee}$ denotes the smooth contragredient. In other words, we find a ``symmetric" simple type contained in $\pi$ with respect to $\tau$. Our strategy follows from \cite{anandavardhanan2019galois}, section 4. First we consider the case where $E/F$ is totally wildly ramified and $n=d$. Then for $E/F$ in general with $n=d$, we make use of the techniques about endo-class and tame lifting developed in \cite{bushnell1996local} to prove the theorem by reducing it to the former case. Finally by using the $n=d$ case, we prove the general theorem.
	
	In section 6, for $\tau$, $\pi$ as in section 5 satisfying the technical condition, we first choose a $\tau$-selfdual simple type $(\boldsymbol{J},\Lambda)$ contained in $\pi$. The main result of section 6, which we call the \emph{distinguished type theorem}, says that $\pi$ is distinguished by $G^{\tau}$ if and only if there exists a $\tau$-selfdual and distinguished simple type of $\pi$. More specifically, by Frobenius reciprocity and the Mackey formula, we have
	$$\mathrm{Hom}_{G^{\tau}}(\pi,1)\cong\prod_{g\in\boldsymbol{J}\backslash G/G^{\tau}}\mathrm{Hom}_{\boldsymbol{J}^{g}\cap G^{\tau}}(\Lambda^{g},1).$$
	We concentrate on those $g$ in the double coset such that $\mathrm{Hom}_{\boldsymbol{J}^{g}\cap G^{\tau}}(\Lambda^{g},1)\neq 0$. The proof of the distinguished type theorem also shows that there are at most two such double cosets which can be written down explicitly. Moreover for those $g$ we have
	$$\mathrm{Hom}_{\boldsymbol{J}^{g}\cap G^{\tau}}(\Lambda^{g},1)\cong\mathrm{Hom}_{\boldsymbol{J}^{g}\cap G^{\tau}}(\boldsymbol{\kappa}^{g},\chi^{-1})\otimes_{R}\mathrm{Hom}_{\boldsymbol{J}^{g}\cap G^{\tau}}(\boldsymbol{\rho}^{g},\chi),$$
	where $\bs{\kappa}$ is well-chosen such that $\boldsymbol{\kappa}^{\tau}\cong\boldsymbol{\kappa}^{\vee}$ and $\chi$ is a quadratic character of $\boldsymbol{J}^{g}\cap G^{\tau}$ which is trivial on $J^{1g}\cap G^{\tau}$. In the tensor product, the first term $\mathrm{Hom}_{\boldsymbol{J}^{g}\cap G^{\tau}}(\boldsymbol{\kappa}^{g},\chi^{-1})$ is of dimension one as an $R$-vector space. So essentially we only need to study the second term. If we denote by $\overline{\rho^{g}}$ the cuspidal representation of $\mathrm{GL}_{m}(\boldsymbol{l})\cong J^{g}/J^{1g}$ whose inflation equals $\boldsymbol{\rho}^{g}|_{J^{g}}$ , and by $\overline{\chi}$ the character of $H:=J^{g}\cap G^{\tau}/J^{1g}\cap G^{\tau}$ whose inflation equals $\chi|_{J^{g}\cap G^{\tau}}$, then we further have
	$$\mathrm{Hom}_{\boldsymbol{J}^{g}\cap G^{\tau}}(\boldsymbol{\rho}^{g},\chi)\cong\mathrm{Hom}_{H}(\overline{\rho^{g}},\overline{\chi}).$$
	Here $H$ could be a unitary subgroup, an orthogonal subgroup or a symplectic subgroup of $\mathrm{GL}_{m}(\boldsymbol{l})$. When $\pi$ is supercuspidal, the technical condition in the $\tau$-selfdual type theorem is always satisfied and we reduce our problem to studying the $H$-distinction of a supercuspidal representation of $\mathrm{GL}_{m}(\boldsymbol{l})$.
	
	Moreover at the beginning of section 6, we use the result in section 5 to extend $\sigma$ to a non-trivial involution on $E$. We write $E_{0}=E^{\sigma}$ and we deduce that $E/E_{0}$ is a quadratic extension. When $E/E_{0}$ is unramified, $H$ is a unitary subgroup. We first use the result of Gow \cite{gow1984two} to deal with the characteristic 0 case. For $\mathrm{char}(R)>0$, we consider the projective envolope as in section 4. When $E/E_{0}$ is ramified, $H$ is either an orthogonal subgroup or a symplectic subgroup. When $H$ is orthogonal, we use Deligne-Lusztig theory \cite{deligne1976representations}, precisely a formula given by Hakim-Lansky \cite{hakim2012distinguished} to calculate the dimension of $\mathrm{Hom}_{H}(\overline{\rho^{g}},\overline{\chi})$ when $\mathrm{char}(R)=0$. For $\mathrm{char}(R)>0$, we use again the same method as in section 4 to finish the proof. When $H$ is symplectic, by \cite{klyachko1984models} the space is always $\{0\}$. These two cases will be studied in section 7 and section 8 separately. Finally in section 9, we give a purely local proof of the main theorem of section 4.
	
	It is worth mentioning that in \cite{secherre2019supercuspidal}, S\'echerre studied the $\sigma$-selfdual supercuspidal representations of $G$ over $R$, with the same notation unchanged as before. He proved the following \emph{Dichotomy Theorem} and \emph{Disjunction Theorem}: For $\pi$ a supercuspidal representation of $G$, it is $\sigma$-selfdual (that is, $\pi^{\sigma}\cong\pi^{\vee}$) if and only if it is either distinguished by $\mathrm{GL}_{n}(F_{0})$ or $\omega$-distinguished, where $\omega$ denotes the unique non-trivial character of $F_{0}^{\times}$ trivial on $\mathrm{N}_{F/F_{0}}(F^{\times})$. The method we use in this paper is the same as that was developed in \emph{ibid.} For example, our section 5 corresponds to section 4 of \cite{anandavardhanan2019galois} and our section 6 corresponds to section 6 of \cite{secherre2019supercuspidal}, etc.
	
	We point out the main differences in our case to end this introduction. First in section 5 we will find that in a certain case, it is even impossible to find a hereditary order $\mathfrak{a}$ such that $\tau(\mathfrak{a})=\mathfrak{a}$ , which isn't a problem in section 4 of \cite{anandavardhanan2019galois}. That's why we need to add a technical condition in the main theorem of section 5 and finally verify it for supercuspidal representations. Precisely, for a $\sigma$-invariant supercuspidal representation, we first consider the unitary involution $\tau=\tau_{1}$ corresponding to the identity hermitian matrix $I_{n}$. In this case, we may use our discussion in section 5 to find a $\tau$-selfdual type contained in $\pi$ and we may further use our discussion in section 6 and section 7 to show that $m$ is odd when $E/E_{0}$ is unramified. This exactly affirms the condition we need, and we may repeat the procedure of section 5 and section 6 for general unitary involutions. This detouring argument also indicates that a $\sigma$-invariant cuspidal not supercuspidal representation does not always contain a $\tau$-selfdual simple type, which justifies that our supercuspidal (instead of cuspidal) assumption is somehow important.
	
	Furthermore in section 8, it is unclear if the character $\chi$ mentioned above can be realized as a character of $\boldsymbol{J}$ or not, thus cannot be assumed to be trivial a priori as in \cite{secherre2019supercuspidal}. It means that we need to consider a supercuspidal representation of the general linear group over finite field distinguished by a non-trivial character of an orthogonal subgroup instead of the trivial one. That's why the result of Hakim-Lansky (\cite{hakim2012distinguished} Theorem 3.11) shows up.
	
	Last but not least, in section 6 a large part of our results are stated and proved for a general involution instead of a unitary one. This provides the possibility of using the same method to study the distinction of supercuspidal representations of $G$ by other involutions. For instance, the similar problem for orthogonal subgroups was also considered by the author \cite{zou2020supercuspidal}.
	
	\textbf{Acknowledgements.} This work forms part of my PhD project. I thank my PhD advisor, Vincent S\'echerre, for starting me this project, for helpful discussions and suggestions, and in particular, for his careful reading and revising for the early versions of this paper. I thank EDMH for financial support of my PhD project and Universit\'e de Versailles for excellent research conditions. I thank an anonymous referee for his detailed report with pertinent advice included.
	
	\section{Notation and basic definitions}
	
	\subsection{Notation}
	
	Let $F/F_{0}$ be a quadratic extension of non-archimedean locally compact fields with residue characteristic $p\neq 2$ and let $\sigma$ be the unique non-trivial involution in the Galois group. Write $\mathfrak{o}_{F}$ (resp. $\mathfrak{o}_{F_{0}}$) for the ring of integers of $F$ (resp. $F_{0}$) and $\boldsymbol{k}$ (resp. $\boldsymbol{k}_{0}$) for the residue field of $F$ (resp. $F_{0}$). The involution $\sigma$ induces a $\boldsymbol{k}_{0}$-automorphism of $\boldsymbol{k}$ generating $\mathrm{Gal}(\boldsymbol{k}/\boldsymbol{k}_{0})$, still denoted by $\sigma$.
	
	Let $R$ be an algebraically closed field of characteristic $l\geq 0$ different from $p$. If $l>0$, then we are in the ``modular case".
	
	We fix a character
	$$\psi_{0}: F_{0}\rightarrow R^{\times}$$
	trivial on the maximal ideal of $\mathfrak{o}_{F_{0}}$ but not on $\mathfrak{o}_{F_{0}}$, and we define $\psi=\psi_{0}\circ\mathrm{tr}_{F/F_{0}}$.
	
	Let $G$ be the locally profinite group $\mathrm{GL}_{n}(F)$ with $n\geq 1$, equipped with the involution $\sigma$ acting componentwise. Let $\varepsilon$ be a \emph{hermitian matrix} in $G$, which means that $\varepsilon^{*}=\varepsilon$. Here $x^{*}:=\sigma(\,^{t}x)$ for any $x\in\mathrm{M}_{n}(F)$ with $\,^{t}$ denoting the transpose operator. Sometimes we write $\sigma_{t}(x):=x^{*}$ for any $x\in\mathrm{M}_{n}(F)$ to emphasize that $\sigma_{t}$ is an anti-involution on $\mathrm{M}_{n}(F)$ extending $\sigma$. For $\varepsilon$ hermitian and $g\in G$, we define $\tau_{\varepsilon}(g)=\varepsilon\sigma(\,^t g^{-1})\varepsilon^{-1}$, called the \emph{unitary involution} corresponding to $\varepsilon$. For $\tau=\tau_{\varepsilon}$ a fixed unitary involution, we denote by $G^{\tau}$ the corresponding \emph{unitary subgroup}, which consists of the elements of $G$ fixed by $\tau$.
	
	By \emph{representations} of a locally profinite group, we always mean smooth representations on an $R$-module. Given a representation $\pi$ of a closed subgroup $H$ of $G$, we write $\pi^{\vee}$ for the smooth contragredient of $\pi$. We write $\pi^{\sigma}$ and $\pi^{\tau}$ for the representations $\pi\circ\sigma$ and $\pi\circ\tau$ of groups $\sigma(H)$ and $\tau(H)$ respectively. We say that $\pi$ is \emph{$\tau$-selfdual} if $H$ is $\tau$-stable and $\pi^{\tau}$ is isomorphic to $\pi^{\vee}$. We say that $\pi$ is \emph{$\sigma$-invariant} if $H$ is $\sigma$-stable and $\pi^{\sigma}$ is isomorphic to $\pi$. For $g\in G$, we write $H^{g}=\{g^{-1}hg|h\in H\}$ a closed subgroup and we write $\pi^{g}:x\mapsto\pi(gxg^{-1})$ a representation of $H^{g}$.
	
	For $\mathfrak{a}$ an $\mathfrak{o}_{F}$-subalgebra of $\mathrm{M}_{n}(F)$ and $\tau=\tau_{\varepsilon}$ a unitary involution, we denote by $$\tau(\mathfrak{a}):=\sigma_{\varepsilon}(\mathfrak{a})=\{\sigma_{\varepsilon}(x)|x\in\mathfrak{a}\}$$
	an $\mathfrak{o}_{F}$-subalgebra of $\mathrm{M}_{n}(F)$, where $\sigma_{\varepsilon}(x):=\varepsilon\sigma(\,^{t}x)\varepsilon^{-1}$ is an anti-involution for any $x\in\mathrm{M}_{n}(F)$. We say that $\mathfrak{a}$ is \emph{$\tau$-stable} if $\tau(\mathfrak{a})=\mathfrak{a}$. Moreover for $g\in G$, we obtain
	$$\tau(\mathfrak{a})^{g}=g^{-1}\sigma_{\varepsilon}(\mathfrak{a})g=\sigma_{\varepsilon}(\sigma_{\varepsilon}(g)\mathfrak{a}\sigma_{\varepsilon}(g^{-1}))=\sigma_{\varepsilon}(\tau(g)^{-1}\mathfrak{a}\tau(g))=\tau(\mathfrak{a}^{\tau( g)})$$
	In other words, the notation $\tau(\mathfrak{a})$ is compatible with $G$-conjugacy.
	
	For $\tau$ a unitary involution and $\pi$ a representation of $H$ as above, we say that $\pi$ is $H\cap G^{\tau}$-\emph{distinguished}, or just \emph{distinguished}, if the space $\mathrm{Hom}_{H\cap G^{\tau}}(\pi,1)$ is non-zero.
	
	An irreducible representation of $G$ is called \emph{cuspidal} (resp. \emph{supercuspidal}) if it doesn't occur as a subrepresentation (resp. subquotient) of a parabolically induced representation with respect to a proper parabolic subgroup of $G$.

	\subsection{Hermitian matrices and unitary groups}\label{subsectionunitary}
	
	We make use of this subsection to introduce basic knowledge of hermitian matrices and unitary groups. The references will be \cite{hakim1998supercuspidal} and \cite{jacobowitz1962hermitian}.
	
	Let $E/E_{0}$ be a quadratic extension of non-archimedean locally compact fields which are algebraic extensions of $F$ and $F_{0}$ respectively. Write $\mathfrak{o}_{E}$ for the ring of integers of $E$ and $\mathfrak{o}_{E_{0}}$ for that of $E_{0}$. Let $\sigma'\in\mathrm{Gal}(E/E_{0})$ be the unique non-trivial involution in the Galois group. For $\varepsilon'\in\mathrm{GL}_{m}(E)$, just as in the last subsection, we say that $\varepsilon'$ is a \emph{hermitian matrix} if $(\varepsilon')^{*}=\varepsilon'$, where we consider $(\cdot)^{*}$ as before with $n$, $F$, $F_{0}$, $\sigma$ replaced by $m$, $E$, $E_{0}$, $\sigma'$ respectively. Write $\varpi_{E}$ for a uniformizer of $E$ such that $$\sigma'(\varpi_{E})=
	\begin{cases}
		\varpi_{E} & \text{if}\ E/E_{0}\ \text{is unramified},\\
		-\varpi_{E} & \text{if}\ E/E_{0}\ \text{is ramified}.
	\end{cases}$$
	
	Let $\mathcal{X}$ denote the set of all the hermitian matrices in $\mrgl_{m}(E)$ for $E/E_{0}$. The group $\mrgl_{m}(E)$ acts on $\mathcal{X}$ by $g\cdot x=gxg^{*}$.
	
	\begin{proposition}[\cite{jacobowitz1962hermitian}, Theorem 3.1]\label{PropGOH}
		
		There are exactly two $\mathrm{GL}_{m}(E)$-orbits of $\mathcal{X}$ with respect to the action given above. Furthermore, the elements in each orbit are exactly determined by the classes of their determinants in $E_{0}^{\times}/\mathrm{N}_{E/E_{0}}(E^{\times})$.
		
	\end{proposition}
	
	We also consider the $\mathrm{GL}_{m}(\mathfrak{o}_{E})$-orbits of $\mathcal{X}$. We consider sequences $\alpha=(\alpha_{1},...,\alpha_{r})$ of certain triples $\alpha_{i}=(a_{i},m_{i},\delta_{i})$, such that $a_{1}>...>a_{r}$ is a decreasing sequence of integers, and $m_{1}+...+m_{r}=m$ is a partition of $m$ by positive integers, and $\delta_{1},...,\delta_{r}$ are elements of $E$ such that:
	
	(1) If $E/E_{0}$ is unramified, then $\delta_{i}=1$;
	
	(2) If $E/E_{0}$ is ramified and $a_{i}$ is odd, then $\delta_{i}=1$ and $m_{i}$ is even;
	
	(3) If $E/E_{0}$ is ramified and $a_{i}$ is even, then $\delta_{i}$ is either 1 or $\epsilon$, with $\epsilon\in \mathfrak{o}_{E_{0}}^{\times}-\mathrm{N}_{E/E_{0}}(\mathfrak{o}_{E}^{\times})$ fixed.\\
	For each $\alpha=(\alpha_{1},...,\alpha_{r})$ as above, we introduce a hermitian matrix $\varpi_{E}^{\alpha}=\varpi_{E}^{\alpha_{1}}\oplus...\oplus\varpi_{E}^{\alpha_{r}}$, where $\varpi_{E}^{\alpha_{i}}\in \mathrm{GL}_{m_{i}}(E)$ is a hermitian matrix, such that:
	
	(\romannumeral1) In the case (1), $\varpi_{E}^{\alpha_{i}}=\varpi_{E}^{a_{i}}I_{m_{i}}$;
	
	(\romannumeral2) In the case (2), $\varpi_{E}^{\alpha_{i}}=\varpi_{E}^{a_{i}}J_{m_{i}/2}$
		, where $J_{m_{i}/2}=\bigg(\begin{matrix}0 & I_{m_{i}/2} \\ -I_{m_{i}/2} & 0\end{matrix}\bigg)$;
		
		(\romannumeral3) In the case (3), $\varpi_{E}^{\alpha_{i}}=\varpi_{E}^{a_{i}}\mathrm{diag}(1,...,1,\delta_{i})$, where $\mathrm{diag}(*,...,*)$ denotes the diagonal matrix with corresponding diagonal elements.
		
		We state the following proposition which classifies all the $\mathrm{GL}_{m}(\mathfrak{o}_{E})$-orbits of $\mathcal{X}$.
		
		\begin{proposition}[\cite{jacobowitz1962hermitian}, Theorem 7.1, Theorem 8.2]\label{PropJOH}
			
			Each class of the $\mathrm{GL}_{m}(\mathfrak{o}_{E})$-orbits of $\mathcal{X}$ contains a unique representative of the form $\varpi_{E}^{\alpha}$ for a certain $\alpha$ as above.
			
		\end{proposition}

		Now we study unitary groups. For $\varepsilon'\in\mathcal{X}$, we denote by $\mathrm{U}_{m}(\varepsilon')$ the unitary group consisting of those $g\in \mathrm{GL}_{m}(E)$ such that $g\varepsilon' g^{*}=\varepsilon'$. We say that two unitary groups are \emph{equivalent} if and only if they are conjugate by some $g\in \mrgl_{m}(E)$. Since it is easy to check that $g\mathrm{U}_{m}(\varepsilon')g^{-1}=\mathrm{U}_{m}(g\varepsilon' g^{*})$, by Proposition \ref{PropGOH}, there are at most two equivalence classes of unitary groups, which are represented by $\mathrm{U}_{m}(E/E_{0}):=\mathrm{U}_{m}(I_{m})$ and $\mathrm{U}'_{m}(E/E_{0}):=\mathrm{U}_{m}(\varepsilon)$ for $\varepsilon=\mathrm{diag}(1,...,1,\epsilon)$, where $\epsilon\in E_{0}^{\times}-\mathrm{N}_{E/E_{0}}(E^{\times})$ is fixed.  
		
		\begin{remark}\label{rmkUequvclass}
			
			While we will not use it, we list the following result for completeness: $\mathrm{U}_{m}(E/E_{0})$ is equivalent to $\mathrm{U}'_{m}(E/E_{0})$ if and only if $m$ is odd. 
			
		\end{remark}

		\begin{remark}\label{remuniequi}
			
			In the future, we only consider the following two cases. First, we consider $E=F$, $E_{0}=F_{0}$, $m=n$ and $\sigma'=\sigma$. For any two unitary involutions with the corresponding hermitian matrices in the same $\mathrm{GL}_{n}(F)$-orbit, we already showed that the corresponding two unitary groups are equivalent. Since distinction is a property invariant up to equivalence of unitary groups, we may choose a hermitian matrix in its $\mrgl_{n}(F)$-orbit such that the corresponding unitary involution $\tau$ is simple enough to simplify the problem. Secondly, we consider $E$ as a finite field extension of $F$ determined by a cuspidal representation $\pi$ such that $n=m[E:F]$. We will find out that if $\pi^{\sigma}\cong\pi$, then we may find an involution $\sigma'$ on $E$ such that $E_{0}=E^{\sigma'}$ and $\sigma'|_{F}=\sigma$. So we may make use of the propositions in this subsection to study hermitian matrices and unitary groups of $\mathrm{GL}_{m}(E)$.
			
		\end{remark}
		
		\section{Preliminaries on simple types}
		
		In this section, we recall the main results we will need on simple strata, characters and types \cite{bushnell129admissible}, \cite{bushnell1996local}, \cite{bushnell2014effective}, \cite{vigneras1996representations}, \cite{minguez2014types}. We mainly follow the structure of \cite{anandavardhanan2019galois} and \cite{secherre2019supercuspidal}.
		
		\subsection{Simple strata and characters}\label{subsectionsimplestratum}
		
		Let $[\mathfrak{a},\beta]$ be a simple stratum in $\mathrm{M}_{n}(F)$ for a certain $n\geq 1$. Recall that $\mathfrak{a}$ is a hereditary order in $\mathrm{M}_{n}(F)$ and $\beta$ is in $G=\mathrm{GL}_{n}(F)$ such that:
		
		(1) the $F$-algebra $E=F[\beta]$ is a field with degree $d$ over $F$;
		
		(2) $E^{\times}$ normalizes $\mathfrak{a}^{\times}$.
		
		The centralizer of $E$ in $\mathrm{M}_{n}(F)$, denoted by $B$, is an $E$-algebra isomorphic to $\mathrm{M}_{m}(E)$ with $n=md$. The intersection $\mathfrak{b}:=\mathfrak{a}\cap B$ is a hereditary order in $B$.
		
		We denote by $\mathfrak{p}_{\mathfrak{a}}$ the Jacobson radical of $\mathfrak{a}$, and $U^{1}(\mathfrak{a})$ the compact open pro-$p$-subgroup $1+\mathfrak{p}_{\mathfrak{a}}$ of $G$. Similarly, we denote by $\mathfrak{p}_{\mathfrak{b}}$ the Jacobson radical of $\mathfrak{b}$ and $U^{1}(\mathfrak{b})$ the compact open pro-$p$-subgroup $1+\mathfrak{p}_{\mathfrak{b}}$ of $B^{\times}$. For any $x\in B^{\times}$, we have (\cite{bushnell129admissible}, Theorem 1.6.1)
		
		\begin{equation}\label{UUinB}
			U^{1}(\mathfrak{a})x U^{1}(\mathfrak{a})\cap B^{\times}=U^{1}(\mathfrak{b})x U^{1}(\mathfrak{b}).
		\end{equation}
		
		Associated with $[\mathfrak{a},\beta]$, there are open compact subgroups
		$$H^{1}(\mathfrak{a},\beta)\subset J^{1}(\mathfrak{a},\beta)\subset J(\mathfrak{a},\beta)$$
		of $\mathfrak{a}^{\times}$ and a finite set $\mathcal{C}(\mathfrak{a},\beta)$ of \emph{simple characters} of $H^{1}(\mathfrak{a},\beta)$ depending on the choice of $\psi$. We denote by $\boldsymbol{J}(\mathfrak{a},\beta)$ the subgroup of $G$ generated by $J(\mathfrak{a},\beta)$ and the normalizer of $\mathfrak{b}^{\times}$ in $B^{\times}$.
		
		\begin{proposition}[\cite{bushnell129admissible}, section 3]\label{PropJandJ1}
			
			We have the following properties:
			
			(1) The group $J(\mathfrak{a},\beta)$ is the unique maximal compact subgroup of $\boldsymbol{J}(\mathfrak{a},\beta)$;
			
			(2) The group $J^{1}(\mathfrak{a},\beta)$ is the unique maximal normal pro-$p$-subgroup of $J(\mathfrak{a},\beta)$;
			
			(3) The group $J(\mathfrak{a},\beta)$ is generated by $J^{1}(\mathfrak{a},\beta)$ and $\mathfrak{b}^{\times}$, and we have
			\begin{equation}
				J(\mathfrak{a},\beta)\cap B^{\times}=\mathfrak{b}^{\times}, J^{1}(\mathfrak{a},\beta)\cap B^{\times}= U^{1}(\mathfrak{b});
			\end{equation}
			
			(4) The normalizer of any simple character $\theta\in\mathcal{C}(\mathfrak{a},\beta)$ in $G$ is equal to $\boldsymbol{J}(\mathfrak{a},\beta)$;
			
			(5) The intertwining set of any $\theta\in\mathcal{C}(\mathfrak{a},\beta)$ in $G$, which we denote by $I_{G}(\theta)$, is equal to $J^{1}(\mathfrak{a},\beta)B^{\times}J^{1}(\mathfrak{a},\beta)=J(\mathfrak{a},\beta)B^{\times}J(\mathfrak{a},\beta)$.
			
		\end{proposition}
		
		\begin{remark}
			
			We write for short $J$, $J^{1}$, $H^{1}$ for $J(\mathfrak{a},\beta)$, $J^{1}(\mathfrak{a},\beta)$, $H^{1}(\mathfrak{a},\beta)$ respectively if $\mathfrak{a}$ and $\beta$ are clear to us.
			
		\end{remark}
		
		
		When $\mathfrak{b}$ is a maximal order in $B$, we call the simple stratum $[\mathfrak{a},\beta]$ and the simple characters in $\mathcal{C}(\mathfrak{a},\beta)$ \emph{maximal}. In this case, we may find an isomorphism of $E$-algebras $B\cong\mathrm{M}_{m}(E)$ which identifies $\mathfrak{b}$ with the standard maximal order, and moreover we have group isomorphisms
		\begin{equation}\label{eqGLml}
			J(\mathfrak{a},\beta)/J^{1}(\mathfrak{a},\beta)\cong\mathfrak{b}^{\times}/U^{1}(\mathfrak{b})\cong\mathrm{GL}_{m}(\boldsymbol{l}),
		\end{equation}
		where $\boldsymbol{l}$ denotes the residue field of $E$.
		
		\subsection{Simple types and cuspidal representations}\label{subsectiontypes}
		
		A pair $(\boldsymbol{J},\Lambda)$, called an \emph{extended maximal simple type} in $G$ (we always write \emph{simple type} for short) and constructed in \cite{bushnell129admissible} in the characteristic 0 case and in \cite{vigneras1996representations}, \cite{minguez2014types} in the modular case, is made of a subgroup $\boldsymbol{J}$ of $G$ which is open and compact modulo its centre, and an irreducible representation $\Lambda$ of $\boldsymbol{J}$. 
		
		Given a simple type ($\boldsymbol{J},\Lambda$) in $G$, there are a maximal simple stratum $[\mathfrak{a},\beta]$ in $\mathrm{M}_{n}(F)$ and a maximal simple character $\theta\in\mathcal{C}(\mathfrak{a},\beta)$, such that $\boldsymbol{J}(\mathfrak{a},\beta)=\boldsymbol{J}$ and $\theta$ is contained in the restriction of $\Lambda$ to $H^{1}(\mathfrak{a},\beta)$. Such a character $\theta$ is said to be \emph{attached to} $\Lambda$. By \cite{bushnell129admissible}, Proposition 5.1.1 (or \cite{minguez2014types}, Proposition 2.1 in the modular case), the group $J^{1}(\mathfrak{a},\beta)$ has, up to isomorphism, a unique irreducible representation $\eta$ whose restriction to $H^{1}(\mathfrak{a},\beta)$ contains $\theta$. Such a representation $\eta$, called the \emph{Heisenberg representation} associated to $\theta$, has the following properties:
		
		(1) the restriction of $\eta$ to $H^{1}(\mathfrak{a},\beta)$ is made of $(J^{1}(\mathfrak{a},\beta):H^{1}(\mathfrak{a},\beta))^{1/2}$ copies of $\theta$. Here $(J^{1}(\mathfrak{a},\beta):H^{1}(\mathfrak{a},\beta))^{1/2}$ is a power of $p$;
		
		(2) the direct sum of $(J^{1}(\mathfrak{a},\beta):H^{1}(\mathfrak{a},\beta))^{1/2}$ copies of $\eta$, which we denote by $\eta^{(J^{1}(\mathfrak{a},\beta):H^{1}(\mathfrak{a},\beta))^{1/2}}$, is isomorphic to $\mathrm{Ind}_{H^{1}}^{J^{1}}\theta$;
		
		(3) the representation $\eta$ extends to $\boldsymbol{J}$;
		
		(4) the intertwining set of $\eta$, which we denote by $I_{G}(\eta)$, equals $I_{G}(\theta)$;
		
		(5) for $h\in I_{G}(\eta)$, we have $\mathrm{dim}_{R}(\mathrm{Hom}_{J^{1}\cap J^{1h}}(\eta^{h},\eta))=1$.
		
		For any representation $\boldsymbol{\kappa}$ of $\boldsymbol{J}$ extending $\eta$, there exists a unique irreducible representation $\boldsymbol{\rho}$ of $\boldsymbol{J}$ trivial on $J^{1}(\mathfrak{a},\beta)$ such that $\Lambda\cong\boldsymbol{\kappa}\otimes\boldsymbol{\rho}$. Through (\ref{eqGLml}), the restriction of $\boldsymbol{\rho}$ to $J=J(\mathfrak{a},\beta)$ is identified with the inflation of a cuspidal representation of $\mathrm{GL}_{m}(\boldsymbol{l})$.
		
		\begin{remark}
			
			Recall that in \cite{bushnell129admissible}, Bushnell and Kutzko also assume $\kappa^{0}=\boldsymbol{\kappa}|_{J(\mathfrak{a},\beta)}$ to be a so called \emph{beta-extension}, which means that:
			
			(1) $\kappa^{0}$ is an extension of $\eta$;
			
			(2) if we denote by $I_{G}(\kappa^{0})$ the intertwining set of $\kappa^{0}$, then $I_{G}(\kappa^{0})=I_{G}(\eta)=I_{G}(\theta)$.
			
			However in our case, since $\mathrm{GL}_{m}(\boldsymbol{l})$ is not isomorphic to $\mathrm{GL}_{2}(\mathbb{F}_{2})$ ($p\neq2$), any character of $\mathrm{GL}_{m}(\boldsymbol{l})$ factors through the determinant. It follows that any representation of $J$ extending $\eta$ is a beta-extension. So finally our consideration of $\kappa^{0}$ coincides with the original assumption of Bushnell and Kutzko.
			
		\end{remark}
		
		
		
		
		
		We now give the classification of irreducible cuspidal representations of $G$ in terms of simple types (see \cite{bushnell129admissible}, \S 6.2, \S 8.4 and \cite{minguez2014types}, section 3 in the modular case).
		
		\begin{proposition}[\cite{bushnell129admissible},\cite{minguez2014types}]\label{PropType}
			
			Let $\pi$ be a cuspidal representation of $G$.
			
			(1) There is a simple type $(\boldsymbol{J},\Lambda)$ such that $\Lambda$ is a subrepresentation of the restriction of $\pi$ to $\boldsymbol{J}$. It is unique up to $G$-conjugacy;
			
			(2) Compact induction $c$-$\mathrm{Ind}_{\boldsymbol{J}}^{G}$ gives a bijection between the $G$-conjugacy classes of simple types and the isomorphism classes of cuspidal representations of $G$.
			
		\end{proposition}
		
		\subsection{Endo-classes, tame parameter fields and tame lifting}\label{subsectionendo}
		
		In this subsection, we introduce the concepts of endo-classes, tame parameter fields and tame lifting. The main references will be  \cite{bushnell129admissible}, \cite{bushnell1996local} and \cite{bushnell2014effective}.
		
		For $[\mathfrak{a},\beta]$ a simple stratum in $\mathrm{M}_{n}(F)$ and $[\mathfrak{a}',\beta']$ a simple stratum in $\mathrm{M}_{n'}(F)$ with $n,n'\geq 1$, if we have an isomorphism of $F$-algebras $\phi: F[\beta]\rightarrow F[\beta']$ such that $\phi(\beta)=\beta'$, then there exists a canonical bijection
		$$t_{\mathfrak{a},\mathfrak{a}'}^{\beta,\beta'}:\mathcal{C}(\mathfrak{a},\beta)\rightarrow\mathcal{C}(\mathfrak{a}',\beta'),$$
		called the \emph{transfer map} (see \cite{bushnell129admissible}, Theorem 3.6.14).
		
		Now let $[\mathfrak{a}_{1},\beta_{1}]$ and $[\mathfrak{a}_{2},\beta_{2}]$ be simple strata in $\mathrm{M}_{n_{1}}(F)$ and $\mathrm{M}_{n_{2}}(F)$ respectively with $n_{1},n_{2}\geq 1$. We call two simple characters $\theta_{1}\in\mathcal{C}(\mathfrak{a}_{1},\beta_{1})$ and $\theta_{2}\in\mathcal{C}(\mathfrak{a}_{2},\beta_{2})$ \emph{endo-equivalent}, if there are simple strata $[\mathfrak{a}',\beta_{1}']$ and $[\mathfrak{a}',\beta_{2}']$ in $\mathrm{M}_{n'}(F)$ for some $n'\geq 1$ such that $\theta_{1}$ and $\theta_{2}$ transfer to two simple characters $\theta_{1}'\in\mathcal{C}(\mathfrak{a}',\beta_{1}')$ and $\theta_{2}'\in\mathcal{C}(\mathfrak{a}',\beta_{2}')$ respectively which intertwine (or by \cite{bushnell129admissible}, Theorem 3.5.11 which are $\mathrm{GL}_{n'}(F)$-conjugate). This defines an equivalence relation on
		$$\bigcup_{[\mathfrak{a},\beta]}\mathcal{C}(\mathfrak{a},\beta),$$
		where the union runs over all simple strata of $\mathrm{M}_{n}(F)$ for all $n\geq 1$ (see \cite{bushnell1996local}, section 8). An equivalence class for this equivalence relation is called an \emph{endo-class}.
		
		For $\pi$ a cuspidal representation of $G=\mathrm{GL}_{n}(F)$, there exist a simple stratum $[\mathfrak{a},\beta]$ and a simple character $\theta\in\mathcal{C}(\mathfrak{a},\beta)$ contained in $\pi$. The set of simple characters $\theta$ contained in $\pi$ constitutes a $G$-conjugacy class, thus those simple characters are endo-equivalent. So we may denote by $\Theta_{\pi}$ the endo-class of $\pi$ which is the endo-class determined by any $\theta$ contained in $\pi$.
		
		Given $\theta\in\mathcal{C}(\mathfrak{a},\beta)$, the degree of $E/F$, its ramification index and its residue degree depend only on the endo-class of $\theta$. They are called the degree, ramification index and residue degree of this endo-class. Although the field extension $E/F$ is not uniquely determined, its maximal tamely ramified subextension is uniquely determined by the endo-class of $\theta$ up to $F$-isomorphism. This field is called a \emph{tame parameter field} of the endo-class (see \cite{bushnell2014effective}, \S 2.2, \S 2.4).
		
		We denote by $\mathcal{E}(F)$ the set of endo-classes of simple characters over $F$. Given a finite tamely ramified extension $T$ of $F$, we have a surjection
		$$\mathcal{E}(T)\rightarrow\mathcal{E}(F)$$
		with finite fibers, which is called a \emph{restriction map} (see \cite{bushnell2014effective}, \S 2.3). Given $\Theta\in\mathcal{E}(F)$, the endo-classes $\Psi\in\mathcal{E}(T)$ restricting to $\Theta$ are called the $T/F$-lifts of $\Theta$. If $\Theta$ has a tame parameter field $T$, then $\mathrm{Aut}_{F}(T)$ acts faithfully and transitively on the set of $T/F$-lifts of $\Theta$ (see \cite{bushnell2014effective}, \S 2.3, \S 2.4).
		
		Let $[\mathfrak{a},\beta]$ be a simple stratum and let $\theta\in\mathcal{C}(\mathfrak{a},\beta)$ be a simple character, let $T$ be the maximal tamely ramified extension of $F$ in $E$, and let $\Theta$ be the endo-class of $\theta$, then $T$ is a tame parameter field for $\Theta$. Let $C\cong\mathrm{M}_{n/t}(T)$ denote the centralizer of $T$ in $\mathrm{M}_{n}(F)$, where $t=[T:F]$. The intersection $\mathfrak{c}=\mathfrak{a}\cap C$ is an order in $C$ 
		which gives rise to a simple stratum $[\mathfrak{c},\beta]$. The restriction of $\theta$ to $H^{1}(\mathfrak{c},\beta)$, denoted by $\theta_{T}$, is a simple character associated to this simple stratum, called the \emph{interior $T/F$-lift} of $\theta$. Its endo-class, denoted by $\Psi$, is a $T/F$-lift of $\Theta$. For the origin and details of the construction of $\Psi$, see \cite{bushnell1996local}.
		
		For $T\subset\mrm_{n}(F)$ a tamely ramified subextension over $F$, the map
		$$\mathfrak{a}\mapsto\mathfrak{a}\cap C$$
		is injective from the set of hereditary orders of $\mathrm{M}_{n}(F)$ normalized by $T^{\times}$ to the set of hereditary orders of $C$ (see \cite{bushnell1996local}, section 2), where we still denote by $C$ the centralizer of $T$ in $\mrm_{n}(F)$. For $[\mfa,\beta_{1}]$, $[\mfa_{2},\theta_{2}]$ two simple strata, and $\theta_{1}\in\mcc(\mfa,\beta_{1})$, $\theta_{2}\in\mcc(\mfa,\beta_{2})$ two simple characters, such that $\theta_{1}$ and $\theta_{2}$ have the same tame parameter field $T$, if
		$$\mcc(\mfc,\beta_{1})=\mcc(\mfc,\beta_{2})\quad \text{and}\quad (\theta_{1})_{T}=(\theta_{2})_{T},$$ 
		then (see [BH96], Theorem 7.10, Theorem 7.15)
		$$\mcc(\mfa,\beta_{1})=\mcc(\mfa,\beta_{2})\quad \text{and}\quad \theta_{1}=\theta_{2}.$$ In particular, when $\beta_{1}=\beta_{2}=\beta$, the interior $T/F$-lift
		is injective from $\mcc(\mfa,\beta)$ to $\mcc(\mfc,\beta)$.
		
		\subsection{Supercuspidal representations}
		
		Let $\pi$ be a cuspidal representation of $G$. By Proposition \ref{PropType}, it contains a simple type $(\boldsymbol{J},\Lambda)$. Fix a maximal simple stratum $[\mathfrak{a},\beta]$ such that $\boldsymbol{J}=\boldsymbol{J}(\mathfrak{a},\beta)$, and write $\Lambda=\boldsymbol{\kappa}\otimes\boldsymbol{\rho}$ as in \S \ref{subsectiontypes}. Let $\overline{\rho}$ be the cuspidal representation of $J/J^{1}\cong\mathrm{GL}_{m}(\boldsymbol{l})$ whose inflation equals $\boldsymbol{\rho}|_{J}$. We have the following proposition:
		
		\begin{proposition}[\cite{vigneras1996representations}, Chapitre III, 5.14]
			
			The representation $\pi$ is supercuspidal if and only if $\overline{\rho}$ is supercuspidal.
			
		\end{proposition}

		\section{Distinction implies Galois invariance for a supercuspidal representation}
		
		Let $G=\mathrm{GL}_{n}(F)$ and let $G^{\tau}$ be the unitary group corresponding to a unitary involution $\tau$. We state the following theorem which is well-known when $R=\mathbb{C}$ and $\mathrm{char}(F)=0$ (see for example \cite{hakim2002globalization}, section 4, corollary or the earlier paper \cite{harder1986algebraische} which illustrates the idea).
		
		\begin{theorem}\label{Thmdistgalinv}
			
			Let $\pi$ be a supercuspidal representation of $G$. If $\pi$ is distinguished by $G^{\tau}$, then $\pi$ is $\sigma$-invariant.
			
		\end{theorem}
		
		Before proving Theorem \ref{Thmdistgalinv}, we state a useful lemma which will be used not only in the proof of the theorem, but also in the latter sections.
		
		\begin{lemma}\label{LemmaunitaryJJ0}
			
			For $\delta$ a unitary involution on $G$ and for $(\boldsymbol{J},\Lambda)$ a simple type in $G$, we have $\boldsymbol{J}\cap G^{\delta}=J\cap G^{\delta}$.
			
		\end{lemma}
		
		\begin{proof}
			
			For $x\in\boldsymbol{J}\cap G^{\delta}$, we have $\delta(x)=x$ which implies that $\sigma(\mathrm{det}(x))\mathrm{det}(x)=1$, where we denote by $\mathrm{det}(\cdot)$ the determinant function defined on $G$. Thus $\mathrm{det}(x)\in\mathfrak{o}_{F}^{\times}$. Since $\boldsymbol{J}=E^{\times}J$, we get $x\in\mathfrak{o}_{E}^{\times}J\cap G^{\delta}=J\cap G^{\delta}$. 
			
		\end{proof}
		
		Moreover, we need the following lemma which says that the properties of distinction and $\sigma$-invariance are maintained up to change of base fields.
		
		\begin{lemma}\label{lemmabasechange}
			
			Let $R_{1}\hookrightarrow R_{2}$ be a fixed embedding of two algebraically closed fields of characteristic $l\geq 0$. Let $\pi_{0}$ be a supercuspidal representation of $G$ over $R_{1}$. Let $\pi=\pi_{0}\otimes_{R_{1}}R_{2}$ be the corresponding representation of $G$ over $R_{2}$. Then:
			
			(1) $\pi_{0}$ is distinguished by $G^{\tau}$ if and only if $\pi$ is distinguished by $G^{\tau}$;
			
			(2) $\pi_{0}^{\sigma}\cong\pi_{0}$ if and only if $\pi^{\sigma}\cong\pi$.
			
			\begin{proof}
				
				For (1), let $(\boldsymbol{J},\Lambda_{0})$ be a simple type of $\pi_{0}$. Then $(\boldsymbol{J},\Lambda):=(\boldsymbol{J},\Lambda_{0}\otimes_{R_{1}}R_{2})$ is a simple type of $\pi$ and thus $\pi$ is also supercuspidal. Using Frobenius reciprocity and the Mackey formula\footnote{This argument will occur several times in this section, so we refer to the reader for more details in the proof of Theorem \ref{Thmdistgalinv}.}, 
				$$\mathrm{Hom}_{R_{1}[G^{\tau}]}(\pi_{0},1)\neq 0\Longleftrightarrow\ \text{There exists }\ g\in G\ \text{such that}\ \mathrm{Hom}_{R_{1}[\boldsymbol{J}^{g}\cap G^{\tau}]}(\Lambda_{0}^{g},1)\neq 0$$
				and
				$$\mathrm{Hom}_{R_{2}[G^{\tau}]}(\pi,1)\neq 0\Longleftrightarrow\ \text{There exists }\ g\in G\ \text{such that}\ \mathrm{Hom}_{R_{2}[\boldsymbol{J}^{g}\cap G^{\tau}]}(\Lambda^{g},1)\neq 0$$
				By Lemma \ref{LemmaunitaryJJ0}, $\boldsymbol{J}^{g}\cap G^{\tau}=J^{g}\cap G^{\tau}$ is a compact group, and $\Lambda_{0}^{g}$ is a representation of finite dimension. Thus 
				$$\mathrm{Hom}_{R_{1}[\boldsymbol{J}^{g}\cap G^{\tau}]}(\Lambda_{0}^{g},1)\otimes_{R_{1}}R_{2}\cong\mathrm{Hom}_{R_{2}[\boldsymbol{J}^{g}\cap G^{\tau}]}(\Lambda^{g},1)$$
				which finishes the proof of (1).
				For (2), from \cite{vigneras1996representations}, Chapitre I, 6.13, we know that $\pi_{0}$ is isomorphic to $\pi_{0}^{\sigma}$ if and only if their trace characters are equal up to a scalar in $R_{1}^{\times}$, which works similarly for $\pi$ and $\pi^{\sigma}$. Since the trace characters of $\pi_{0}$ and $\pi$ are equal up to the change of scalars, which works similarly for $\pi_{0}^{\sigma}$ and $\pi^{\sigma}$, we finish the proof of (2). \footnote{Note that if the trace characters of $\pi_{0}^{\sigma}$ and $\pi_{0}$ are equal up to a scalar in $R_{2}^{\times}$, then that scalar is in $R_{1}^{\times}$ since the trace of $\pi_{0}$ and $\pi_{0}^{\sigma}$ take values in $R_{1}$.}
				
			\end{proof}
			
		\end{lemma}
		
		\begin{proof}[Proof of Theorem \ref{Thmdistgalinv}]
			
			First we consider $R=\mathbb{C}$. If $\mathrm{char}(F)=0$, it is a standard result proved by using a global method (\cite{hakim2002globalization}, section 4, Corollary). Especially, their result is based on the globalization theorem, saying a distinguished $\pi$ under our settings can be realized as a local component of a cuspidal automorphic representation $\Pi$ of $\mathrm{GL}_{n}(\mathbb{A}_{\mathcal{K}})$, which is distinguished by a unitary subgroup of $\mathrm{GL}_{n}(\mathbb{A}_{\mathcal{K}})$ with respect to a quadratic extension of number fields $\mathcal{K}/\mathcal{K}_{0}$ (see \emph{ibid.}, Theorem 1). If $\mathrm{char}(F)>0$, in order to use the proof of Hakim-Murnaghan, we only need a variant of globalization theorem for characteristic positive case. Fortunately, Gan-Lomel\'{\i} already built up the globalization theorem for general reductive groups over function fields and locally compact fields of characteristic positive (see \cite{ganglobalization}, Theorem 1.3). Following their notations, we choose the reductive group $H$ to be $\mathrm{R}_{\mathcal{K}/\mathcal{K}_{0}}(\mathrm{GL}_{n}(\mathcal{K}))$, where $\mathcal{K}/\mathcal{K}_{0}$ is a quadratic extension of function fields, and  $\mathrm{R}_{\mathcal{K}/\mathcal{K}_{0}}$ is the Weil restriction. We choose $V$ to be $\mathrm{M}_{n}(\mathcal{K})$ as a $\mathcal{K}_{0}$-vector space and $\iota:H\rightarrow\mathrm{GL}(V)$ to be a representation over $\mathcal{K}_{0}$ defined by
			$$\iota(h)x=hx\sigma(\,^{t}h),\quad x\in V,\ h\in H, $$
			where $\sigma$ denotes the non-trivial involution in $\mathrm{Gal}(\mathcal{K}/\mathcal{K}_{0})$. If we choose $x_{0}\in V$ to be a hermitian matrix in $\mathrm{M}_{n}(\mathcal{K})$ and $H^{x_{0}}$ to be the stabilizer of $x_{0}$, then $H^{x_{0}}$ becomes a unitary subgroup of $H$ which satisfies the condition of
			\emph{loc. cit.} In order to use their result, we only need to verify the conditions (a) and (b) in their theorem. For condition (a), $\iota$ is semi-simple since it is the direct sum of two irreducible subrepresentations, composing of hermitian matrices and anti-hermitian matrices respectively \footnote{Here we need the assumption $p\neq 2$.}. For condition (b), since we only care about the case where $\chi=1$, it is automatically satisfied. Thus, if we use \cite{ganglobalization}, Theorem 1.3 to replace \cite{hakim2002globalization}, Theorem 1 and follow the proof in \cite{hakim2002globalization}, then we finish the proof when $R=\mathbb{C}$ and $F/F_{0}$ is a quadratic extension of locally compact fields of characteristic $p$.
			
			For $\mathrm{char}(R)=0$ in general, a supercuspidal representation of $G$ can be realized as a representation over $\overline{\mathbb{Q}}$ up to twisting by an unramified character, where $\overline{\mathbb{Q}}$ is the algebraic closure of $\mathbb{Q}$. More precisely, there exists a character $\chi:F^{\times}\rightarrow R^{\times}$ such that $\chi|_{\mathfrak{o}_{F}^{\times}}=1$ and $\pi\cdot\chi\circ\mrdet$ can be realized as a representation over $\overline{\mathbb{Q}}$. Since $\mrdet(G^{\tau})\subset\mfo_{F}^{\times}$ and $\chi\circ\mrdet|_{G^{\tau}}$ is trivial, $\pi$ is $G^{\tau}$-distinguished if and only if $\pi\cdot\chi\circ\mrdet$ is, as a representation over $R$, and also as a representation over $\overline{\mathbb{Q}}$ or $\mathbb{C}$ by Lemma \ref{lemmabasechange}.(1). Using the complex case, $\pi\cdot\chi\circ\mrdet$ is $\sigma$-invariant as a representation over $\mathbb{C}$, and also as a representation over $\overline{\mathbb{Q}}$ or $R$ by Lemma \ref{lemmabasechange}.(2). By definition, $\chi$ is $\sigma$-invariant, thus $\pi$ is also $\sigma$-invariant.
			
			For $R=\overline{\mathbb{F}_{l}}$, we write $\pi\cong\mathrm{c}$-$\mathrm{Ind}_{\boldsymbol{J}}^{G}\Lambda$ for a simple type $(\boldsymbol{J},\Lambda)$. Using the Mackey formula and Frobenius reciprocity, we have
			$$0\neq\mathrm{Hom}_{G^{\tau}}(\pi,1)\cong\prod_{g\in\boldsymbol{J}\backslash G/G^{\tau}}\mathrm{Hom}_{\boldsymbol{J}^{g}\cap G^{\tau}}(\Lambda^{g},1).$$
			Thus $\pi$ is distinguished if and only if there exists $g\in G$ such that $\mathrm{Hom}_{\boldsymbol{J}^{g}\cap G^{\tau}}(\Lambda^{g},1)\neq 0$. Let $\gamma=\tau(g)g^{-1}$ and let $\delta(x)=\gamma^{-1}\tau(x)\gamma$ for $x\in G$ which is also a unitary involution, then we have
			$$0\neq\mathrm{Hom}_{\boldsymbol{J}^{g}\cap G^{\tau}}(\Lambda^{g},1)\cong\mathrm{Hom}_{\boldsymbol{J}\cap G^{\delta}}(\Lambda,1)=\mathrm{Hom}_{J\cap G^{\delta}}(\Lambda^{0},1)\cong\mathrm{Hom}_{J}(\Lambda^{0},\mathrm{Ind}_{J\cap G^{\delta}}^{J}\overline{\mathbb{F}_{l}}),$$
			where $\Lambda^{0}=\Lambda|_{J}$ and we use the fact that $\boldsymbol{J}\cap G^{\delta}=J\cap G^{\delta}$ by Lemma \ref{LemmaunitaryJJ0}.
			
			We consider $P_{\Lambda^{0}}$ to be the projective envelope of $\Lambda^{0}$ as a $\overline{\mathbb{Z}_{l}}[J]$-module, where we denote by $\overline{\mathbb{Z}_{l}}$ the ring of integers of $\overline{\mathbb{Q}_{l}}$, then we have (\cite{vigneras1996representations}, Chapitre III, 4.28 and \cite{serre2012linear}, Proposition 42 for finite group case. Since $\Lambda^{0}$ is a smooth representation of the compact group $J$ of finite dimension, it can be regarded as a representation of a finite group.):
			
			(1) $P_{\Lambda^{0}}\otimes_{\overline{\mathbb{Z}_{l}}}\overline{\mathbb{F}_{l}}$ is the projective envelope of $\Lambda^{0}$ as a $\overline{\mathbb{F}_{l}}[J]$-module, which is indecomposable of finite length, with each irreducible component isomorphic to $\Lambda^{0}$. Thus $\mathrm{Hom}_{\overline{\mathbb{F}_{l}}[J]}(P_{\Lambda^{0}}\otimes_{\overline{\mathbb{Z}_{l}}}\overline{\mathbb{F}_{l}},\mathrm{Ind}_{J\cap G^{\delta}}^{J}\overline{\mathbb{F}_{l}})\neq 0$;

			(2) For $\widetilde{P_{\Lambda^{0}}}=P_{\Lambda^{0}}\otimes_{\overline{\mathbb{Z}_{l}}}\overline{\mathbb{Q}_{l}}$ the $\overline{\mathbb{Q}_{l}}$-lift of $P_{\Lambda^{0}}$, we have $\widetilde{P_{\Lambda^{0}}}\cong\bigoplus\widetilde{\Lambda^{0}}$, where $\widetilde{\Lambda^{0}}$ in the direct sum are $\overline{\mathbb{Q}_{l}}$-lifts of $\Lambda^{0}$ of multiplicity 1 (The multiplicity 1 statement is derived from counting the length of $P_{\Lambda^{0}}\otimes_{\overline{\mathbb{Z}_{l}}}\overline{\mathbb{F}_{l}}$, and the number of different $\widetilde{\Lambda^{0}}$ in $\widetilde{P_{\Lambda^{0}}}$, and then showing that they are equal. The argument is indicated in the proof of \cite{vigneras1996representations}, Chapitre III, 4.28, or more precisely, \emph{ibid.}, Chapitre III, Th\'eor\`eme 2.2 and Th\'eor\`eme 2.9);
			
			(3) In (2), each $(J,\widetilde{\Lambda^{0}})$ can be extended to a simple type $(\boldsymbol{J},\widetilde{\Lambda})$ of $G$ as a $\overline{\mathbb{Q}_{l}}$-lift of $(\boldsymbol{J},\Lambda)$ (\cite{vigneras1996representations}, Chapitre III, 4.29).
			
			Using (1), $\mathrm{Hom}_{\overline{\mathbb{F}_{l}}[J]}(P_{\Lambda^{0}}\otimes_{\overline{\mathbb{Z}_{l}}}\overline{\mathbb{F}_{l}},\mathrm{Ind}_{J\cap G^{\delta}}^{J}\overline{\mathbb{F}_{l}})\neq 0$. Since $P_{\Lambda^{0}}$ is a projective $\overline{\mathbb{Z}_{l}}[J]$-module, it is a free $\overline{\mathbb{Z}_{l}}$-module.
			Since $\mathrm{Ind}_{J\cap G^{\delta}}^{J}\overline{\mathbb{Z}_{l}}$ is a free $\overline{\mathbb{Z}_{l}}$-module, $$\mathrm{Hom}_{\overline{\mathbb{Z}_{l}}[J]}(P_{\Lambda^{0}},\mathrm{Ind}_{J\cap G^{\delta}}^{J}\overline{\mathbb{Z}_{l}})$$
			is a free $\overline{\mathbb{Z}_{l}}$-module. As a result,
			$$\mathrm{Hom}_{\overline{\mathbb{F}_{l}}[J]}(P_{\Lambda^{0}}\otimes_{\overline{\mathbb{Z}_{l}}}\overline{\mathbb{F}_{l}},\mathrm{Ind}_{J\cap G^{\delta}}^{J}\overline{\mathbb{F}_{l}})\cong\mathrm{Hom}_{\overline{\mathbb{Z}_{l}}[J]}(P_{\Lambda^{0}},\mathrm{Ind}_{J\cap G^{\delta}}^{J}\overline{\mathbb{Z}_{l}})\otimes_{\overline{\mathbb{Z}_{l}}}\overline{\mathbb{F}_{l}}\neq 0$$
			if and only if
			$$\mathrm{Hom}_{\overline{\mathbb{Z}_{l}}[J]}(P_{\Lambda^{0}},\mathrm{Ind}_{J\cap G^{\delta}}^{J}\overline{\mathbb{Z}_{l}})\neq 0$$
			if and only if
			$$\mathrm{Hom}_{\overline{\mathbb{Q}_{l}}[J]}(\widetilde{P_{\Lambda^{0}}},\mathrm{Ind}_{J\cap G^{\delta}}^{J}\overline{\mathbb{Q}_{l}})\cong\mathrm{Hom}_{\overline{\mathbb{Z}_{l}}[J]}(P_{\Lambda^{0}},\mathrm{Ind}_{J\cap G^{\delta}}^{J}\overline{\mathbb{Z}_{l}})\otimes_{\overline{\mathbb{Z}_{l}}}\overline{\mathbb{Q}_{l}}\neq 0,$$
			So there exists $\widetilde{\Lambda^{0}}$ as in condition (2) such that $\mathrm{Hom}_{\overline{\mathbb{Q}_{l}}[J]}(\widetilde{\Lambda^{0}},\mathrm{Ind}_{J\cap G^{\delta}}^{J}\overline{\mathbb{Q}_{l}})\neq 0$. Using (3), we may choose $(\boldsymbol{J},\widetilde{\Lambda})$ as an extension of $(J,\widetilde{\Lambda^{0}})$. We write $\widetilde{\pi}=c$-$\mathrm{Ind}_{\boldsymbol{J}}^{G}\widetilde{\Lambda}$ which is a supercuspidal representation of $G$ over $\overline{\mathbb{Q}_{l}}$. By using
			$$\mathrm{Hom}_{\boldsymbol{J}^{g}\cap G^{\tau}}(\widetilde{\Lambda}^{g},1)\cong\mathrm{Hom}_{\boldsymbol{J}\cap G^{\delta}}(\widetilde{\Lambda},1)=\mathrm{Hom}_{J\cap G^{\delta}}(\widetilde{\Lambda^{0}},1)\cong\mathrm{Hom}_{J}(\widetilde{\Lambda^{0}},\mathrm{Ind}_{J\cap G^{\delta}}^{J}\overline{\mathbb{Q}_{l}})\neq 0$$ and by the Mackey formula and Frobenius reciprocity as before, $\widetilde{\pi}$ is $G^{\tau}$-distinguished. Using the result of the characteristic 0 case, we have $\widetilde{\pi}^{\sigma}\cong\widetilde{\pi}$. By (3), $\widetilde{\Lambda}$ is a $\overline{\mathbb{Q}_{l}}$-lift of $\Lambda$. So $\widetilde{\pi}$ is a $\overline{\mathbb{Q}_{l}}$-lift of $\pi$ and we have $\pi^{\sigma}\cong\pi$.
			
			For $\mathrm{char}(R)=l>0$ in general, as in the characteristic zero case, there exists a character $\chi:F^{\times}\rightarrow R^{\times}$ such that $\chi|_{\mathfrak{o}_{F}^{\times}}=1$ and $\pi\cdot\chi\circ\mrdet$ can be realized as a representation over $\overline{\mathbb{F}_{l}}$. Similarly we deduce that $\pi$ is $G^{\tau}$-distinguished if and only if $\pi\cdot\chi\circ\mrdet$ is, as a representation over $R$, and also as a representation over $\overline{\mathbb{F}}_{l}$ by Lemma \ref{lemmabasechange}.(1). Using the case above, $\pi\cdot\chi\circ\mrdet$ is $\sigma$-invariant, as a representation over $\overline{\mathbb{F}}_{l}$, and also as a representation over $R$ by Lemma \ref{lemmabasechange}.(2). By definition, $\chi$ is $\sigma$-invariant, thus $\pi$ is also $\sigma$-invariant.
			
		\end{proof}
		
		\begin{remark}
			
			It is also possible to give a purely local proof (without using the result of the complex supercuspidal case) for this theorem which also works for cuspidal representations. Since our proof relies on the refinement of the results and the arguments in section 5-8, we leave it to the last section to avoid breaking the structure of the paper.
			
		\end{remark}
		
		
		
		
		\section{The $\tau$-selfdual type theorem}
		
		Let $G=\mathrm{GL}_{n}(F)$ and let $\tau$ be the unitary involution of $G$ corresponding to a hermitian matrix $\varepsilon$. Let $\pi$ be a cuspidal representation of $G$. We choose a maximal simple stratum $[\mathfrak{a},\beta]$ and a simple character $\theta\in\mathcal{C}(\mathfrak{a},\beta)$ contained in $\pi$. 
		
		\begin{lemma}\label{LemmatauE}
			
			If $\pi$ is $\sigma$-invariant, then we may choose the simple stratum above such that $\sigma(\,^{t}\beta)=\beta$. As a result, $\sigma_{t}$ (see section 2) is an involution defined on $E$ whose restriction to $F$ is $\sigma$.
			
		\end{lemma}
		
		Let $E_{0}=E^{\sigma_{t}}$, where $E=F[\beta]$ and $\beta$ is chosen as in Lemma \ref{LemmatauE}.
		
		\begin{theorem}\label{Thmtautheta}
			
			Let $\pi$ be a $\sigma$-invariant cuspidal representation of $G$ and let $\tau$ be a unitary involution. We also assume the following additional condition:
			
			\emph{If the hermitian matrix corresponding to $\tau$ is not in the same $G$-class as $I_{n}$ in $\mathcal{X}$ and if there exists a maximal simple stratum $[\mathfrak{a},\beta]$ as in Lemma \ref{LemmatauE} with a $\theta\in\mathcal{C}(\mathfrak{a},\beta)$ contained in $\pi$, such that the corresponding $E/E_{0}$ is unramified, then $m$ is odd.}
			
			Then there exist a maximal simple stratum $[\mathfrak{a}',\beta']$ and a simple character $\theta'\in\mathcal{C}(\mathfrak{a}',\beta')$ contained in $\pi$ such that:
			
			(1) $\tau(\beta')=\beta'^{-1}$;
			
			(2) $\tau(\mathfrak{a}')=\mathfrak{a}'$ and\footnote{For the definition of $\tau(\mathfrak{a}')$, see \S2.1. We will use the same notation for Theorem \ref{Thmendotautheta} and further proofs.} $\tau(H^{1}(\mathfrak{a}',\beta'))=H^{1}(\mathfrak{a}',\beta')$;
			
			(3) $\theta'\circ \tau=\theta'^{-1}$.
			
		\end{theorem}
		
		As a corollary of Theorem \ref{Thmtautheta}, we state the main theorem of this section:
		
		\begin{theorem}[The $\tau$-selfdual type theorem]\label{Thmselfdualtype}
			
			Under the same conditions as Theorem \ref{Thmtautheta}, there exists a simple type $(\boldsymbol{J},\Lambda)$ contained in $\pi$ such that $\tau(\boldsymbol{J})=\boldsymbol{J}$ and $\Lambda^{\tau}\cong\Lambda^{\vee}$.
			
		\end{theorem}
		
		In the following subsections, we will focus on the proof of the results stated.
		
		\subsection{Endo-class version of main results}
		
		To prove Theorem \ref{Thmtautheta} and Theorem \ref{Thmselfdualtype}, we consider their corresponding analogues for the endo-class. Let $\Theta$ be an endo-class over $F$. As mentioned in section 3, we write $d=\mathrm{deg}(\Theta)$. Moreover, its tame parameter field $T$ as a tamely ramified extension over $F$ is unique up to $F$-isomorphism.
		
		From the definition of the endo-class, we may choose a maximal simple stratum $[\mathfrak{a},\beta]$ and a simple character $\theta\in\mathcal{C}(\mathfrak{a},\beta)$ such that $\theta\in\Theta$. We denote by $\Theta^{\sigma}$ the endo-class of $\theta^{\sigma}$ which doesn't depend on the choice of $\theta$.
		We denote by $n$ the size of $\mathfrak{a}$, that is,  $\mathfrak{a}\hookrightarrow\mathrm{M}_{n}(F)$ as a hereditary order. We write $n=md$ with $m$ a positive integer. First of all, we have the following lemma as an endo-class version of Lemma \ref{LemmatauE} which will be proved in \S 5.4.
		
		\begin{lemma}\label{LemmaendotauE}
			
			If $\Theta^{\sigma}=\Theta$, then we may choose the simple stratum above such that $\sigma(\,^{t}\beta)=\beta$. As a result, $\sigma_{t}$ is an involution defined on $E$ whose restriction to $F$ is $\sigma$.
			
		\end{lemma}
		
		Let $E_{0}=E^{\sigma_{t}}$, where $E=F[\beta]$ and $\beta$ is chosen as in Lemma \ref{LemmaendotauE}. The following theorem as an endo-class version of Theorem \ref{Thmtautheta} says that we may adjust our choice of the simple stratum and simple character such that they are $\tau$-selfdual with respect to a unitary involution $\tau$:
		
		\begin{theorem}\label{Thmendotautheta}
			
			Let $\Theta\in\mathcal{E}(F)$ be an endo-class over $F$ such that $\Theta^{\sigma}=\Theta$. Let $\tau$ be a unitary involution of $G$. We also assume the following additional condition:
			
			\emph{If the hermitian matrix corresponding to $\tau$ is not in the same $G$-class as $I_{n}$ in $\mathcal{X}$ and if there exists a maximal simple stratum $[\mathfrak{a},\beta]$ as in Lemma \ref{LemmaendotauE} with a $\theta\in\mathcal{C}(\mathfrak{a},\beta)$ contained in $\Theta$, such that the corresponding $E/E_{0}$ is unramified, then $m=n/d$ is odd.}
			
			Then there exist a maximal simple stratum $[\mathfrak{a}',\beta']$ in $\mathrm{M}_{n}(F)$ and a simple character $\theta'\in\mathcal{C}(\mathfrak{a}',\beta')$ such that:
			
			(1) $\tau(\beta')=\beta'^{-1}$;
			
			(2) $\tau(\mathfrak{a}')=\mathfrak{a}'$ and $\tau(H^{1}(\mathfrak{a}',\beta'))=H^{1}(\mathfrak{a}',\beta')$;
			
			(3) $\theta'\in\Theta$ and $\theta'\circ\tau=\theta'^{-1}$.
			
		\end{theorem}
		
		Later we will focus on the proof of Lemma \ref{LemmaendotauE} and Theorem \ref{Thmendotautheta}. So before we begin our proof, we illustrate how this theorem implies Lemma \ref{LemmatauE}, Theorem \ref{Thmtautheta} and Theorem \ref{Thmselfdualtype}. First, we have the following important result due to Gelfand and Kazhdan (see \cite{bernstein1976representations}, Theorem 7.3 for the complex case and \cite{secherre2016modular}, Proposition 8.4 for the $l$-modular case):
		
		\begin{proposition}\label{PropGelKa}
			
			For $\pi$ an irreducible representation of $\mathrm{GL}_{n}(F)$, the representation defined by $g\mapsto\pi(\,^{t}g^{-1})$ is isomorphic to $\pi^{\vee}$.
			
		\end{proposition}
		
		For $\pi$ given as in Lemma \ref{LemmatauE}, if we denote by $\Theta_{\pi}$ the endo-class of $\pi$, then we get $\Theta_{\pi}^{\sigma}=\Theta_{\pi}$. So we may use Lemma \ref{LemmaendotauE} to get Lemma \ref{LemmatauE} and use Theorem \ref{Thmendotautheta} to get Theorem \ref{Thmtautheta}.
		
		Now we show that Theorem \ref{Thmtautheta} implies Theorem \ref{Thmselfdualtype}. Using Proposition \ref{PropGelKa}, we have $\pi^{\tau\vee}\cong\pi^{\sigma}\cong\pi$. Let $(\boldsymbol{J},\Lambda)$ be a simple type of $\pi$ containing $\theta'$, where $\theta'$ is obtained from Theorem \ref{Thmtautheta} such that $\theta'\circ\tau=\theta'^{-1}$. Thus $\tau(\boldsymbol{J})=\boldsymbol{J}$ since they are the $G$-normalizers of $\theta'\circ\tau$ and $\theta'^{-1}$ respectively. Since $\pi^{\tau\vee}\cong\pi$, it contains both $(\boldsymbol{J},\Lambda)$ and $(\boldsymbol{J},\Lambda^{\tau\vee})$. By Proposition \ref{PropType}, there exists $g\in G$ such that $(\boldsymbol{J},\Lambda^{\tau\vee})=(\boldsymbol{J}^{g},\Lambda^{g})$. Since $\Lambda^{\tau\vee}\cong\Lambda^{g}$  contains both $(\theta'\circ\tau)^{-1}=\theta'$ and $\theta'^{g}$ as simple characters, the restriction of $\Lambda^{g}$ to the intersection
		\begin{equation}\label{eqH1int}
			H^{1}(\mathfrak{a}',\beta')\cap H^{1}(\mathfrak{a}',\beta')^{g},
		\end{equation}
		which is a direct sum of copies of $\theta'^{g}$ restricting to (\ref{eqH1int}), contains the restriction of $\theta'$ to (\ref{eqH1int}). It follows that $g$ intertwines $\theta'$. By Proposition \ref{PropJandJ1}.(5), $g\in J(\mathfrak{a}',\beta')B'^{\times}J(\mathfrak{a}',\beta')$ with $B'$ the centralizer of $E'$ in $\mathrm{M}_{n}(F)$. Thus we may assume $g\in B'^{\times}$. From the uniqueness of the maximal compact subgroup in $\boldsymbol{J}$, we deduce that $\boldsymbol{J}^{g}=\boldsymbol{J}$ implies $J(\mathfrak{a}',\beta')^{g}=J(\mathfrak{a}',\beta')$. Intersecting it with $B'^{\times}$ implies that $\mathfrak{b}'^{\times g}=\mathfrak{b}'^{\times}$. Since $\mathfrak{b}'^{\times}$ is a maximal compact subgroup of $B'^{\times}\cong\mathrm{GL}_{m}(E')$ and $g\in B'^{\times}$, we deduce that $g\in E'^{\times}\mathfrak{b}'^{\times}\subset J(\mathfrak{a}',\beta')$. Thus $(\boldsymbol{J}^{g},\Lambda^{g})=(\boldsymbol{J},\Lambda)$, which finishes the proof of Theorem \ref{Thmselfdualtype}.
		
		Finally we state the following two lemmas which will be useful in our further proof:
		
		\begin{lemma}\label{Lemmasigmaendo}
			
			Let $[\mathfrak{a},\beta]$ be a maximal simple stratum in $\mathrm{M}_{n}(F)$ and let $\Theta$ be a $\sigma$-invariant endo-class over $F$, such that there exists $\theta\in\mathcal{C}(\mathfrak{a},\beta)$ a simple character in $\Theta$. Then $\theta\circ\tau$ and $\theta^{-1}$ are in the same endo-class. In particular, if the hereditary order $\mathfrak{a}$ is $\tau$-invariant, then $\theta\circ\tau$ conjugates to $\theta^{-1}$ by an element in $U(\mathfrak{a})$.
			
		\end{lemma}
		
		\begin{proof}
			
			We choose $\pi$ a cuspidal representation of $G$ containing $\theta$. Thus by definition, we have $\Theta_{\pi}=\Theta$. Using Proposition \ref{PropGelKa}, we have $\pi^{\tau}\cong\pi^{\sigma\vee}$. So $\theta\circ\tau\in\Theta_{\pi^{\tau}}=\Theta_{\pi^{\sigma\vee}}=\Theta_{\pi^{\vee}}^{\sigma}$ and $\theta^{-1}\in\Theta_{\pi^{\vee}}$. Since $\Theta^{\sigma}=\Theta$, we have $\Theta_{\pi^{\vee}}^{\sigma}=\Theta_{\pi^{\vee}}$, which means that $\theta\circ\tau$ and $\theta^{-1}$ are in the same endo-class. If $\tau(\mathfrak{a})=\mathfrak{a}$, then by definition of the endo-equivalence (\cite{bushnell1996local}, Theorem 8.7), $\theta\circ\tau$ intertwines with $\theta^{-1}$. By \cite{bushnell129admissible}, Theorem 3.5.11, $\theta\circ\tau$ conjugates to $\theta^{-1}$ by an element in $U(\mathfrak{a})$.
			
		\end{proof}
		
		The following lemma will be used to change the choice of a unitary involution up to $G$-action on its corresponding hermitian matrix.
		
		\begin{lemma}\label{lemGconjtau}
			
			Let $\tau=\tau_{\varepsilon}$ be the unitary involution on $G$ corresponding to a hermitian matrix $\varepsilon$, and let $[\mathfrak{a},\beta]$ be a maximal simple stratum in $\mathrm{M}_{n}(F)$ and let $\theta\in\mathcal{C}(\mathfrak{a},\beta)$ be a simple character, such that $$\tau(\mathfrak{a})=\mathfrak{a},\quad \theta\circ\tau=\theta^{-1}\quad (\text{and}\ \tau(\beta)=\beta^{-1}).$$
			Then for $\tau'=\tau_{\varepsilon'}$ the unitary involution corresponding to a hermitian matrix $\varepsilon'=g^{-1}\varepsilon\sigma(\,^{t}g^{-1})$, we have $$\tau'(\mathfrak{a}^{g})=\mathfrak{a}^{g},\quad \theta^{g}\circ\tau'=(\theta^{g})^{-1}\quad (\text{and}\ \tau'(\beta^{g})=(\beta^{g})^{-1}).$$
			
		\end{lemma}
		
		\begin{proof}
			
			The proof is just a simple calculation. We have
			\begin{align*}
				\tau'(\mathfrak{a}^{g})=\tau'(g^{-1})\tau'(\mathfrak{a})\tau'(g)=\tau'(g^{-1})\varepsilon'\varepsilon^{-1}\tau(\mathfrak{a})(\varepsilon'\varepsilon^{-1})^{-1}\tau'(g)=g^{-1}\tau(\mathfrak{a})g,
			\end{align*}
			where in the last step we use
			$$(\varepsilon'\varepsilon^{-1})^{-1}\tau'(g)=\varepsilon\sigma(\,^{t}g^{-1})\varepsilon'^{-1}=g.$$
			Since $\tau(\mathfrak{a})=\mathfrak{a}$, we get $\tau'(\mathfrak{a}^{g})=\mathfrak{a}^{g}$. The other two equations can be proved in a similar way.
		\end{proof}
		
		\subsection{The maximal and totally wildly ramified case}
		
		Now we focus on the proof of Theorem \ref{Thmendotautheta}. We imitate the strategy in \cite{anandavardhanan2019galois}, section 4 which first considered a special case, and then used tame lifting developed by Bushnell and Henniart \cite{bushnell1996local} and other tools developed by Bushnell and Kutzko \cite{bushnell129admissible} to generalize their result. In this subsection, we prove the following proposition as a special case of (2) and (3) of Theorem \ref{Thmendotautheta}:
		
		\begin{proposition}\label{propmaxwild}
			
			Let $[\mathfrak{a},\beta]$ be a simple stratum in $\mathrm{M}_{n}(F)$ and let $\theta\in\mathcal{C}(\mathfrak{a},\beta)$ such that $\theta\in\Theta$ with $\Theta$ a $\sigma$-invariant endo-class. Let $E/F$ be totally wildly ramified of degree $n$. Let $\tau=\tau_{1}$ with $\tau_{1}(x):=\sigma(\,^{t}x^{-1})$ for any $x\in G$. Then there exist a simple stratum $[\mathfrak{a}'',\beta'']$ and a simple character $\theta''\in\mathcal{C}(\mathfrak{a}'',\beta'')$ such that $(\mathfrak{a}'',\theta'')$  is $G$-conjugate to $(\mathfrak{a},\theta)$ with the property $\tau(\mathfrak{a}'')=\mathfrak{a}''$ and $\theta''\circ\tau=\theta''^{-1}$.
			
		\end{proposition}
		
		\begin{remark}
			
			In Proposition \ref{propmaxwild} we have $[E:F]=d=n$, which is a power of $p$ as an odd number.
			
		\end{remark}
		
		Up to $G$-conjugacy, we may and will assume $\mathfrak{a}$ to be standard (that is, $\mathfrak{a}$ is made of matrices with upper triangular elements in $\mathfrak{o}_{F}$ and other elements in $\mathfrak{p}_{F}$.). 
		
		\begin{lemma}\label{lemmataua}
			
			There exist $g_{1}\in G$ and $a_{1},...,a_{n}\in\mathfrak{o}_{F}^{\times}$ such that $$\tau(g_{1})g_{1}^{-1}=A:=\begin{pmatrix}0 & 0 & \ldots & 0 & a_{1}\\ 0 & \iddots & \iddots & a_{2} & 0 \\ \vdots & \iddots & \iddots & \iddots & \vdots \\ 0 & a_{n-1} & \iddots & \iddots & 0 \\ a_{n} & 0 & \ldots & 0 & 0 \end{pmatrix}.$$
			Moreover, if we define $\mathfrak{a}':=\mathfrak{a}^{g_{1}}$, then we have $\tau(\mathfrak{a}')=\mathfrak{a}'$.
			
		\end{lemma}
		
		\begin{proof}
			
			First we claim that we may choose $a_{i}\in\mathfrak{o}_{F}^{\times}$ such that $A$ is a hermitian matrix and $\mathrm{det}(A)\in\mathrm{N}_{F/F_{0}}(F^{\times})$. To do this, noting that $A^{*}=A$ if and only if $a_{i}=\sigma(a_{n+1-i})$ for $i=1,2,...,n$, we choose $a_{i}=\sigma(a_{n+1-i})$ for $i=1,2,...,(n-1)/2$ randomly but only to make sure that they are in $\mathfrak{o}_{F}^{\times}$ and we choose $a_{(n+1)/2}\in\mathfrak{o}_{F_{0}}^{\times}$ to make sure that $\mathrm{det}(A)\in\mathrm{N}_{F/F_{0}}(F^{\times})$. 
			
			Since $A$ is a hermitian matrix which is in the same $G$-orbit as $I_{n}$ by considering the determinant, using Proposition \ref{PropGOH}, there exists an element $g_{1}\in G$ such that $(g_{1}^{-1})^{*}g_{1}^{-1}=A$, which means that $\tau(g_{1})g_{1}^{-1}=A$. By definition $\tau(\mathfrak{a}')=\mathfrak{a}'$ if and only if $\tau(g_{1}^{-1})\tau(\mathfrak{a})\tau(g_{1})=g_{1}^{-1}\mathfrak{a}g_{1}$. Since $\mathfrak{a}^{*}=\,^{t}\mathfrak{a}$, we deduce that $\tau(\mathfrak{a}')=\mathfrak{a}'$ if and only if $A^{-1}\,^{t}\mathfrak{a}A=(\tau(g_{1})g_{1}^{-1})^{-1}\,^{t}\mathfrak{a}\tau(g_{1})g_{1}^{-1}=\mathfrak{a}$. From our choice of $A$ and the definition of $\mathfrak{a}$, this can be verified directly. 
			
		\end{proof}
		
		Now fix $g_{1}$ as in Lemma \ref{lemmataua}. We write $\theta'=\theta^{g_{1}}$ and $\beta'=\beta^{g_{1}}$. Since $\mathfrak{a'}=\mathfrak{a}^{g_{1}}$, we also have:
		
		(1) $U'^{i}:=U^{i}(\mathfrak{a}')=U^{i}(\mathfrak{a})^{g_{1}}$, where $U^{i}(\mathfrak{a}):=1+\mathfrak{p}_{\mathfrak{a}}^{i}$ for $i\geq 1$;
		
		(2) $J':=J(\mathfrak{a}',\beta')=J(\mathfrak{a},\beta)^{g_{1}}$;
		
		(3) $J'^{1}:=J^{1}(\mathfrak{a}',\beta')=J^{1}(\mathfrak{a},\beta)^{g_{1}}$;
		
		(4)  $\boldsymbol{J}':=\boldsymbol{J}(\mathfrak{a}',\beta')=\boldsymbol{J}(\mathfrak{a},\beta)^{g_{1}}$;
		
		(5) $H'^{1}:=H^{1}(\mathfrak{a}',\beta')=H^{1}(\mathfrak{a},\beta)^{g_{1}}$;
		
		(6) $M':=M^{g_{1}}$, where $M=\mathfrak{o}_{F}^{\times}\times...\times\mathfrak{o}_{F}^{\times}$ is the subgroup of diagonal matrices contained in $\mathfrak{a}$.

		
		Since $\mathfrak{a}'$ is $\tau$-stable and $\Theta^{\sigma}=\Theta$, using Lemma \ref{Lemmasigmaendo}, there exists $u'\in U(\mathfrak{a}')$ such that $\theta'\circ\tau=(\theta'^{-1})^{u'}$. Since $\theta'=\theta'\circ\tau\circ\tau=(\theta'^{-1})^{u'}\circ\tau=\theta'^{u'\tau(u')}$, we deduce that $u'\tau(u')$ normalizes $\theta'$, which means that $u'\tau(u')\in \boldsymbol{J}'\cap  U(\mathfrak{a}')=J'$ by using Proposition \ref{PropJandJ1}.(4). To prove Proposition \ref{propmaxwild}, we only need to find $x'\in G$ such that $\mathfrak{a}'':=\mathfrak{a'}^{x'}$ and $\theta'':=\theta'^{x'}$ have the desired property. By direct calculation, it means that $\tau(x')x'^{-1}$ normalizes $\mathfrak{a}'$ and $u'\tau(x')x'^{-1}$ normalizes $\theta'$, so using Proposition \ref{PropJandJ1}.(4) and the fact that $u'^{-1}\boldsymbol{J}'$ is contained in the normalizer of $\mathfrak{a}'$, it suffices to choose $x'$ such that $u'\tau(x')x'^{-1}\in\boldsymbol{J}'$.
		
		\begin{lemma}\label{LemmayU1}
			
			There exists $y'\in M'$ such that $u'\tau(y')y'^{-1}\in J(\mathfrak{a'},\beta')U^{1}(\mathfrak{a}')=\mathfrak{o}_{F}^{\times}U^{1}(\mathfrak{a}')$.
			
		\end{lemma}
		
		\begin{proof}
			
			First we write $u'=g_{1}^{-1}ug_{1}$ for a certain $u\in U(\mathfrak{a})$. Then $u'\tau(u')\in J(\mathfrak{a'},\beta')$ implies that $uA^{-1}(u^{-1})^{*}A\in J(\mathfrak{a},\beta)\subset \mathfrak{o}_{F}^{\times}U^{1}(\mathfrak{a})$ by direct calculation, where $A$ is defined as in Lemma
			\ref{lemmataua}.
			
			We choose $y'=g_{1}^{-1}yg_{1}$ with $y=\mathrm{diag}(y_{1},...,y_{n})\in M=\mathfrak{o}_{F}^{\times}\times...\times\mathfrak{o}_{F}^{\times}$ to be determined. By direct calculation, $u'\tau(y')y'^{-1}\in J(\mathfrak{a'},\beta')U^{1}(\mathfrak{a}')$ if and only if $uA^{-1}(y^{-1})^{*}Ay^{-1}\in J(\mathfrak{a},\beta)U^{1}(\mathfrak{a})=\mathfrak{o}_{F}^{\times}U^{1}(\mathfrak{a})$. We use $\overline{u_{i}}$, $\overline{a}$, $\overline{y_{i}}$ and $\overline{b}$ to denote the image of $u_{i}$, $a$, $y_{i}$, $b$ in $k_{F}\cong\mathfrak{o}_{F}/\mathfrak{p}_{F}$ respectively, where $u_{i},a,b\in\mathfrak{o}_{F}$ will be defined in the following two paragraphs.
			
			We write $A=\begin{pmatrix}0 & 0 & \ldots & 0 & a_{1}\\ 0 & \iddots & \iddots & a_{2} & 0 \\ \vdots & \iddots & \iddots & \iddots & \vdots \\ 0 & a_{n-1} & \iddots & \iddots & 0 \\ a_{n} & 0 & \ldots & 0 & 0 \end{pmatrix}$ and
			$u=\begin{pmatrix}
				u_{1} & *_{\mathfrak{o}_{F}} & \ldots & \ldots & \ *_{\mathfrak{o}_{F}}  \\ *_{\mathfrak{p}_{F}} & u_{2} & \ddots & \ddots & \vdots \\
				\vdots & \ddots & \ddots & \ddots & \vdots \\
				\vdots & \ddots & \ddots & u_{n-1} & *_{\mathfrak{o}_{F}} \\
				*_{\mathfrak{p}_{F}} & \ldots & \ldots & *_{\mathfrak{p}_{F}} & u_{n}\end{pmatrix}$, where $*_{\mathfrak{o}_{F}}$ and $*_{\mathfrak{p}_{F}}$ represent elements in $\mathfrak{o}_{F}$ and $\mathfrak{p}_{F}$ respectively. By direct calculation, we have $$uA^{-1}(u^{-1})^{*}A=
			\begin{pmatrix}
				u_{1}\sigma(u_{n}^{-1}) & *_{\mathfrak{o}_{F}} & \ldots & \ldots & \ *_{\mathfrak{o}_{F}}  \\ *_{\mathfrak{p}_{F}} & u_{2}\sigma(u_{n-1}^{-1}) & \ddots & \ddots & \vdots \\
				\vdots & \ddots & \ddots & \ddots & \vdots \\
				\vdots & \ddots & \ddots & u_{n-1}\sigma(u_{2}^{-1}) & *_{\mathfrak{o}_{F}} \\
				*_{\mathfrak{p}_{F}} & \ldots & \ldots & *_{\mathfrak{p}_{F}} & u_{n}\sigma(u_{1}^{-1})\end{pmatrix}\in\mathfrak{o}_{F}^{\times}U^{1}(\mathfrak{a}),$$ which means that there exists $a\in\mathfrak{o}_{F}^{\times}$ such that
			\begin{equation}\label{equsigmau}
				u_{1}\sigma(u_{n}^{-1}), u_{2}\sigma(u_{n-1}^{-1}),...,u_{n}\sigma(u_{1}^{-1})\in a(1+\mathfrak{p}_{F}).
			\end{equation}
			
			Also by direct calculation, we have
			$$uA^{-1}(y^{-1})^{*}Ay^{-1}=\begin{pmatrix}
				u_{1}y_{1}^{-1}\sigma(y_{n}^{-1}) & *_{\mathfrak{o}_{F}} & \ldots & \ldots & \ *_{\mathfrak{o}_{F}}  \\ *_{\mathfrak{p}_{F}} & u_{2}y_{2}^{-1}\sigma(y_{n-1}^{-1}) & \ddots & \ddots & \vdots \\
				\vdots & \ddots & \ddots & \ddots & \vdots \\
				\vdots & \ddots & \ddots & u_{n-1}y_{n-1}^{-1}\sigma(y_{2}^{-1}) & *_{\mathfrak{o}_{F}} \\
				*_{\mathfrak{p}_{F}} & \ldots & \ldots & *_{\mathfrak{p}_{F}} & u_{n}y_{n}^{-1}\sigma(y_{1}^{-1})\end{pmatrix},$$
			which means that the lemma is true if and only if there exists $b\in\mathfrak{o}_{F}^{\times}$ such that
			\begin{equation}\label{equAsigmay}
				u_{1}y_{1}^{-1}\sigma(y_{n}^{-1}),u_{2}y_{2}^{-1}\sigma(y_{n-1}^{-1}),...,u_{n}y_{n}^{-1}\sigma(y_{1}^{-1})\in b(1+\mathfrak{p}_{F}).
			\end{equation}
			
			If we consider modulo $\mathfrak{p}_{F}$, then the condition (\ref{equsigmau}) becomes
			\begin{equation}\label{eqmodusigmau}
				\overline{u_{1}}\sigma(\overline{u_{n}}^{-1})=\overline{u_{2}}\sigma(\overline{u_{n-1}}^{-1})=...=\overline{u_{n}}\sigma(\overline{u_{1}}^{-1})=\overline{a}.
			\end{equation}
			Moreover, if we consider modulo $U^{1}(\mathfrak{a})$, then $uA^{-1}(y^{-1})^{*}Ay^{-1}\in\mathfrak{o}_{F}^{\times}U^{1}(\mathfrak{a})$ if and only if there exist $y_{i}\in\mathfrak{o}_{F}^{\times}$ such that there exists $b\in\mathfrak{o}_{F}^{\times}$ in the condition (\ref{equAsigmay}) such that
			\begin{equation}\label{eqmoduAsigmay}
				\overline{u_{1}}\overline{y_{1}}^{-1}\sigma(\overline{y_{n}}^{-1})=\overline{u_{2}}\overline{y_{2}}^{-1}\sigma(\overline{y_{n-1}}^{-1})=...=\overline{u_{n}}\overline{y_{n}}^{-1}\sigma(\overline{y_{1}}^{-1})=\overline{b}.
			\end{equation}

			We choose $b=u_{(n+1)/2}$, then $\overline{b}\sigma(\overline{b}^{-1})=\overline{a}$. Furthermore we choose $y_{i}=b^{-1}u_{i}$ when $i=1,2,...,(n-1)/2$ and $y_{i}=1$ when $i=(n+1)/2,...,n$. Combining with the equation (\ref{eqmodusigmau}), the equation (\ref{eqmoduAsigmay}) is satisfied.

			
		\end{proof}
		
		Let us write $z'u'\tau(y')y'^{-1}\in U'^{1}$ for some $y'\in M'$ and $z'\in \mathfrak{o}_{F}^{\times}$ given by Lemma \ref{LemmayU1}. By replacing the simple
		stratum $[\mathfrak{a}',\beta']$ with $[\mathfrak{a}'^{y'},\beta'^{y'}]$, the simple character $\theta'$ with $\theta'^{y'}$ and $u'$ with $y'^{-1}z'u'\tau(y')$, which doesn't affect the fact that the order is $\tau$-stable, we can and will assume that $u'\in U'^{1}$. We write $J'^{i}=J'\cap U'^{i}$ for $i\geq 1$. We state the following two lemmas which correspond to Lemma 4.16 and Lemma 4.17 in \cite{anandavardhanan2019galois}. Actually the same proofs work when one replaces the Galois involution $\sigma$ in the original lemmas with any involution $\tau$ on $G$.
		
		\begin{lemma}\label{Lemmavtauvi+1}
			
			Let $v'\in U'^{i}$ for some $i\geq 1$ and assume that $v'\tau(v')\in J'^{i}$. Then there exist $j'\in J'^{i}$ and $x'\in U'^{i}$ such that $j'v'\tau(x')x'^{-1}\in U'^{i+1}$.
			
		\end{lemma}
		
		Using Lemma \ref{Lemmavtauvi+1} to replace Lemma 4.16 in \cite{anandavardhanan2019galois}, we may prove the following lemma:
		
		\begin{lemma}\label{Lemmaxjvh}
			
			There exists a sequence of $(x_{i}',j_{i}',v_{i}')\in U'^{i}\times J'^{i}\times U'^{i+1}$ for $i\geq 0$, satisfying the following conditions:
			
			(1) $(x_{0}',j_{0}',v_{0}')=(1,1,u')$;
			
			(2) for all $i\geq 0$, if we set $y'_{i}=x_{0}'x_{1}'...x_{i}'\in U'^{1}$, then the simple character $\theta_{i}'=\theta'^{y_{i}'}\in\mathcal{C}(\mathfrak{a}',\beta'^{y_{i}'})$
			satisfies $\theta_{i}'\circ\tau=(\theta_{i}'^{-1})^{v_{i}'}$;
			
			(3) for all $i\geq 1$, we have $y_{i}'v_{i}'=j_{i}'y_{i-1}'v_{i-1}'\tau(x_{i}')$.
			
		\end{lemma}
		
		Let $x'\in U'^{1}$ be the limit of $y_{i}'=x_{0}'x_{1}'...x_{i}'$ and let $h'\in J'^{1}$ be that of $j_{i}'...j_{1}'j_{0}'$ when $i$ tends to infinity. By Lemma \ref{Lemmaxjvh}.(3),
		we have
		$$y_{i}'v_{i}'\tau(y_{i}'^{-1})=j_{i}'y_{i-1}'v_{i-1}'\tau(y_{i-1}'^{-1})=...=j_{i}'...j_{1}'j_{0}'u'.$$
		Passing to the limit, we get $x'\tau(x')^{-1}=h'u'$, which implies that $u'\tau(x')x'^{-1}=h'^{-1}\in J'$. Let $(\mathfrak{a}'',\theta'')=(\mathfrak{a}'^{x'},\theta'^{x'})$, which finishes the proof of Proposition \ref{propmaxwild}.
		
		\subsection{The maximal case}
		
		In this subsection, we generalize Proposition \ref{propmaxwild} to the following situation:
		
		\begin{proposition}\label{Propmaximalonly}
			
			Let $[\mathfrak{a},\beta]$ be a simple stratum in $\mathrm{M}_{n}(F)$ such that $[E:F]=n$, let $\theta\in\mathcal{C}(\mathfrak{a},\beta)$ such that $\theta\in\Theta$ with $\Theta$ a $\sigma$-invariant endo-class and let $\tau$ be a given unitary involution. Then there exist a simple stratum $[\mathfrak{a}'',\beta'']$ and a simple character $\theta''\in\mathcal{C}(\mathfrak{a}'',\beta'')$ such that $(\mathfrak{a}'',\theta'')$  is $G$-conjugate to $(\mathfrak{a},\theta)$ with the property $\tau(\mathfrak{a}'')=\mathfrak{a}''$ and $\theta''\circ\tau=\theta''^{-1}$.
			
		\end{proposition}
		
		To prove the proposition, we first study an endo-class $\Theta$ over $F$ being $\sigma$-invariant, that is, $\Theta^{\sigma}=\Theta$. Let $T$ be a tame parameter field of $\Theta$. 
		
		\begin{lemma}\label{LemmaTFlift}
			
			Let $\Theta$ be a $\sigma$-invariant endo-class and let $T/F$ be its tame parameter field. Then given a $T/F$-lift $\Psi$ of $\Theta$, there is a unique involution $\alpha$ of $T$ extending $\sigma$ such that $\Psi^{\alpha}=\Psi$.
			
		\end{lemma}
		
		\begin{proof}
			
			The proof of Lemma 4.8 in \cite{anandavardhanan2019galois} can be used almost unchanged for our lemma. We only need to consider $\Theta$ instead of $\Theta^{\vee}$ and $\Psi$ instead of $\Psi^{\vee}$.
			
		\end{proof}
		
		Let $\alpha$ be the involution of $T$ given by Lemma \ref{LemmaTFlift}, and let $T_{0}$ be the subfield of $T$ fixed by $\alpha$. Then $T_{0}\cap F=F_{0}$. We write $t=[T:F]=[T_{0}:F_{0}]$. We need the following proposition due to Hakim and Murnaghan:
		
		\begin{proposition}[\cite{hakim2002tame}, Proposition 2.1]\label{Proptameembed}
			
			There exists an embedding $\iota: T\hookrightarrow \mathrm{M}_{t}(F)$ of $F$-algebras such that for $x\in T$, we have $\iota(\alpha(x))=\iota(x)^{*}:=\sigma(\,^{t}\iota(x))$. 
			
		\end{proposition}
		
		\begin{proof}[Proof of Proposition \ref{Propmaximalonly}]
			
			Let $E=F[\beta]$ and let $T$ be the maximal tamely ramified extension of $F$ in $E$. It is a tame parameter field of the endo-class $\Theta$. The simple character $\theta$ gives $\Psi$, the endo-class of the interior $T/F$-lift of $\Theta$, as we introduced in \S \ref{subsectionendo}. Let $\alpha$ be defined as in Lemma \ref{LemmaTFlift} and let $\iota$ be defined as in Proposition \ref{Proptameembed}. By abuse of notation, we define $$\iota:\mathrm{M}_{n/t}(T)\hookrightarrow\mathrm{M}_{n/t}(\mathrm{M}_{t}(F))=\mathrm{M}_{n}(F)$$ with each block defined by the original $\iota$. First we consider $\tau(x)=\varepsilon\sigma(\,^{t}x^{-1})\varepsilon^{-1}$ for any $x\in G$ with $\varepsilon=I_{n}$ or $\mathrm{diag}(\iota(\epsilon),...,\iota(\epsilon),\iota(\epsilon))$, where $\epsilon\in T_{0}^{\times}-\mathrm{N}_{T/T_{0}}(T^{\times})$. The determinant of the latter matrix is $\mathrm{N}_{T_{0}/F_{0}}(\epsilon)^{n/t}$. Since $$\mathrm{N}_{T_{0}/F_{0}}:T_{0}^{\times}\rightarrow F_{0}^{\times}$$ is a homomorphism which maps $\mathrm{N}_{T/T_{0}}(T^{\times})$ to $\mathrm{N}_{F/F_{0}}(F^{\times})$, it leads to a group homomorphism  $$\mathrm{N}_{T_{0}/F_{0}}:T_{0}^{\times}/\mathrm{N}_{T/T_{0}}(T^{\times})\rightarrow F_{0}^{\times}/\mathrm{N}_{F/F_{0}}(F^{\times})$$ between two groups of order 2. We state and prove the following lemma in general:
			
			\begin{lemma}\label{Lemmafieldext}
				
				Let $F, F_{0}$ be defined as before. Let $L_{0}/F_{0}$ be a finite extension such that $L=L_{0}F$ is a field with $[L:L_{0}]=2$ and $F_{0}=L_{0}\cap F$. Then the group homomorphism $$\mathrm{N}_{L_{0}/F_{0}}:L_{0}^{\times}\rightarrow F_{0}^{\times}$$ induces an isomorphism  $$\mathrm{N}_{L_{0}/F_{0}}:L_{0}^{\times}/\mathrm{N}_{L/L_{0}}(L^{\times})\rightarrow F_{0}^{\times}/\mathrm{N}_{F/F_{0}}(F^{\times})$$
				of groups of order 2.
				
			\end{lemma}
			
			\begin{proof}
				
				We first consider the case where $L_{0}/F_{0}$ is abelian. If on the contrary the induced homomorphism is not an isomorphism, then  $\mathrm{N}_{L_{0}/F_{0}}(L_{0}^{\times})\subset\mathrm{N}_{F/F_{0}}(F^{\times})$ which means that $F$ is contained in $L_{0}$ by the local class field theory (\cite{serre2013local}, Chapter 14, Theorem 1), which is absurd.
				
				When $L_{0}/F_{0}$ is Galois, we may write $F_{0}=L^{0}_{0}\subsetneq...\subsetneq L^{r}_{0}=L_{0}$, such that $L^{i+1}_{0}/L^{i}_{0}$ is abelian for $i=0,...,r-1$ (\cite{serre2013local}, Chapter 4, Proposition 7). We write $L^{i}=L^{i}_{0}F$. Thus it is easy to show that $L^{i}/L^{i}_{0}$ is quadratic, $L^{i}_{0}=L^{i+1}_{0}\cap L^{i}$ and $L^{i+1}_{0}L^{i}=L^{i+1}$ for $i=0,...,r-1$. Using the abelian case,  $$\mathrm{N}_{L_{0}^{i+1}/L_{0}^{i}}:L_{0}^{i+1\times}/\mathrm{N}_{L^{i+1}/L_{0}^{i+1}}(L^{i+1\times})\rightarrow L_{0}^{i\times}/\mathrm{N}_{L^{i}/L^{i}_{0}}(L^{i\times})$$
				is an isomorphism for $i=0,1,...,r-1$. Composing them together, we finish the proof.
				
				When $L_{0}/F_{0}$ is separable, we write $L_{0}'$ the normal closure of $L_{0}$ over $F_{0}$. Thus $L_{0}'$ contains $L_{0}$ and $L_{0}'/F_{0}$ is a finite Galois extension. We write $L'=L_{0}'F$. Using the Galois case,  $$\mathrm{N}_{L_{0}'/F_{0}}:L_{0}'^{\times}/\mathrm{N}_{L'/L_{0}'}(L'^{\times})\rightarrow F_{0}^{\times}/\mathrm{N}_{F/F_{0}}(F^{\times})$$
				is an isomorphism. Since $\mathrm{N}_{L_{0}'/F_{0}}(L_{0}'^{\times})\subset\mathrm{N}_{L_{0}/F_{0}}(L_{0}^{\times})$,  $$\mathrm{N}_{L_{0}/F_{0}}:L_{0}^{\times}/\mathrm{N}_{L/L_{0}}(L^{\times})\rightarrow F_{0}^{\times}/\mathrm{N}_{F/F_{0}}(F^{\times})$$
				is also an isomorphism.
				
				In the characteristic $p$ case in general, we write $L_{0}^{sep}$ the maximal separable subextension of $F_{0}$ contained in $L_{0}$, thus $L_{0}/L_{0}^{sep}$ is purely inseparable. Thus $\mathrm{N}_{L_{0}/L_{0}^{sep}}(x)=x^{p^{[L_{0}:L_{0}^{sep}]}}$ for any $x\in L_{0}^{\times}$. Since $p\neq 2$ and $L_{0}^{\times}/\mathrm{N}_{L/L_{0}}(L^{\times})$ is of order 2, 
				$$\mathrm{N}_{L_{0}/L_{0}^{sep}}:L_{0}^{\times}/\mathrm{N}_{L/L_{0}}(L^{\times})\rightarrow L_{0}^{sep\times}/\mathrm{N}_{L^{sep}/L_{0}^{sep}}(L^{sep\times})$$
				is an isomorphism, where $L^{sep}:=LL_{0}^{sep}$. So we come back to the separable case which finishes the proof.
				
			\end{proof}
			
			Using Lemma \ref{Lemmafieldext} for $L_{0}=T_{0}$, the homomorphism above is actually an isomorphism. Since $n/t$ is odd and $\epsilon\in T_{0}^{\times}-\mathrm{N}_{T/T_{0}}(T^{\times})$, we have $\mathrm{det}(\varepsilon)=\mathrm{N}_{T_{0}/F_{0}}(\epsilon)^{n/t}\in F_{0}^{\times}-\mathrm{N}_{F/F_{0}}(F^{\times})$. So indeed these two involutions represent both of the $G$-classes of hermitian matrices. Thus using Lemma \ref{lemGconjtau}, we may from now on assume $\tau$ to be the two unitary involutions we mentioned above. Furthermore, $\iota(T)^{\times}$ is normalized by $\tau$ from the exact construction of $\tau$ and Proposition \ref{Proptameembed}, where we regard $T$ as an $F$-subalgebra of $\mathrm{M}_{n/t}(T)$ given by the diagonal embedding.
			
			Since $T$ and $\iota(T)$ are isomorphic as $F$-subalgebras contained in $\mathrm{M}_{n}(F)$, by the Skolem-Noether theorem, there exists $g\in G$ such that $\iota(T)=T^{g}$. Thus, if we write $[\mathfrak{a}',\beta']=[\mathfrak{a}^{g},\beta^{g}]$, $\theta'=\theta^{g}$ and $E'=F[\beta']$, then $\theta'\in\Theta$ such that its tame parameter field equals $\iota(T)$. Since $\tau$ normalizes $\iota(T)^{\times}$, we deduce that $\theta'\circ\tau$ and $\theta'^{-1}$ have the same parameter field $\iota(T)$. If we write $\Psi'$ the endo-class of the interior $\iota(T)/F$-lift corresponding to $\theta'$, and if we choose $\alpha'=\iota|_{T}\circ\alpha\circ\iota|_{\iota(T)}^{-1}$, then we have $\Psi'^{\alpha'}=\Psi'$.
			
			Let $C'=\mathrm{M}_{n/t}(\iota(T))$ denote the centralizer of $\iota(T)$ in $\mathrm{M}_{n}(F)$. For $c\in \mathrm{M}_{n/t}(T)$, we have
			$$\tau(\iota(c))=\varepsilon\sigma(\,^{t}\iota(c)^{-1})\varepsilon^{-1}=\varepsilon(\,^{t_{C'}}\iota(\alpha(c))^{-1})\varepsilon^{-1}=\varepsilon(\alpha'(\,^{t_{C'}}\iota(c))^{-1})\varepsilon^{-1}=\tau'(\iota(c)),$$
			where we denote by $t_{C'}$ the transpose on $C'=\mathrm{M}_{n/t}(\iota(T))$ and $\tau'(c')=\varepsilon(\alpha'(\,^{t_{C'}}c'^{-1}))\varepsilon^{-1}$ for any $c'\in C'^{\times}$ . Thus $\tau'$, the restriction of $\tau$ to $C'^{\times}$, is the unitary involution $\tau_{1}$ on $C'^{\times}=\mathrm{GL}_{n/t}(\iota(T))$ with respect to the Galois involution $\alpha'\in\mathrm{Gal}(\iota(T)/F)$. The intersection $\mathfrak{c}'=\mathfrak{a}'\cap C'$ gives rise to a simple stratum $[\mathfrak{c}',\beta']$. The restriction of $\theta'$ to $H^{1}(\mathfrak{c}',\beta')$, denoted by $\theta_{\iota(T)}'$, is a simple character associated to this simple stratum with endo-class $\Psi'$. Since $E'/\iota(T)$ is totally wildly ramified, using Proposition \ref{propmaxwild} with $G$, $\theta$, $\Theta$, $\sigma$ and $\tau$ replaced by $C'^{\times}$, $\theta_{\iota(T)}'$, $\Psi'$, $\alpha'$ and $\tau'$ respectively, there exists $c'\in C'^{\times}$ such that $\tau'(\mathfrak{c}'^{c'})=\mathfrak{c}'^{c'}$ and $\theta_{\iota(T)}'^{c'}\circ\tau'=(\theta_{\iota(T)}'^{c'})^{-1}$.
			
			By the injectivity of $\mathfrak{a}\mapsto \mathfrak{a}\cap C'$ between sets of hereditary orders as mentioned in \S \ref{subsectionendo}, $\mathfrak{a}'':=\mathfrak{a'}^{c'}$ is $\tau$-stable. Moreover if we write $\theta''=\theta'^{c'}$, then from our construction of $\tau$ and the definition of $\iota(T)/F$-lift, the simple characters $$(\theta''\circ\tau)_{\iota(T)}=\theta''\circ\tau|_{H^{1}(\tau(\mathfrak{c}'),\tau(\beta'))}=\theta''\circ\tau'|_{H^{1}(\tau(\mathfrak{c}'),\tau(\beta'))}=\theta_{\iota(T)}''\circ\tau'$$
			and
			$$(\theta''^{-1})_{\iota(T)}=\theta_{\iota(T)}''^{-1}$$ are equal. By the last paragraph of \S \ref{subsectionendo}, the simple character $\theta''$ satisfies the property $\theta''\circ\tau=\theta''^{-1}$.
			
		\end{proof}
		
		\subsection{The general case}
		
		In this subsection, we finish the proof of Lemma \ref{LemmaendotauE} and Theorem \ref{Thmendotautheta}. First of all, we recall the following result of Stevens: 
		
		\begin{proposition}[\cite{stevens2001intertwining}, Theorem 6.3]\label{Propstevens}
			
			Let $[\mathfrak{a},\beta]$ be a simple stratum in $\mathrm{M}_{n}(F)$ with $\sigma_{t}(\mathfrak{a})=\mathfrak{a}$. Suppose that there exists a simple character $\theta\in\mathcal{C}(\mathfrak{a},\beta)$ such that $H^{1}(\mathfrak{a},\beta)$ is $\sigma_{t}$-stable and $\theta\circ\sigma_{t}=\theta$. Then there exists a simple stratum $[\mathfrak{a},\gamma]$ such that $\theta\in\mathcal{C}(\mathfrak{a},\gamma)$ and $\sigma_{t}(\gamma)=\gamma$.
			
		\end{proposition}
		
		\begin{proof}
			
			The original proof of \cite{stevens2001intertwining}, Theorem 6.3 can be modified as follows. For any $x\in \mathrm{M}_{n}(F)$, we use $-\sigma_{t}(x)$ to replace $\overline{x}$; we use $\sigma_{t}$ to replace $\sigma$; for $[\mathfrak{a},\beta]$ a simple stratum, we say that it is \emph{$\sigma_{t}$-invariant} if $\sigma_{t}(\mathfrak{a})=\mathfrak{a}$, and $\sigma_{t}(\beta)=\beta$ and we use this concept to replace the concept \emph{skew simple stratum} in the original proof. With these replacements, the original proof can be used in our case without difficulty (see also the last paragraph of \emph{ibid.}).
			
		\end{proof}

		We choose $[\mathfrak{a}_{0},\beta_{0}]$ to be a maximal simple stratum in $\mathrm{M}_{d}(F)$ 
		and $\theta_{0}\in\mathcal{C}(\mathfrak{a}_{0},\beta_{0})$ such that $\theta_{0}\in\Theta$. 
		By Proposition \ref{Propmaximalonly}, there are a maximal simple stratum $[\mathfrak{a}_{0}',\beta_{0}']$ and a simple character $\theta_{0}'\in\mathcal{C}(\mathfrak{a}_{0}',\beta_{0}')$ which is $\mathrm{GL}_{d}(F)$-conjugate to $\theta_{0}$, such that:
		
		(1) the order $\mathfrak{a}_{0}'$ is $\tau_{1}$-stable;
		
		(2) the group $H^{1}(\mathfrak{a}_{0}',\beta_{0}')$ is $\tau_{1}$-stable and $\theta_{0}'\circ\tau_{1}=\theta_{0}'^{-1}$;
		
		Furthermore, using Proposition \ref{Propstevens} we may assume that:
		
		(3) $\sigma_{t}(\beta_{0}')=\beta_{0}'$.
		
		We embed $\mathrm{M}_{d}(F)$ diagonally into the $F$-algebra $\mathrm{M}_{n}(F)$. This gives an $F$-algebra homomorphism $\iota':F[\beta_{0}']\hookrightarrow\mathrm{M}_{n}(F)$. Write $\beta'=\iota'(\beta_{0}')=\beta_{0}'\otimes...\otimes\beta_{0}'$ and $E'=F[\beta']$. The centralizer $B'$ of $E'$ in $\mathrm{M}_{n}(F)$ is naturally identified with $\mathrm{M}_{m}(E')$. We regard $\sigma_{t}$ as an involution on $E'$ extending $\sigma$ and we write $E_{0}'=E'^{\sigma_{t}}$. Let $\mathfrak{b}'$ be a maximal standard hereditary order in $B'$ which may be identified with $\mathrm{M}_{m}(\mathfrak{o}_{E'})$, and let $\mathfrak{a}'=\mathrm{M}_{m}(\mathfrak{a}_{0}')$ be the unique hereditary order in $\mathrm{M}_{n}(F)$ normalized by $E'^{\times}$ such that $\mathfrak{a}'\cap B'=\mathfrak{b}'$. Then the simple stratum $[\mathfrak{a}',\beta']$ satisfies the requirement of Lemma \ref{LemmaendotauE}, finishing its proof.
		
		Now we focus on the proof of Theorem \ref{Thmendotautheta}. By Lemma \ref{lemGconjtau}, we may change $\tau$ up to $G$-action on its corresponding hermitian matrix which doesn't change the content of the theorem. So if $\varepsilon$ is in the same $G$-class as $I_{n}$, we may simply choose $\tau=\tau_{1}$, where $\tau_{1}(x)=\sigma(\,^{t}x^{-1})$ for any $x\in G$. If not, we fix an $\epsilon\in E_{0}'^{\times}-\mathrm{N}_{E'/E_{0}'}(E'^{\times})$. Regarding $\epsilon$ as an element in $\mathrm{M}_{d}(F)$, we have $\mathrm{det}(\epsilon)=\mathrm{N}_{E_{0}'/F_{0}}(\epsilon)$.
		Since
		$$\mathrm{N}_{E_{0}'/F_{0}}:E_{0}'^{\times}\rightarrow F_{0}^{\times}$$
		is a homomorphism which maps $\mathrm{N}_{E'/E_{0}'}(E'^{\times})$ to $\mathrm{N}_{F/F_{0}}(F^{\times})$, by Lemma \ref{Lemmafieldext} with $L_{0}=E_{0}'$, it leads to an isomorphism  $$\mathrm{N}_{E_{0}'/F_{0}}:E_{0}'^{\times}/\mathrm{N}_{E'/E_{0}'}(E'^{\times})\rightarrow F_{0}^{\times}/\mathrm{N}_{F/F_{0}}(F^{\times})$$ of the two groups of order 2. Thus $\mathrm{N}_{E_{0}'/F_{0}}(\epsilon)\in F_{0}^{\times}-\mathrm{N}_{F/F_{0}}(F^{\times})$. If $E'/E_{0}'$ is unramified, we write $\varepsilon=\mathrm{diag}(\epsilon,...,\epsilon)$. Then $\mathrm{det}(\varepsilon)=\mathrm{N}_{E_{0}'/F_{0}}(\epsilon)^{m}\in F_{0}^{\times}-\mathrm{N}_{F/F_{0}}(F^{\times})$, since $F_{0}^{\times}/\mathrm{N}_{F/F_{0}}(F^{\times})$ is a group of order 2, and $m$ is odd from the condition of the theorem. If $E'/E_{0}'$ is ramified, we may assume further that $\epsilon\in\mathfrak{o}_{E_{0}'}^{\times}$. We write $\varepsilon=\mathrm{diag}(I_{d},...,I_{d},\epsilon)$ and we have $\mathrm{det}(\varepsilon)=\mathrm{N}_{E_{0}'/F_{0}}(\epsilon)\in F_{0}^{\times}-\mathrm{N}_{F/F_{0}}(F^{\times})$. For both cases, $\tau_{\varepsilon}$ is a unitary involution whose corresponding hermitian matrix is not in the same $G$-class as $I_{n}$. So from now on, we only consider the three unitary involutions above. From our assumption of $\tau$, the restriction of $\tau$ to $\mathrm{GL}_{m}(E')$ is also a unitary involution $\tau'=\tau_{1}$ or $\tau_{\varepsilon}$ with $\varepsilon=\mathrm{diag}(1,...,1,\epsilon)$. In particular, since $\epsilon$ is an element in $E'$, we know that $\varepsilon$ commutes with elements in $E'$ and we have $\tau(\beta')=\beta'^{-1}$.
		
		Since $\mathfrak{a}_{0}'$ is $\tau_{1}$-stable and $\mathfrak{b}'$ is $\tau'$-stable, from our assumption of $\tau$ we deduce that $\mathfrak{a}'$ is $\tau$-stable, or by definition $\varepsilon\sigma_{t}(\mathfrak{a}')\varepsilon^{-1}=\mathfrak{a}'$. Since $\sigma_{t}(\beta')=\beta'$, by direct calculation we have
		$$\tau(H^{1}(\mathfrak{a}',\beta'))=\varepsilon H^{1}(\sigma_{t}(\mathfrak{a}'),\sigma_{t}(\beta'))^{-1}\varepsilon^{-1}=H^{1}(\sigma_{t}(\mathfrak{a}')^{\varepsilon^{-1}},\beta'^{\varepsilon^{-1}})=H^{1}(\mathfrak{a}',\beta'^{\varepsilon^{-1}})=H^{1}(\mathfrak{a}',\beta').$$
		Let $M$ be the standard Levi subgroup of $G$ isomorphic to $\mathrm{GL}_{d}(F)\times...\times\mathrm{GL}_{d}(F)$, let $P$ be the standard parabolic subgroup of $G$ generated by $M$ and upper triangular matrices, and let $N$ be its unipotent radical. Let $N^{-}$ be the unipotent radical of the parabolic subgroup opposite to $P$ with respect to $M$. By \cite{secherre2008representations}, Th\'eor\`eme 2.17, we have
		\begin{align}\label{eqHNMN}
			H^{1}(\mathfrak{a}',\beta')&=(H^{1}(\mathfrak{a}',\beta')\cap N^{-})\cdot(H^{1}(\mathfrak{a}',\beta')\cap M)\cdot(H^{1}(\mathfrak{a}',\beta')\cap N),\\
			H^{1}(\mathfrak{a}',\beta')\cap M&=H^{1}(\mathfrak{a}_{0}',\beta'_{0})\times...\times H^{1}(\mathfrak{a}_{0}',\beta'_{0}).
		\end{align}
		Let $\theta'\in\mathcal{C}(\mathfrak{a}',\beta')$ be the transfer of $\theta_{0}'$. By \emph{loc. cit.}, the character $\theta'$ is trivial on $H^{1}(\mathfrak{a}',\beta')\cap N^{-}$ and $H^{1}(\mathfrak{a}',\beta')\cap N$, and the restriction of $\theta'$ to $H^{1}(\mathfrak{a}',\beta')\cap M$ equals $\theta_{0}'\otimes...\otimes\theta_{0}'$. We have $$\theta'\circ\tau|_{H^{1}(\mathfrak{a}',\beta')\cap N^{-}}=\theta'\circ\tau|_{H^{1}(\mathfrak{a}',\beta')\cap N}=\theta'^{-1}|_{H^{1}(\mathfrak{a}',\beta')\cap N^{-}}=\theta'^{-1}|_{H^{1}(\mathfrak{a}',\beta')\cap N}=1$$
		and
		$$\theta'\circ\tau|_{H^{1}(\mathfrak{a}',\beta')\cap M} =\theta_{0}'\circ\tau_{1}\otimes...\otimes\theta_{0}'\circ\tau_{1}=\theta_{0}'^{-1}\otimes...\otimes\theta_{0}'^{-1}=\theta'^{-1}|_{H^{1}(\mathfrak{a}',\beta')\cap M}$$
		for $\tau=\tau_{1}$ or $\tau_{\varepsilon}$ with $\varepsilon=\mathrm{diag}(\epsilon,...,\epsilon)$ or $\mathrm{diag}(1,...,1,\epsilon)$,
		since $\epsilon\in F[\beta_{0}']^{\times}$ normalizes $\theta_{0}'$.
		Thus by equation (\ref{eqHNMN}), we have $\theta'\circ\tau=\theta'^{-1}$.
		
		\begin{remark}\label{remindchoose}
			
			From the proof of Theorem \ref{Thmendotautheta}, we observe that if $\tau$ is chosen as one of the three unitary involutions mentioned in the proof, then we may choose the same simple stratum and simple character satisfying the conclusion of the theorem.
			
		\end{remark}
		
		\begin{remark}
			
			We give a counter-example to show that the condition in Theorem \ref{Thmendotautheta} is necessary. Let $n=2$, let $F/F_{0}$ be unramified, let  $\Theta$ be trivial and let $\varepsilon=\mathrm{diag}(1,\varpi_{F_{0}})$. Then $d=1$, $m=n=2$, $E=F$ and $E_{0}=F_{0}$.  If the theorem is true, then $\mathfrak{a}=\mathrm{M}_{2}(\mathfrak{o}_{F})^{g}$ for some $g\in\mathrm{GL}_{2}(F)$ and $\tau(\mathfrak{a})=\mathfrak{a}$. By direct calculation $\sigma(\,^{t}g^{-1})\varepsilon^{-1} g^{-1}$ normalizes $\mathrm{M}_{2}(\mathfrak{o}_{F})$, which means that $\sigma(\,^{t}g^{-1})\varepsilon^{-1}g^{-1}\in F^{\times}\mathrm{GL}_{2}(\mathfrak{o}_{F})$. It is impossible since $\mathrm{det}(\sigma(\,^{t}g^{-1})\varepsilon^{-1}g^{-1})\in \varpi_{F_{0}}\mathrm{N}_{F/F_{0}}(F^{\times})$, while $\mathrm{det}(F^{\times}\mathrm{GL}_{2}(\mathfrak{o}_{F}))\subset\mathrm{N}_{F/F_{0}}(F^{\times})$.
			
		\end{remark}
		
		\section{The distinguished type theorem}
		
		Let $\pi$ be a cuspidal representation of $G$ such that $\pi^{\sigma}\cong\pi$. From the statements and proofs of Theorem \ref{Thmtautheta}, \ref{Thmselfdualtype} and \ref{Thmendotautheta}, we may assume the following conditions:
		
		\begin{remark}\label{rempicond}
			
			(1) For $\tau=\tau_{1}$, there exist a simple stratum $[\mathfrak{a},\beta]$ and a simple character $\theta\in\mathcal{C}(\mathfrak{a},\beta)$ contained in $\pi$ such that $\tau(\mathfrak{a})=\mathfrak{a}$, $\tau(H^{1}(\mathfrak{a},\beta))=H^{1}(\mathfrak{a},\beta)$, $\theta\circ \tau=\theta^{-1}$ and $\tau(\beta)=\beta^{-1}$, where $\tau_{1}(x):=\sigma(\,^{t}x^{-1})$ for any $x\in G$;
			
			(2) For $\tau=\tau_{1}$, there exists a simple type $(\boldsymbol{J},\Lambda)$ containing $\theta$ and contained in $\pi$ such that $\tau(\boldsymbol{J})=\boldsymbol{J}$ and $\Lambda^\tau\cong\Lambda^{\vee}$;
			
			(3) $\sigma_{t}$ is an involution on $E=F[\beta]$, whose restriction to $F$ equals $\sigma$. So by abuse of notation, we identify $\sigma$ with $\sigma_{t}$. Let $E_{0}=E^{\sigma}$.\textbf{ We assume further in this section that if $E/E_{0}$ is unramified, then $m$ is odd\footnote{ Although this condition seems a little bit annoying, finally in section \ref{sectionunram} we find out that this condition is automatically satisfied for $\pi$ a $\sigma$-invariant supercuspidal representation.}.};
			
			(4) Write $\tau(x)=\varepsilon\sigma(\,^{t}x^{-1})\varepsilon^{-1}$ for any $x\in G$ such that: when $E/E_{0}$ is unramified, we assume $\varepsilon=I_{n}$ or $\mathrm{diag}(\varpi_{E},...,\varpi_{E})\in\mathrm{GL}_{m}(E)\hookrightarrow G$; when $E/E_{0}$ is ramified, we assume $\varepsilon=I_{n}$ or $\mathrm{diag}(1,...,1,\epsilon)\in\mathrm{GL}_{m}(E)\hookrightarrow G$ with $\epsilon\in\mathfrak{o}_{E_{0}}^{\times}-\mathrm{N}_{E/E_{0}}(\mathfrak{o}_{E}^{\times})$. By Remark \ref{remindchoose}, we assume further that for these three unitary involutions, condition (1) and (2) are also satisfied. \textbf{From now on until the end of this section, we assume $\varepsilon$ to be one of these three hermitian matrices and $\tau$ to be one of these three corresponding involutions.}
			
			(5) the element $\beta$ has the block diagonal form:
			$$\beta=\mathrm{diag}(\beta_{0},...,\beta_{0})\in\mathrm{M}_{m}(\mathrm{M}_{d}(F))=\mathrm{M}_{n}(F)$$
			for some $\beta_{0}\in\mathrm{M}_{d}(F)$, where $d$ is the degree of $\beta$ over $F$ and $n=md$. The centralizer $B$ of $E$ in $\mathrm{M}_{n}(F)$ is identified with $\mathrm{M}_{m}(E)$. If we regard $\tau$ as the restriction of the original involution to $B^{\times}$, then it is a unitary involution with respect to $B^{\times}=\mathrm{GL}_{m}(E)$, $E/E_{0}$ and $\sigma\in\mathrm{Gal}(E/E_{0})$;
			
			(6) the order $\mathfrak{b}=\mathfrak{a}\cap B$ is the standard maximal order $\mathrm{M}_{m}(\mathfrak{o}_{E})$ of $\mathrm{M}_{m}(E)$. Thus if we write $\mathfrak{a}_{0}$ as the hereditary order of $\mathrm{M}_{d}(F)$ normalized by $E$, then $\mathfrak{a}$ is identified with $\mathrm{M}_{m}(\mathfrak{a}_{0})$;
			
			(7) $\varpi_{E}$ is a uniformizer of $E$ such that:
			$$\sigma(\varpi_{E})=
			\begin{cases}
				\varpi_{E} & \text{if}\ E\ \text{is unramified over}\ E_{0}; \\
				-\varpi_{E} & \text{if}\ E\ \text{is ramified over}\ E_{0}.
			\end{cases}$$
			
		\end{remark}
		
		Now we state the main theorem of this section:
		
		\begin{theorem}[distinguished type theorem]\label{thmtype}
			
			For $\pi$ a $\sigma$-invariant cuspidal representation, it is $G^{\tau}$-distinguished if and only if it contains a $\tau$-selfdual simple type $(\boldsymbol{J},\Lambda)$ such that $\mathrm{Hom}_{\boldsymbol{J}\cap G^{\tau}}(\Lambda,1)\neq 0$.
			
		\end{theorem}
		
		\begin{remark}
			
			Since every hermitian matrix is equivalent to one of the hermitian matrices mentioned in Remark \ref{rempicond}.(4) up to $G$-action, and the property of distinction is invariant up to equivalence of unitary group, the theorem works for every unitary involution, although we only consider those occurring in \emph{loc. cit.}
			
		\end{remark}
		
		Choose $(\boldsymbol{J},\Lambda)$ as in Remark \ref{rempicond}, using the Mackey formula and Frobenius reciprocity, we have
		
		$$\mathrm{Hom}_{G^{\tau}}(\pi,1)\cong\prod_{g}\mathrm{Hom}_{\boldsymbol{J}^{g}\cap G^{\tau}}(\Lambda^{g},1),$$
		where $g$ ranges over a set of representatives of $(\boldsymbol{J},G^{\tau})$-double cosets in $G$. So $\pi$ is $G^{\tau}$-distinguished if and only if there exists $g$ as a representative of a $(\boldsymbol{J},G^{\tau})$-double coset such that $\mathrm{Hom}_{\boldsymbol{J}^{g}\cap G^{\tau}}(\Lambda^{g},1)\neq 0$. We will study such $g$ and will show that $(\boldsymbol{J}^{g},\Lambda^{g})$ is actually $\tau$-selfdual. So  $(\boldsymbol{J}^{g},\Lambda^{g})$ is a distinguished and $\tau$-selfdual simple type we are looking for, finishing the proof of the theorem.
		
		\subsection{Double cosets contributing to the distinction of $\theta$}
		
		\begin{proposition}\label{propint}
			
			For $g\in G$, the character $\theta^{g}$ is trivial on $H^{1g}\cap G^{\tau}$ if and only if $\tau(g)g^{-1}\in JB^{\times}J$.
			
		\end{proposition}
		
		\begin{proof}
			
			We only need to use the same proof of \cite{secherre2019supercuspidal}, Proposition 6.6, with $\sigma$ replaced by $\tau$.
			
		\end{proof}
		
		As a result, since $\mathrm{Hom}_{\boldsymbol{J}^{g}\cap G^{\tau}}(\Lambda^{g},1)\neq 0$ implies that $\mathrm{Hom}_{H^{1g}\cap G^{\tau}}(\theta^{g},1)\neq 0$, using Proposition \ref{propint} we have $\gamma:=\tau(g)g^{-1}\in JB^{\times}J$.
		
		\subsection{The double coset lemma}
		
		The next step is to prove the following double coset lemma:
		
		\begin{lemma}\label{lemdoucos}
			
			Let $g\in G$. Then $\gamma=\tau(g)g^{-1}\in JB^{\times}J$ if and only if $g\in JB^{\times}G^{\tau}$.
			
		\end{lemma}
		
		\begin{proof}
			
			If $g\in JB^{\times}G^{\tau}$, one verifies immediately that $\gamma\in JB^{\times}J$. Conversely, suppose that $\gamma\in JB^{\times}J$, first we need the following lemma:
			
			\begin{lemma}\label{lembfix}
				
				There exists an element $b\in B^{\times}$ such that $\gamma\in JbJ$ and $b\tau(b)=1$.
				
			\end{lemma}
			
			\begin{proof}
				
				Since $B^{\times}\cap J=\mathfrak{b}^{\times}$ is a maximal compact subgroup of $B^{\times}$, using the Cartan decomposition over $B^{\times}\cong \mathrm{GL}_{m}(E)$, we write $\gamma=xcy$ with $x,y\in J$ and $c=\mathrm{diag}(\varpi_{E}^{a_{1}}I_{m_{1}},...,\varpi_{E}^{a_{r}}I_{m_{r}})$, where $a_{1}>...>a_{r}$ are integers and $m_{1}+...+m_{r}=m$.
				
				If $E/E_{0}$ is unramified, then by definition $c^{*}=c$. So if we choose $b=c\varepsilon^{-1}$, then $b\varepsilon (b^{*})^{-1}\varepsilon^{-1}=c(c^{*})^{-1}=1$, that is, $b\tau(b)=1$.
				
				If $E/E_{0}$ is ramified, since $\tau(\gamma)\gamma=1$, we know that $xcy=\varepsilon y^{*}c^{*}x^{*}\varepsilon^{-1}$ which is equivalent to $(y^{*})^{-1}\varepsilon^{-1}xc=c^{*}x^{*}\varepsilon^{-1}y^{-1}$. Let $z=x^{*}\varepsilon^{-1}y^{-1}\in J$, then we have $z^{*}c=c^{*}z$. We regard $z$ and $c$ as matrices in $\mathrm{M}_{m}(\mathrm{M}_{d}(F))$. Denote by $z^{(j)}\in \mathrm{M}_{m_{j}}(\mathrm{M}_{d}(F))$ the block matrix in $z$ which is at the same place as $\varpi_{E}^{a_{j}}I_{m_{j}}$ in $c$. Since $z^{*}c=c^{*}z$, by direct calculation 
				\begin{equation}\label{eqzi}
					(z^{(j)})^{*}\varpi_{E}^{a_{j}}=(-1)^{a_{j}}\varpi_{E}^{a_{j}}z^{(j)}\quad \text{for}\ j=1,...,r.
				\end{equation}
				
				By considering the following embedding
				\begin{align*}
					\mathrm{M}_{m_{j}}(\mathrm{M}_{d}(F))\hookrightarrow&\mathrm{M}_{m}(\mathrm{M}_{d}(F))\\
					h\mapsto&\mathrm{diag}(0_{m_{1}d},...,0_{m_{j-1}d},h,0_{m_{j+1}d},...,0_{m_{r}d}),
				\end{align*}
				we regard $\mathrm{M}_{m_{j}d}(F)$ as a subalgebra of $\mathrm{M}_{md}(F)$ denoted by $A^{(j)}$, where $0_{m_{j}d}$ represents the zero matrix of size $m_{j}d\times m_{j}d$. We write $\mathfrak{a}^{(j)}=\mathfrak{a}\cap A^{(j)}$. By abuse of notation, we identify the element $\beta_{0}\otimes...\otimes{\beta}_{0}$, which consists of $m_{j}$ copies of $\beta_{0}$ and is contained in $\mathrm{M}_{m_{j}}(\mathrm{M}_{d}(F))$, with $\beta$.
				By \cite{secherre2008representations}, Th\'eor\`eme 2.17, since $z\in J(\mathfrak{a},\beta)$, we get $z^{(j)}\in J(\mathfrak{a}^{(j)},\beta)$ for $j=1,...,r$. By \emph{loc. cit.}, if we denote by $$M=\mathrm{GL}_{m_{1}d}(F)\times...\times\mathrm{GL}_{m_{r}d}(F)$$ the Levi subgroup of $G$ corresponding to the partition $n=m_{1}d+...+m_{r}d$, then
				$$ M\cap J=J(\mathfrak{a}^{(1)},\beta)\times...\times J(\mathfrak{a}^{(r)},\beta)$$
				and
				$$M\cap J^{1}= J^{1}(\mathfrak{a}^{(1)},\beta)\times...\times J^{1}(\mathfrak{a}^{(r)},\beta).$$
				Thus we get $\mathrm{diag}(z^{(1)},...,z^{(r)})\in M\cap J$. Furthermore we have
				$$M\cap J/M\cap J^{1}\cong J(\mathfrak{a}^{(1)},\beta)/J^{1}(\mathfrak{a}^{(1)},\beta)\times...\times J(\mathfrak{a}^{(r)},\beta)/J^{1}(\mathfrak{a}^{(r)},\beta)\cong\mathrm{GL}_{m_{1}}(\boldsymbol{l})\times...\times\mathrm{GL}_{m_{r}}(\boldsymbol{l}).$$
				Since $(\cdot)^{*}$ fixes $M\cap J$ and $M\cap J^{1}$, it induces a map
				\begin{align*}
					M\cap J/M\cap J^{1}\cong \mathrm{GL}_{m_{1}}(\boldsymbol{l})\times...\times\mathrm{GL}_{m_{r}}(\boldsymbol{l})&\longrightarrow \mathrm{GL}_{m_{1}}(\boldsymbol{l})\times...\times\mathrm{GL}_{m_{r}}(\boldsymbol{l})\cong M\cap J/M\cap J^{1},\\
					(\overline{z^{(1)}},...,\overline{z^{(r)}})&\longmapsto(\overline{(z^{(1)})^{*}},...,\overline{(z^{(r)})^{*}}),
				\end{align*}
				where $\boldsymbol{l}$ is the residue field of $E$ and $E_{0}$, and $\overline{z^{(j)}}\in J(\mathfrak{a}^{(j)},\beta)/J^{1}(\mathfrak{a}^{(j)},\beta)\cong\mathrm{GL}_{m_{j}}(\boldsymbol{l})$ is the image of $z^{(j)}$. 
				
				We show that for any $i$ such that $2\nmid a_{i}$, we have $2\mid m_{i}$. Consider $j=i$ in equation (\ref{eqzi}), we get $(z^{(i)})^{*}=-\varpi_{E}^{a_{i}}z^{(i)}\varpi_{E}^{-a_{i}}$. Since $J/J^{1}\cong U(\mathfrak{b})/U^{1}(\mathfrak{b})$ on which $E^{\times}$ acts trivially by conjugation, we get  $\overline{z^{(i)}}=\overline{\varpi_{E}^{a_{i}}z^{(i)}\varpi_{E}^{-a_{i}}}=\overline{-(z^{(i)})^{*}}=-\,^{t}\overline{z^{(i)}}$. Since there exists no anti-symmetric invertible matrix of odd dimension, we must have $2| m_{i}$. Now for $\alpha_{j}=(a_{j},m_{j})$, define $$\varpi_{E}^{\alpha_{j}}=\begin{cases}\varpi_{E}^{a_{j}}I_{m_{j}}\ &\text{if} \ 2|a_{j}; \\ \varpi_{E}^{a_{j}}J_{m_{j}/2}\ &\text{if}\ 2\nmid a_{j}.\end{cases}$$
				and $c'=\mathrm{diag}(\varpi_{E}^{\alpha_{1}},...,\varpi_{E}^{\alpha_{r}})$, where $J_{m_{j}/2}:=\bigg(\begin{matrix}0 & I_{m_{j}/2} \\ -I_{m_{j}/2} & 0\end{matrix}\bigg)$ . We have $c'=c'^{*}$ and $c'$ is in the same $J$-$J$ double coset as $c$. Let $b=c'\varepsilon^{-1}$, we get $b\tau(b)=1$.
				
			\end{proof}
			
			Now we write $\gamma=x'bx$ with $x,x'\in J$ and $b\in B^{\times}$ as in Lemma \ref{lembfix}. Replacing $g$ by $\tau(x')^{-1}g$ does not change the double coset $JgG^{\tau}$ but changes $\gamma$ into $bx\tau(x')$. So from now on, we assume that
			\begin{equation}\label{eqcondgamma}
				\gamma=bx,\quad b\tau(b)=1,\ x\in J, \quad b\ \text{is of the form in the proof of Lemma \ref{lembfix}}.
			\end{equation}
			
			Write $K$ for the group $J\cap b^{-1}Jb$. Since $\tau(b)=b^{-1}$ and $J$ is $\tau$-stable, we have $x\in K$. The following corollary of Lemma \ref{lembfix} is obvious.
			
			\begin{corollary}
				
				The map $\delta_{b}:k\mapsto b^{-1}\tau(k)b$ is an involution on $K$.
				
			\end{corollary}
			
			Now for $a_{1}>...>a_{r}$ as in the proof of Lemma \ref{lembfix}, and $M=\mathrm{GL}_{m_{1}d}(F)\times...\times\mathrm{GL}_{m_{r}d}(F)\subseteq G$, we write $P$ for the standard parabolic subgroup of $G$ generated by $M$ and upper triangular matrices, $N$ for the unipotent radical of $P$ and $N^{-}$ for the opposite of $N$ as a unipotent subgroup. By definition, $b$ normalizes $M$ and we have
			\begin{align*}
				K=(K\cap N ^{-})\cdot(K\cap M)\cdot(K\cap N).
			\end{align*}
			For $V=K\cap B^{\times}=U\cap b^{-1}Ub$ a subgroup of $B^{\times}$, similarly we have
			\begin{align*}
				V=(V\cap N ^{-})\cdot(V\cap M)\cdot(V\cap N),
			\end{align*}
			where $U=U(\mathfrak{b})$ and $U^{1}=J^{1}\cap B^{\times}=U^{1}(\mathfrak{b})$. By definition, $V$ is also fixed by $\delta_{b}$.
			
			\begin{lemma}\label{lemK1}
				
				The subset
				$$K^{1}=(K\cap N ^{-})\cdot(J^{1}\cap M)\cdot(K\cap N)$$
				is a $\delta_{b}$-stable normal pro-$p$-subgroup of $K$, and we have $K=VK^{1}$.
				
			\end{lemma}
			
			\begin{proof}
				
				The proof is the same as that in \cite{secherre2019supercuspidal}, Lemma 6.10.
				
			\end{proof}
			
			\begin{lemma}\label{lemredpgroup}
				
				Let $x\in K$ such that $x\delta_{b}(x)=1$, then there are $k\in K$ and $v\in V$ such that:
				
				(1) the element $v$ is in $\mathrm{GL}_{m_{1}}(\mathfrak{o}_{E})\times...\times\mathrm{GL}_{m_{r}}(\mathfrak{o}_{E})\subseteq B^{\times}$ such that $v\delta_{b}(v)=1$;
				
				(2) one has $\delta_{b}(k)xk^{-1}\in vK^{1}$.
				
			\end{lemma}
			
			\begin{proof}
				
				Let $V^{1}=V\cap K^{1}$. We have
				$$V^{1}=(V\cap N ^{-})\cdot(U^{1}\cap M)\cdot(V\cap N).$$
				Thus we have canonical $\delta_{b}$-equivariant group isomorphisms
				\begin{equation}\label{eqKK1}
					K/K^{1}\cong V/V^{1}\cong(U\cap M)/(U^{1}\cap M).
				\end{equation}
				Since $ B^{\times}\cap M=\mathrm{GL}_{m_{1}}(E)\times...\times\mathrm{GL}_{m_{r}}(E)$, the right hand side of (\ref{eqKK1}) is identified with $\mathcal{M}=\mathrm{GL}_{m_{1}}(\boldsymbol{l})\times...\times\mathrm{GL}_{m_{r}}(\boldsymbol{l})$, where $\boldsymbol{l}$ denotes the residue field of $E$. As in the proof of Lemma \ref{lembfix}, we may write
				$\varepsilon^{-1}b=\mathrm{diag}(\varpi_{E}^{a_{1}}c_{1},...,\varpi_{E}^{a_{r}}c_{r})$ with $c_{j}\in\mathrm{GL}_{m_{j}}(\mathfrak{o}_{E})$. Moreover, the involution $\delta_{b}$ acts on $\mathcal{M}$ by
				$$(g_{1},...,g_{r})\mapsto(\overline{c_{1}}^{-1}\sigma(^tg_{1}^{-1})\overline{c_{1}},...,\overline{c_{r}}^{-1}\sigma(^tg_{r}^{-1})\overline{c_{r}}),$$
				where we denote by $\overline{c_{j}}$ the image of $c_{j}$ in $\mathrm{GL}_{m_{j}}(\boldsymbol{l})$. We denote by $(g_{1},...,g_{r})$ the image of $x$ in $\mathcal{M}=\mathrm{GL}_{m_{1}}(\boldsymbol{l})\times...\times\mathrm{GL}_{m_{r}}(\boldsymbol{l})$.
				
				When $E/E_{0}$ is unramified, we denote by $\boldsymbol{l_{0}}$ the residue field of $E_{0}$. So $\boldsymbol{l}/\boldsymbol{l_{0}}$ is quadratic and the restriction of $\sigma$ to $\boldsymbol{l}$ is the non-trivial involution in $\mathrm{Gal}(\boldsymbol{l}/\boldsymbol{l_{0}})$. Since $(b^{-1}\varepsilon)^{*}=\varepsilon(b^{*})^{-1}\varepsilon^{-1}\varepsilon=\tau(b)\varepsilon=b^{-1}\varepsilon$, we get $\overline{c_{j}}^{*}=\overline{c_{j}}$. If $x\delta_{b}(x)=1$, then  $(\overline{c_{j}}g_{j})^{*}=g_{j}^{*}\overline{c_{j}}=\overline{c_{j}}g_{j}$. 
				
				\begin{lemma}[\cite{kleidman1990subgroup}, Proposition 2.3.1]\label{lemherfin}
					
					For $\overline{x}=\overline{x}^{*}$ in $\mathrm{GL}_{s}(\boldsymbol{l})$, there exists $A\in\mathrm{GL}_{s}(\boldsymbol{l})$ such that $A\overline{x}A^{*}=I_{s}$.
					
				\end{lemma}
				
				Using Lemma \ref{lemherfin}, we may choose $k_{j}\in\mathrm{GL}_{m_{j}}(\mathfrak{o}_{E})$ such that its image $\overline{k_{j}}$ in $\mathrm{GL}_{m_{j}}(\boldsymbol{l})$ satisfies $(\overline{k_{j}}^{*})^{-1}\overline{c_{j}}g_{j}\overline{k_{j}}^{-1}=I_{m_{j}}$. Choose $k=\mathrm{diag}(k_{1},...,k_{r})$ and $v=\mathrm{diag}(v_{1},...,v_{r})=\mathrm{diag}(c_{1}^{-1},...,c_{r}^{-1})$, we get $\delta_{b}(k)xk^{-1}\in vV^{1}$ and $\delta_{b}(v)v=\mathrm{diag}(c_{1}^{-1}c_{1}^{*}c_{1}c_{1}^{-1},...,c_{r}^{-1}c_{r}^{*}c_{r}c_{r}^{-1})=1$.

				When $E/E_{0}$ is ramified, the restriction of $\sigma$ to $\boldsymbol{l}$ is trivial. Since $(b^{-1}\varepsilon)^{*}=b^{-1}\varepsilon$, we get $c_{j}^{*}=(-1)^{a_{j}}c_{j}$ and $\,^{t}\overline{c_{j}}=(-1)^{a_{j}}\overline{c_{j}}$. 
				
				\begin{lemma}[\cite{kleidman1990subgroup}, Proposition 2.5.4]\label{lemothfin}
					
					For $\overline{x}=\,^t \overline{x}$ in $\mathrm{GL}_{s}(\boldsymbol{l})$, there exists $A\in\mathrm{GL}_{s}(\boldsymbol{l})$ such that $ A\overline{x} \,^t A$ is either $I_{s}$ or $\overline{\varepsilon_{s}}=\mathrm{diag}(1,...,1,\overline{\epsilon})$, where $\overline{\epsilon}\in\boldsymbol{l}^{\times}-\boldsymbol{l}^{\times2}$ with $\boldsymbol{l}^{\times2}$ denoting the group of square elements of $\boldsymbol{l}^{\times}$.
					
				\end{lemma}
				
				\begin{lemma}[\cite{kleidman1990subgroup}, Proposition 2.4.1]\label{lemsymfin}
					
					For $\overline{x}=-\,^t \overline{x}$ in $\mathrm{GL}_{s}(\boldsymbol{l})$ and $2\mid s$, there exists $A\in\mathrm{GL}_{s}(\boldsymbol{l})$ such that $A\overline{x}\,^t A=J_{s/2}$.
					
				\end{lemma}
				
				When $a_{j}$ is even, using Lemma \ref{lemothfin} we may choose $k_{j}\in\mathrm{GL}_{m_{j}}(\mathfrak{o}_{E})$ such that its image $\overline{k_{j}}$ in $\mathrm{GL}_{m_{j}}(\boldsymbol{l})$ satisfies that $(\,^{t}\overline{k_{j}})^{-1}\overline{c_{j}}g_{j}\overline{k_{j}}^{-1}$ equals either $I_{m_{j}}$ or $\overline{\varepsilon_{m_{j}}}$, where we choose $\varepsilon_{m_{j}}=\mathrm{diag}(1,...,1,\epsilon)\in\mathrm{GL}_{m_{j}}(\mathfrak{o}_{E})$ such that its image $\overline{\varepsilon_{m_{j}}}$ in $\mathrm{GL}_{m_{j}}(\boldsymbol{l})$ is $\mathrm{diag}(1,...,1,\overline{\epsilon})$ as in Lemma \ref{lemothfin}. Let $v_{j}$ be $c_{j}^{-1}$ or $c_{j}^{-1}\varepsilon_{m_{j}}$ in the two cases respectively.
				
				When $a_{j}$ is odd we deduce that $m_{j}$ is even from the proof of Lemma \ref{lembfix}. Using Lemma \ref{lemsymfin}, we may choose $k_{j}\in\mathrm{GL}_{m_{j}}(\mathfrak{o}_{E})$ such that its image $\overline{k_{j}}$ in $\mathrm{GL}_{m_{j}}(\boldsymbol{l})$ satisfies $(\,^{t}\overline{k_{j}})^{-1}\overline{c_{j}}g_{j}\overline{k_{j}}^{-1}=J_{m_{j/2}}$. We choose $v_{j}=c_{j}^{-1}J_{m_{j}/2}$.
				
				Choose $k=\mathrm{diag}(k_{1},...,k_{r})$ and $v=\mathrm{diag}(v_{1},...,v_{r})$, we know that
				$$\delta_{b}(k)xk^{-1}\in vV^{1}$$
				and $$\delta_{b}(v)v=\mathrm{diag}(c_{1}^{-1}(v_{1}^{*})^{-1}c_{1}v_{1},...,c_{r}^{-1}(v_{r}^{*})^{-1}c_{r}v_{r})=1$$ by direct calculation in the two cases respectively. So no matter $E/E_{0}$ is ramified or not, we finish the proof.
				
			\end{proof}
			
			Now we finish the proof of Lemma \ref{lemdoucos}. Using Lemma \ref{lemredpgroup}, we choose $k\in K$ and $v\in V$ such that $bv\tau(bv)=1$ and $\delta_{b}(k)xk^{-1}\in vK^{1}$. Thus we have $\tau(k)\gamma k^{-1}\in bv K^{1}$. Therefore replacing $g$ by $kg$ and $b$ by $bv$, we may assume
			\begin{equation}\label{eqgamma}
				\gamma=bx,\quad b\tau(b)=1, \quad x\in K^{1},\quad b\in\varpi_{E}^{a_{1}}\mathrm{GL}_{m_{1}}(\mathfrak{o}_{E})\times...\times\varpi_{E}^{a_{r}}\mathrm{GL}_{m_{r}}(\mathfrak{o}_{E}).
			\end{equation}
			Furthermore, we have $\delta_{b}(x)x=1$.
			
			Since $K^{1}$ is a $\delta_{b}$-stable pro-$p$-group and $p$ is odd, the first cohomology set of $\delta_{b}$ on $K^{1}$ is trivial. Thus $x=\delta_{b}(y)y^{-1}$ for some $y\in K^{1}$, hence $\gamma=\tau(y)by^{-1}$. Consider the determinant of this equation, we have $\mathrm{det}(b)\in\mathrm{N}_{F/F_{0}}(F^{\times})$. If we denote by $\mathrm{det}_{B}$ the determinant function defined on $B^{\times}=\mathrm{GL}_{m}(E)$, then $\mathrm{det}(b)=\mathrm{N}_{E/F}(\mathrm{det}_{B}(b))$. Using Lemma \ref{Lemmafieldext} for $L=E$, we get $\mathrm{det}_{B}(b)\in\mathrm{N}_{E/E_{0}}(E^{\times})$ and $\mathrm{det}_{B}(\varepsilon^{-1}b)\in\mathrm{det}_{B}(\varepsilon^{-1})\mathrm{N}_{E/E_{0}}(E^{\times})$. Since $\tau(b)b=1$, we have $(\varepsilon^{-1}b)^{*}=\varepsilon^{-1}b$. Using Proposition \ref{PropGOH}, there exists $h\in B^{\times}$ such that $\varepsilon^{-1}b=(h^{*})^{-1}\varepsilon^{-1} h^{-1}$. So we have $b=\tau(h)h^{-1}$. Thus $g\in yhG^{\tau}\subseteq JB^{\times}G^{\tau}$, which finishes the proof of Lemma \ref{lemdoucos}.
			
		\end{proof}
		
		\subsection{Distinction of the Heisenberg representation}
		
		Now let $\eta$ be the Heisenberg representation of $J^{1}$ associated to $\theta$. We have the following result similar to \cite{secherre2019supercuspidal}, Proposition 6.12. by replacing $\sigma$ with $\tau$:
		
		\begin{proposition}\label{propheis}
			
			Given $g\in G$, we have:
			$$\mathrm{dim}_{R}\mathrm{Hom}_{J^{1g}\cap G^{\tau}}(\eta^{g},1)=\begin{cases}1\quad &\text{if}\ g\in JB^{\times}G^{\tau},\\ 0\quad &\text{otherwise}.\end{cases}$$
			
		\end{proposition}
		
		\begin{proof}
			
			It is useful to recall some details of the proof of this proposition, which will be used in the next subsection. We write $\delta(x):=\gamma^{-1}\tau(x)\gamma$ for any $x\in G$ which is an involution on $G$. And for any subgroup $H\subset G$, we have $H^{g}\cap G^{\tau}=(H\cap G^{\delta})^{g}$.
			
			When $g\notin JB^{\times}G^{\tau}$, restricting $\eta^{g}$ to $H^{1g}$ and using Proposition \ref{propint} and Lemma \ref{lemdoucos}, the dimension equals 0. When $g\in JB^{\times}G^{\tau}$, we need to prove that $\mathrm{Hom}_{J^{1g}\cap G^{\tau}}(\eta^{g},1)=\mathrm{Hom}_{J^{1}\cap G^{\delta}}(\eta,1)$ is of dimension 1. We state the following general proposition which works for a general involution on $G$:
			
			\begin{proposition}\label{Propheisdelta}
				
				Let $\delta$ be an involution on $G$ such that $\delta(H^{1})=H^{1\gamma}$ and $\theta\circ\delta=\theta^{-1\gamma}$, where $\gamma\in B^{\times}$ such that $\delta(\gamma)\gamma=1$. Then we have
				$$\mathrm{dim}_{R}\mathrm{Hom}_{J^{1}\cap G^{\delta}}(\eta,1)=1.$$
				
			\end{proposition}
			
			Since Proposition \ref{Propheisdelta} in our special case implies Proposition \ref{propheis}, we only need to focus on the proof of this proposition. We only need to prove that the space
			$$\mathrm{Hom}_{J^{1}\cap G^{\delta}}(\eta^{(J^{1}:H^{1})^{1/2}},1)\cong\mathrm{Hom}_{J^{1}\cap G^{\delta}}(\mathrm{Ind}_{H^{1}}^{J^{1}}(\theta),1)$$
			is of dimension $(J^{1}:H^{1})^{1/2}$.
			
			\begin{lemma}\label{LemmaHg=Htaug}
				
				For $H$ a subgroup of $G$ such that $\delta(H)=H^{\gamma}$ with $\delta$ and $\gamma$ as in Proposition \ref{Propheisdelta}, we have
				$$H\cap G^{\delta}=H^{\gamma}\cap G^{\delta}= H\cap H^{\gamma}\cap G^{\delta}.$$
				
			\end{lemma}
			
			\begin{proof}
				
				We have $H\cap G^{\delta}=\delta(H\cap G^{\delta})=\delta(H)\cap\delta(G^{\delta})=H^{\gamma}\cap G^{\delta}$ which proves the lemma.
				
			\end{proof}
			
			\begin{lemma}\label{Lemmaindthetares}
				
				Let $\delta$ and $\gamma$ be as in Proposition \ref{Propheisdelta}, then we have the following isomorphisms of finite dimensional representations:
				
				(1) $\mathrm{Ind}_{H^{1}}^{J^{1}}\theta|_{J^{1}\cap J^{1\gamma}}\cong\bigoplus_{H^{1}\backslash J^{1}/J^{1}\cap J^{1\gamma}}\mathrm{Ind}_{H^{1}\cap J^{1\gamma}}^{J^{1}\cap J^{1\gamma}}\theta$;
				
				(2) $\mathrm{Ind}_{H^{1\gamma}}^{J^{1\gamma}}\theta^{\gamma}|_{J^{1}\cap J^{1\gamma}}\cong\bigoplus_{H^{1\gamma}\backslash J^{1\gamma}/J^{1}\cap J^{1\gamma}}\mathrm{Ind}_{J^{1}\cap H^{1\gamma}}^{J^{1}\cap J^{1\gamma}}\theta^{\gamma}$;
				
				(3) $\mathrm{Ind}_{H^{1}}^{J^{1}}\theta|_{J^{1}\cap G^{\delta}}\cong\bigoplus_{H^{1}\backslash J^{1}/J^{1}\cap J^{1\gamma}}\bigoplus_{H^{1}\cap J^{1\gamma}\backslash J^{1}\cap J^{1\gamma}/J^{1}\cap G^{\delta}}\mathrm{Ind}_{H^{1}\cap G^{\delta}}^{J^{1}\cap G^{\delta}}\theta$;
				
				(4) $\mathrm{Ind}_{H^{1\gamma}}^{J^{1\gamma}}\theta^{\gamma}|_{J^{1\gamma}\cap G^{\delta}}\cong\bigoplus_{H^{1\gamma}\backslash J^{1\gamma}/J^{1}\cap J^{1\gamma}}\bigoplus_{J^{1}\cap H^{1\gamma}\backslash J^{1}\cap J^{1\gamma}/J^{1\gamma}\cap G^{\delta}}\mathrm{Ind}_{H^{1\gamma}\cap G^{\delta}}^{J^{1\gamma}\cap G^{\delta}}\theta$.
				
			\end{lemma}
			
			\begin{proof}
				
				We only prove (1) and (3), since the proofs of (2) and (4) are similar to the proofs of (1) and (3) respectively.
				
				For (1), using the Mackey formula, we have
				\begin{align*}
					\mathrm{Ind}_{H^{1}}^{J^{1}}\theta|_{J^{1}\cap J^{1\gamma}}&\cong\bigoplus_{x\in H^{1}\backslash J^{1}/J^{1}\cap J^{1\gamma}}\mathrm{Ind}_{H^{1x}\cap(J^{1}\cap J^{1\gamma})}^{J^{1}\cap J^{1\gamma}}\theta^{x}\\
					&\cong\bigoplus_{H^{1}\backslash J^{1}/J^{1}\cap J^{1\gamma}}\mathrm{Ind}_{H^{1}\cap J^{1\gamma}}^{J^{1}\cap J^{1\gamma}}\theta.
				\end{align*}
				The last step is because $x\in J^{1}$ normalizes $H^{1}$ and $\theta$.
				
				For (3), using the Mackey formula again, we have
				\begin{align*}
					\mathrm{Ind}_{H^{1}}^{J^{1}}\theta|_{J^{1}\cap G^{\delta}}&\cong\bigoplus_{H^{1}\backslash J^{1}/J^{1}\cap J^{1\gamma}}\mathrm{Ind}_{H^{1}\cap J^{1\gamma}}^{J^{1}\cap J^{1\gamma}}\theta|_{J^{1}\cap G^{\delta}}\\
					&\cong\bigoplus_{H^{1}\backslash J^{1}/J^{1}\cap J^{1\gamma}}\bigoplus_{y\in H^{1}\cap J^{1\gamma}\backslash J^{1}\cap J^{1\gamma}/J^{1}\cap G^{\delta}}\mathrm{Ind}_{(H^{1}\cap J^{1\gamma})^{y}\cap (J^{1}\cap G^{\delta})}^{J^{1}\cap G^{\delta}}\theta^{y}\\
					&\cong\bigoplus_{H^{1}\backslash J^{1}/J^{1}\cap J^{1\gamma}}\bigoplus_{H^{1}\cap J^{1\gamma}\backslash J^{1}\cap J^{1\gamma}/J^{1}\cap G^{\delta}}\mathrm{Ind}_{H^{1}\cap G^{\delta}}^{J^{1}\cap G^{\delta}}\theta.
				\end{align*}
				The last step is because $y\in J^{1}\cap J^{1\gamma}$ normalizes $H^{1}\cap J^{1\gamma}$ and $\theta$, and $H^{1}\cap J^{1\gamma}\cap J^{1}\cap G^{\delta}=H^{1}\cap G^{\delta}$ by Lemma \ref{LemmaHg=Htaug}.(2) for $H=J^{1}$. So we finish the proof.
				
			\end{proof}
			
			\begin{lemma}\label{LemmaJHJgammaHgamma}
				
				Let $\delta$ and $\gamma$ be as in Proposition \ref{Propheisdelta}, then we have:
				
				(1) $|H^{1}\backslash J^{1}/J^{1}\cap J^{1\gamma}|\cdot |H^{1}\cap J^{1\gamma}\backslash J^{1}\cap J^{1\gamma}/J^{1}\cap G^{\delta}|=(J^{1}:H^{1})^{1/2};$
				
				(2) $|H^{1\gamma}\backslash J^{1\gamma}/J^{1}\cap J^{1\gamma}|\cdot |J^{1}\cap H^{1\gamma}\backslash J^{1}\cap J^{1\gamma}/J^{1\gamma}\cap G^{\delta}|=(J^{1\gamma}:H^{1\gamma})^{1/2};$
				
				(3) $(J^{1}:H^{1})^{1/2}=(J^{1\gamma}:H^{1\gamma})^{1/2}=(J^{1}\cap G^{\delta}:H^{1}\cap G^{\delta}).$
				
			\end{lemma}
			
			\begin{proof}
				
				For (3), we refer to \cite{secherre2019supercuspidal} \S 6.3 for a proof, by noting that all the results and proofs from Lemma 6.14 to the end of \S 6.3 in \emph{ibid.} can be generalized to a general involution $\delta$ on $G$, with $\tau$ in \emph{loc. cit.} replaced by $\delta$ in our settings. For (1), since $J^{1}$ normalizes $H^{1}$ and $J^{1}\cap J^{1\gamma}$ normalizes $H^{1}\cap J^{1\gamma}$, we have
				\begin{align*}
					\text{left hand side of (1)}=&(J^{1}:H^{1}(J^{1}\cap J^{1\gamma}))\cdot (J^{1}\cap J^{1\gamma}:(H^{1}\cap J^{1\gamma})(J^{1}\cap G^{\delta}))\\
					=&(J^{1}:H^{1})\cdot(J^{1}\cap J^{1\gamma}:H^{1}\cap J^{1\gamma})^{-1}\cdot\\
					&\cdot(J^{1}\cap J^{1\gamma}:H^{1}\cap J^{1\gamma})\cdot(J^{1}\cap G^{\delta}:H^{1}\cap J^{1\gamma}\cap G^{\delta})^{-1}\\
					=&(J^{1}:H^{1})\cdot(J^{1}\cap G^{\delta}:H^{1}\cap G^{\delta})^{-1}\\
					=&(J^{1}:H^{1})^{1/2},
				\end{align*}
				where we use Lemma \ref{LemmaHg=Htaug} for $H=J^{1\gamma}$ and (3) in the last two equations. So we finish the proof of (1), and the proof of (2) is similar.
				
			\end{proof}
			
			
			
			
			
			
			
			
			Combining Lemma \ref{Lemmaindthetares}.(3) with Lemma \ref{LemmaJHJgammaHgamma}.(1)(3), we have
			\begin{align*}\mathrm{dim}_{R}\mathrm{Hom}_{J^{1}\cap G^{\delta}}(\mathrm{Ind}_{H^{1}}^{J^{1}}\theta,1)&=\mathrm{dim}_{R}\bigoplus_{H^{1}\backslash J^{1}/J^{1}\cap J^{1\gamma}}\bigoplus_{H^{1}\cap J^{1\gamma}\backslash J^{1}\cap J^{1\gamma}/J^{1}\cap G^{\delta}}\mathrm{Hom}_{J^{1}\cap G^{\delta}}(\mathrm{Ind}_{H^{1}\cap G^{\delta}}^{J^{1}\cap G^{\delta}}\theta,1)\\
				&=(J^{1}:H^{1})^{1/2}\mathrm{dim}_{R}\mathrm{Hom}_{H^{1}\cap G^{\delta}}(\theta|_{H^{1}\cap G^{\delta}},1)\\
				&=(J^{1}:H^{1})^{1/2}.
			\end{align*}
			For the last step, since $\gamma$ intertwines $\theta^{-1}$ and $\theta\circ\delta=\theta^{-1\gamma}$, we know that $\theta$ is trivial on $$\{y\delta(y)|y\in H^{1}\cap H^{1\gamma}\}.$$
			This set equals $H^{1}\cap G^{\delta}$ since the the first cohomology group of $\delta^{-1}$-action on $H^{1}\cap H^{1\gamma}$ is trivial. Thus $\theta|_{H^{1}\cap G^{\delta}}$ is the trivial character.
			
		\end{proof}
		
		\subsection{Distinction of extensions of the Heisenberg representation}
		
		Let $\boldsymbol{\kappa}$ be an irreducible representation of $\boldsymbol{J}$ extending $\eta$. There is a unique irreducible representation $\boldsymbol{\rho}$ of $\boldsymbol{J}$, which is trivial on $J^{1}$ satisfying $\Lambda\cong\boldsymbol{\kappa}\otimes\boldsymbol{\rho}$. 
		
		\begin{lemma}\label{lemkappa}
			
			Let $g\in JB^{\times}G^{\tau}$.
			
			(1) There is a unique character $\chi$ of $\boldsymbol{J}^{g}\cap G^{\tau}$ trivial on $J^{1g}\cap G^{\tau}$ such that
			$$\mathrm{Hom}_{J^{1g}\cap G^{\tau}}(\eta^{g},1)=\mathrm{Hom}_{\boldsymbol{J}^{g}\cap G^{\tau}}(\boldsymbol{\kappa}^{g},\chi^{-1}).$$
			
			(2) The canonical linear map
			$$\mathrm{Hom}_{J^{1g}\cap G^{\tau}}(\eta^{g},1)\otimes\mathrm{Hom}_{\boldsymbol{J}^{g}\cap G^{\tau}}(\boldsymbol{\rho}^{g},\chi)\rightarrow\mathrm{Hom}_{\boldsymbol{J}^{g}\cap G^{\tau}}(\Lambda^{g},1).$$
			is an isomorphism.
			
		\end{lemma}
		
		\begin{proof}
			
			The proof is the same as that in \cite{secherre2019supercuspidal}, Lemma 6.20.
			
		\end{proof}
		
		For $g\in JB^{\times}G^{\tau}$, we have $\tau(g)\in \tau(JB^{\times}G^{\tau})=JB^{\times}G^{\tau}$, which means that we may consider the similar thing for $\tau(g)$ to that for $g$ in Lemma \ref{lemkappa}. Thus, there exists a unique character $\chi'$ of $\boldsymbol{J}^{\tau(g)}\cap G^{\tau}$ trivial on $J^{1\tau(g)}\cap G^{\tau}$ such that
		$$\mathrm{Hom}_{J^{1\tau(g)}\cap G^{\tau}}(\eta^{\tau(g)},1)\cong\mathrm{Hom}_{\boldsymbol{J}^{\tau(g)}\cap G^{\tau}}(\boldsymbol{\kappa}^{\tau(g)},\chi'^{-1}).$$
		Moreover $\tau(\boldsymbol{J})=\boldsymbol{J}$, $\tau(J)=J$, $\tau(J^{1})=J^{1}$ and $\tau(H^{1})=H^{1}$, thus using Lemma \ref{LemmaunitaryJJ0} and Lemma \ref{LemmaHg=Htaug}, we have $\boldsymbol{J}^{g}\cap G^{\tau}=\boldsymbol{J}^{\tau(g)}\cap G^{\tau}=J^{g}\cap G^{\tau}=J^{\tau(g)}\cap G^{\tau}$, $J^{1g}\cap G^{\tau}=J^{1\tau(g)}\cap G^{\tau}$ and $H^{1g}\cap G^{\tau}=H^{1\tau(g)}\cap G^{\tau}$. As a result, $\chi$ and $\chi'$ are characters defined on the same group $\boldsymbol{J}^{g}\cap G^{\tau}=\boldsymbol{J}^{\tau(g)}\cap G^{\tau}$. A natural idea is to compare them. For the rest of this subsection, we focus on the proof of the following proposition:
		
		\begin{proposition}\label{Propchichi'}
			
			For $\chi$ and $\chi'$ defined above as characters of $\boldsymbol{J}^{g}\cap G^{\tau}=\boldsymbol{J}^{\tau(g)}\cap G^{\tau}$, we have $\chi=\chi'$.
			
		\end{proposition}
		
		We write $\delta(x)=\gamma^{-1}\tau(x)\gamma$ for any $x\in G$ with $\gamma=\tau(g)g^{-1}$. From \S \ref{subsectionsimplestratum}, we have $\gamma\in I_{G}(\eta)=I_{G}(\kappa^{0})$, where $\kappa^{0}=\boldsymbol{\kappa}|_{J}$. Moreover we have
		$$\mathrm{dim}_{R}(\mathrm{Hom}_{J\cap J^{\gamma}}(\kappa^{0\gamma},\kappa^{0}))=\mathrm{dim}_{R}(\mathrm{Hom}_{J^{1}\cap J^{1\gamma}}(\eta^{\gamma},\eta))=1.$$
		Using Lemma \ref{LemmaHg=Htaug}, we have $J^{1}\cap G^{\delta}=J^{1\gamma}\cap G^{\delta}$ as a subgroup of $J^{1}\cap J^{1\gamma}$ and $H^{1}\cap G^{\delta}=H^{1\gamma}\cap G^{\delta}$.
		We claim the following proposition which works for general $\gamma$ and $\delta$:
		
		\begin{proposition}\label{Propvarphibijec}
			
			Let $\delta$ and $\gamma$ be as in Proposition \ref{Propheisdelta}, then for a non-zero homomorphism $\varphi\in\mathrm{Hom}_{J^{1}\cap J^{1\gamma}}(\eta^{\gamma},\eta)=\mathrm{Hom}_{J\cap J^{\gamma}}(\kappa^{0\gamma},\kappa^{0})$, it naturally induces an $R$-vector space isomorphism:
			\begin{align*}
				f_{\varphi}:\mathrm{Hom}_{J^{1}\cap G^{\delta}}(\eta,1)&\rightarrow\mathrm{Hom}_{J^{1\gamma}\cap G^{\delta}}(\eta^{\gamma},1),\\
				\lambda\quad&\mapsto\quad \lambda\circ\varphi
			\end{align*}
			
		\end{proposition}
		
		First we show that how Proposition \ref{Propvarphibijec} implies Proposition \ref{Propchichi'}. Using Proposition \ref{propheis} for $g$ and $\tau(g)$ respectively, we have $\mathrm{dim}_{R}\mathrm{Hom}_{J^{1g}\cap G^{\tau}}(\eta^{g},1)=\mathrm{dim}_{R}\mathrm{Hom}_{J^{1\tau(g)}\cap G^{\tau}}(\eta^{\tau(g)},1)=1$. By Proposition \ref{Propvarphibijec}, 
		\begin{align*}
			f_{\varphi}:\mathrm{Hom}_{J^{1g}\cap G^{\tau}}(\eta^{g},1)&\rightarrow\mathrm{Hom}_{J^{1\tau(g)}\cap G^{\tau}}(\eta^{\tau(g)},1),\\
			\lambda\quad&\mapsto\quad \lambda\circ\varphi,
		\end{align*}
		is bijective. 
		If we choose $$0\neq \lambda\in\mathrm{Hom}_{J^{1g}\cap G^{\tau}}(\eta^{g},1)\quad \text{and}\quad 0\neq \lambda':=f_{\varphi}(\lambda)=\lambda\circ\varphi\in\mathrm{Hom}_{J^{1\tau(g)}\cap G^{\tau}}(\eta^{\tau(g)},1),$$
		then for any $v$ in the representation space of $\eta$ and any $x\in J^{g}\cap G^{\tau}=J^{\tau(g)}\cap G^{\tau}$, we have
		\begin{align}
			\chi'(x)^{-1}\lambda'(v)&=\lambda'(\kappa^{0\tau(g)}(x)v)\qquad(\text{by Lemma \ref{lemkappa}.(1)}) \nonumber \\ 
			&=\lambda(\varphi(\kappa^{0\tau(g)}(x)v))\qquad(\text{by definition of}\ \lambda') \nonumber \\
			&=\lambda(\kappa^{0g}(x)\varphi(v))\qquad(\text{since}\ \varphi\in\mathrm{Hom}_{J^{g}\cap J^{\tau(g)}}(\kappa^{0\tau(g)},\kappa^{0g})) \nonumber \\
			&=\chi(x)^{-1}\lambda(\varphi(v))\qquad(\text{by Lemma \ref{lemkappa}.(1)}) \nonumber \\
			&=\chi(x)^{-1}\lambda'(v)\qquad(\text{by definition of}\ \lambda').  \label{eqchi'lambda'chilambda}
		\end{align}
		Since $v$ and $x\in J^{g}\cap G^{\tau}=J^{\tau(g)}\cap G^{\tau}$ are arbitrary, we have $\chi'|_{J^{\tau(g)}\cap G^{\tau}}=\chi|_{J^{g}\cap G^{\tau}}$ which is Proposition \ref{Propchichi'}.
		
		So we only need to focus on the proof of Proposition \ref{Propvarphibijec}. 
		
		\begin{lemma}\label{LemmaPhicont}
			
			Let $\delta$ and $\gamma$ be as in Proposition \ref{Propheisdelta}, then there exist an $R[J^{1}\cap J^{1\gamma}]$-module homomorphism
			$$\Phi:\eta^{\gamma(J^{1}:H^{1})^{1/2}}|_{J^{1}\cap J^{1\gamma}}\cong\mathrm{Ind}_{H^{1\gamma}}^{J^{1\gamma}}\theta^{\gamma}|_{J^{1}\cap J^{1\gamma}}\rightarrow\mathrm{Ind}_{H^{1}}^{J^{1}}\theta|_{J^{1}\cap J^{1\gamma}}\cong\eta^{(J^{1}:H^{1})^{1/2}}|_{J^{1}\cap J^{1\gamma}}$$
			and a linear form $\widetilde{L_{0}}\in\mathrm{Hom}_{J^{1}\cap G^{\delta}}(\eta^{(J^{1}:H^{1})^{1/2}},1)$,
			such that
			$$0\neq\widetilde{L_{0}}\circ\Phi\in\mathrm{Hom}_{J^{1}\cap G^{\delta}}(\eta^{\gamma(J^{1\gamma}:H^{1\gamma})^{1/2}},1).$$
			
		\end{lemma}
		
		\begin{proof}
			
			We prove this lemma by giving a direct construction of $\Phi$ and $\widetilde{L_{0}}$. First we choose our $\widetilde{L_{0}}$. We choose $\lambda_{0}\in\mathrm{Hom}_{J^{1}\cap G^{\delta}}(\mathrm{Ind}_{H^{1}\cap G^{\delta}}^{J^{1}\cap G^{\delta}}1,1)\cong R$ with the isomorphism given by Frobenius reciprocity, such that its corresponding image in $R$ equals 1. Then we choose $\widetilde{L_{0}}=(\lambda_{0},...,\lambda_{0})$ as an element in
			$$\bigoplus_{H^{1}\backslash J^{1}/J^{1}\cap J^{1\gamma}}\bigoplus_{H^{1}\cap J^{1\gamma}\backslash J^{1}\cap J^{1\gamma}/J^{1}\cap G^{\delta}}\mathrm{Hom}_{J^{1}\cap G^{\delta}}(\mathrm{Ind}_{H^{1}\cap G^{\delta}}^{J^{1}\cap G^{\delta}}1,1)\cong\mathrm{Hom}_{J^{1}\cap G^{\delta}}(\eta^{(J^{1}:H^{1})^{1/2}},1),$$
			where the isomorphism is determined by Lemma \ref{Lemmaindthetares}.(3), and by Lemma \ref{LemmaJHJgammaHgamma} the number of copies equals $(J^{1}:H^{1})^{1/2}$.
			
			Now we focus on the construction of $\Phi$. We define
			\begin{equation}\label{eqf0}
				f_{0}(g):=\begin{cases}
					\theta^{\gamma}(g_{1})\theta(g_{2})\ &\text{if}\quad g=g_{1}g_{2}\in(J^{1}\cap H^{1\gamma})(H^{1}\cap J^{1\gamma})\\
					0 & \text{if}\quad g\in J^{1}\cap J^{1\gamma}-(J^{1}\cap H^{1\gamma})(H^{1}\cap J^{1\gamma})
				\end{cases}.
			\end{equation}
			as a continuous function defined on $J^{1}\cap J^{1\gamma}$ with values in $R$. Since $(J^{1}\cap H^{1\gamma})\cap(H^{1}\cap J^{1\gamma})=H^{1}\cap H^{1\gamma}$ and $\theta^{\gamma}=\theta$ on $H^{1}\cap H^{1\gamma}$, we know that $f_{0}$ is well-defined.
			
			We want to verify that $f_{0}\in\mathrm{Ind}_{H^{1}\cap J^{1\gamma}}^{J^{1}\cap J^{1\gamma}}\theta$ and $f_{0}\in\mathrm{Ind}_{J^{1}\cap H^{1\gamma}}^{J^{1}\cap J^{1\gamma}}\theta^{\gamma}$. Since $J^{1}$ normalizes $H^{1}$ and $J^{1\gamma}$ normalizes $H^{1\gamma}$, by direct calculation $J^{1}\cap J^{1\gamma}$ normalizes $J^{1}\cap H^{1\gamma}$ and $H^{1}\cap J^{1\gamma}$. In particular, we have $(J^{1}\cap H^{1\gamma})(H^{1}\cap J^{1\gamma})=(H^{1}\cap J^{1\gamma})(J^{1}\cap H^{1\gamma})$. Moreover, since $J^{1}$ and $J^{1\gamma}$ normalize $\theta$ and $\theta^{\gamma}$ respectively, $(J^{1}\cap H^{1\gamma})(H^{1}\cap J^{1\gamma})=(H^{1}\cap J^{1\gamma})(J^{1}\cap H^{1\gamma})$ normalizes $\theta$ and $\theta^{\gamma}$.
			
			For $g_{1}'\in J^{1}\cap H^{1\gamma}$, $g_{2}'\in H^{1}\cap J^{1\gamma}$ and $g\in J^{1}\cap J^{1\gamma}$, if $g\notin(J^{1}\cap H^{1\gamma})(H^{1}\cap J^{1\gamma})$, then $g_{1}'g, g_{2}'g\notin(J^{1}\cap H^{1\gamma})(H^{1}\cap J^{1\gamma})$, thus $$f_{0}(g_{1}'g)=f_{0}(g_{2}'g)=0;$$ if $g=g_{1}g_{2}\in(J^{1}\cap H^{1\gamma})(H^{1}\cap J^{1\gamma})$, then $$f_{0}(g_{1}'g)=\theta^{\gamma}(g_{1}')\theta^{\gamma}(g_{1})\theta(g_{2})=\theta^{\gamma}(g_{1}')f_{0}(g)$$ and
			$$f_{0}(g_{2}'g)=f_{0}(g_{2}'g_{1}g_{2}'^{-1}g_{2}'g_{2})=\theta^{\gamma}(g_{2}'g_{1}g_{2}'^{-1})\theta(g_{2}')\theta(g_{2})=\theta(g_{2}')\theta^{\gamma}(g_{1})\theta(g_{2})=\theta(g_{2}')f_{0}(g).$$
			Considering these facts, we have $f_{0}\in\mathrm{Ind}_{H^{1}\cap J^{1\gamma}}^{J^{1}\cap J^{1\gamma}}\theta$ and $f_{0}\in\mathrm{Ind}_{J^{1}\cap H^{1\gamma}}^{J^{1}\cap J^{1\gamma}}\theta^{\gamma}$.
			
			We consider $J^{1}\cap J^{1\gamma}$-action on $f_{0}$ given by the right translation and we let $\langle f_{0}\rangle$ be the $R[J^{1}\cap J^{1\gamma}]$-subspace of both $\mathrm{Ind}_{J^{1}\cap H^{1\gamma}}^{J^{1}\cap J^{1\gamma}}\theta^{\gamma}$ and $\mathrm{Ind}_{H^{1}\cap J^{1\gamma}}^{J^{1}\cap J^{1\gamma}}\theta$ generated by $f_{0}$.  We choose $V_{f_{0}}$ to be an $R[J^{1}\cap J^{1\gamma}]$-invariant subspace of $\mathrm{Ind}_{J^{1}\cap H^{1\gamma}}^{J^{1}\cap J^{1\gamma}}\theta^{\gamma}$ such that $\mathrm{Ind}_{J^{1}\cap H^{1\gamma}}^{J^{1}\cap J^{1\gamma}}\theta^{\gamma}=\langle f_{0}\rangle\oplus V_{f_{0}}$.
			
			We define the $R[J^{1}\cap J^{1\gamma}]$-module homomorphism
			$$\Phi_{1}:\mathrm{Ind}_{J^{1}\cap H^{1\gamma}}^{J^{1}\cap J^{1\gamma}}\theta^{\gamma}\rightarrow\mathrm{Ind}_{H^{1}\cap J^{1\gamma}}^{J^{1}\cap J^{1\gamma}}\theta$$
			such that $\Phi_{1}(f_{0})=f_{0}$ and $\Phi_{1}|_{V_{f_{0}}}=0$. And we define
			$$\Phi:\bigoplus_{H^{1\gamma}\backslash J^{1\gamma}/J^{1}\cap J^{1\gamma}}\mathrm{Ind}_{J^{1}\cap H^{1\gamma}}^{J^{1}\cap J^{1\gamma}}\theta^{\gamma}\rightarrow\bigoplus_{H^{1}\backslash J^{1}/J^{1}\cap J^{1\gamma}}\mathrm{Ind}_{H^{1}\cap J^{1\gamma}}^{J^{1}\cap J^{1\gamma}}\theta$$
			given by $\Phi=\mathrm{diag}(\Phi_{1},0,...,0)\in \mathrm{M}_{N_{1}}(\mathrm{Hom}_{R[J^{1}\cap J^{1\gamma}]}(\mathrm{Ind}_{J^{1}\cap H^{1\gamma}}^{J^{1}\cap J^{1\gamma}}\theta^{\gamma},\mathrm{Ind}_{H^{1}\cap J^{1\gamma}}^{J^{1}\cap J^{1\gamma}}\theta))$, where the coordinates are indexed by $N_{1}:=|H^{1\gamma}\backslash J^{1\gamma}/J^{1}\cap J^{1\gamma}|=|H^{1}\backslash J^{1}/J^{1}\cap J^{1\gamma}|$. In particular, we let the first coordinate correspond to the trivial double cosets $H^{1\gamma}(J^{1}\cap J^{1\gamma})$ and $H^{1}(J^{1}\cap J^{1\gamma})$ respectively. As a result, $\Phi$ gives an $R[J^{1}\cap J^{1\gamma}]$-module homomorphism. By Lemma \ref{Lemmaindthetares} we have
			\begin{equation}\label{eqetaind}
				\eta^{(J^{1}:H^{1})^{1/2}}\cong\mathrm{Ind}_{H^{1}}^{J^{1}}\theta|_{J^{1}\cap J^{1\gamma}}\cong\bigoplus_{H^{1}\backslash J^{1}/J^{1}\cap J^{1\gamma}}\mathrm{Ind}_{H^{1}\cap J^{1\gamma}}^{J^{1}\cap J^{1\gamma}}\theta
			\end{equation}
			and
			\begin{equation}\label{eqgammaetaind}
				\eta^{\gamma(J^{1}:H^{1})^{1/2}}\cong\mathrm{Ind}_{H^{1\gamma}}^{J^{1\gamma}}\theta^{\gamma}|_{J^{1}\cap J^{1\gamma}}\cong\bigoplus_{H^{1\gamma}\backslash J^{1\gamma}/J^{1}\cap J^{1\gamma}}\mathrm{Ind}_{J^{1}\cap H^{1\gamma}}^{J^{1}\cap J^{1\gamma}}\theta^{\gamma}.
			\end{equation}
			With these two isomorphisms, we may regard $\Phi$ as a homomorphism from $\eta^{\gamma(J^{1}:H^{1})^{1/2}}|_{J^{1}\cap J^{1\gamma}}$ to $\eta^{(J^{1}:H^{1})^{1/2}}|_{J^{1}\cap J^{1\gamma}}$.
			
			Finally, we study $\widetilde{L_{0}}\circ\Phi$. First we calculate
			$$\Phi_{1}:\mathrm{Ind}_{J^{1}\cap H^{1\gamma}}^{J^{1}\cap J^{1\gamma}}\theta^{\gamma}|_{J^{1}\cap G^{\delta}}\rightarrow\mathrm{Ind}_{H^{1}\cap J^{1\gamma}}^{J^{1}\cap J^{1\gamma}}\theta|_{J^{1}\cap G^{\delta}}.$$
			We have the following isomorphism
			
			\begin{equation}\label{eqindthetares}
				\mathrm{Ind}_{H^{1}\cap J^{1\gamma}}^{J^{1}\cap J^{1\gamma}}\theta|_{J^{1}\cap G^{\delta}}\cong\bigoplus_{H^{1}\cap J^{1\gamma}\backslash J^{1}\cap J^{1\gamma}/J^{1}\cap G^{\delta}}\mathrm{Ind}_{H^{1}\cap G^{\delta}}^{J^{1}\cap G^{\delta}}1.
			\end{equation}
			
			
			By definition of $\Phi_{1}$ and (\ref{eqf0}),(\ref{eqindthetares}),
			$\Phi_{1}(f_{0}|_{J^{1}\cap G^{\delta}})=f_{0}|_{J^{1}\cap G^{\delta}}$ equals
			\begin{equation}\label{eqf0res}
				(\mathbf{1}_{H^{1}\cap G^{\delta}},...,\mathbf{1}_{H^{1}\cap G^{\delta}},0,...,0)\in\bigoplus_{H^{1}\cap J^{1\gamma}\backslash J^{1}\cap J^{1\gamma}/J^{1}\cap G^{\delta}}\mathrm{Ind}_{H^{1}\cap G^{\delta}}^{J^{1}\cap G^{\delta}}1,
			\end{equation}
			where the coordinates are indexed by the double coset $H^{1}\cap J^{1\gamma}\backslash J^{1}\cap J^{1\gamma}/J^{1}\cap G^{\delta}$, and those coordinates which equal the characteristic function $\mathbf{1}_{H^{1}\cap G^{\delta}}$ are exactly indexed by the subset $H^{1}\cap J^{1\gamma}\backslash (J^{1}\cap H^{1\gamma})(J^{1}\cap H^{1\gamma})/J^{1}\cap G^{\delta}$.
			
			We define $v_{0}=(f_{0}|_{J^{1}\cap G^{\delta}},0,...,0)$ as an element in both
			$$\bigoplus_{H^{1}\backslash J^{1}/J^{1}\cap J^{1\gamma}}\mathrm{Ind}^{J^{1}\cap J^{1\gamma}}_{J^{1}\cap H^{1\gamma}}\theta^{\gamma}|_{J^{1}\cap G^{\delta}}$$
			and
			$$\bigoplus_{H^{1\gamma}\backslash J^{1\gamma}/J^{1}\cap J^{1\gamma}}\mathrm{Ind}_{H^{1}\cap J^{1\gamma}}^{J^{1}\cap J^{1\gamma}}\theta|_{J^{1}\cap G^{\delta}},$$
			where the first coordinate corresponds to the trivial double cosets $H^{1}(J^{1}\cap J^{1\gamma})$ and $H^{1\gamma}(J^{1}\cap J^{1\gamma})$ respectively as in our definition of $\Phi$. Thus we have
			\begin{align*}(\widetilde{L_{0}}\circ\Phi)(v_{0})&=\widetilde{L_{0}}((\Phi_{1}(f_{0}|_{J^{1}\cap G^{\delta}}),0,...,0))=\widetilde{L_{0}}((f_{0}|_{J^{1}\cap G^{\delta}},0,...,0))\\
				&=|H^{1}\cap J^{1\gamma}\backslash (H^{1}\cap J^{1\gamma})(J^{1}\cap H^{1\gamma})/J^{1}\cap G^{\delta}|\cdot\lambda_{0}(\mathbf{1}_{H^{1}\cap G^{\delta}})\neq 0,
			\end{align*}
			where we use the definition of $\widetilde{L_{0}}$ and (\ref{eqf0res}) for the last equation. Thus we get $\widetilde{L_{0}}\circ\Phi\neq 0$ which finishes the proof.
			
		\end{proof}
		
		\begin{lemma}\label{Lemmapushpull}
			
			We keep the same notations as in Proposition \ref{Propvarphibijec} and we fix
			$$0\neq \lambda_{0}'\in\mathrm{Hom}_{J^{1}\cap G^{\delta}}(\eta,1)\quad \text{and}\quad 0\neq \lambda_{0}''\in\mathrm{Hom}_{J^{1}\cap G^{\delta}}(\eta^{\gamma},1).$$
			Then:
			
			(1) For any $\widetilde{L}\in\mathrm{Hom}_{J^{1}\cap G^{\delta}}(\eta^{(J^{1}:H^{1})^{1/2}},1)$, there exists an $R[J^{1}\cap J^{1\gamma}]$-homomorphism $$\mathrm{Pr}:\eta^{(J^{1}:H^{1})^{1/2}}|_{J^{1}\cap J^{1\gamma}}\rightarrow \eta|_{J^{1}\cap J^{1\gamma}}$$ such that $\widetilde{L}=\lambda_{0}'\circ \mathrm{Pr}$;
			
			(2) For any $\widetilde{L}\in\mathrm{Hom}_{J^{1}\cap G^{\delta}}(\eta^{\gamma(J^{1}:H^{1})^{1/2}},1)$, there exists an $R[J^{1}\cap J^{1\gamma}]$-homomorphism $$s:\eta^{\gamma}|_{J^{1}\cap J^{1\gamma}}\rightarrow\eta^{\gamma(J^{1}:H^{1})^{1/2}}|_{J^{1}\cap J^{1\gamma}}$$ such that $\lambda_{0}''=\widetilde{L}\circ s$.
			
		\end{lemma}
		
		\begin{proof}
			
			The proof is just a simple application of linear algebra. We write $N=(J^{1}:H^{1})^{1/2}$. For (1), we define $\mathrm{pr}_{i}:\eta^{(J^{1}:H^{1})^{1/2}}|_{J^{1}\cap J^{1\gamma}}\rightarrow \eta|_{J^{1}\cap J^{1\gamma}}$ as the projection with respect to the $i$-th coordinate. Since $\lambda_{0}'\circ\mathrm{pr}_{1} $,...,$\lambda_{0}'\circ\mathrm{pr}_{N}$ are linearly independent, and $\mathrm{dim}_{R}\mathrm{Hom}_{J^{1}\cap G^{\delta}}(\eta^{(J^{1}:H^{1})^{1/2}},1)=N$ by Proposition \ref{propheis}, $\lambda_{0}'\circ\mathrm{pr}_{1}$,...,$\lambda_{0}'\circ\mathrm{pr}_{N}$ generate $\mathrm{Hom}_{J^{1}\cap G^{\delta}}(\eta^{(J^{1}:H^{1})^{1/2}},1)$. So we may choose $\mathrm{Pr}$ to be a linear combination of $\mathrm{pr}_{j}$ which proves (1). The proof of (2) is similar.
			
		\end{proof}
		
		Now we finish the proof of Proposition \ref{Propvarphibijec}. Using Lemma \ref{Lemmapushpull}.(1) we choose $\mathrm{Pr}$ such that $\widetilde{L_{0}}=\lambda_{0}'\circ\mathrm{Pr}$, where $\widetilde{L_{0}}$ is defined as in the statement of Lemma \ref{LemmaPhicont}. Using Lemma \ref{LemmaPhicont}, there exists $\Phi$ such that $\widetilde{L_{0}}\circ\Phi\neq 0$. Using Lemma \ref{Lemmapushpull}.(2) we choose $s$ such that $\widetilde{L_{0}}\circ\Phi\circ s=\lambda_{0}''\neq 0$. We define $\varphi'=\mathrm{Pr}\circ\Phi\circ s$ and we have the following commutative diagram
		$$\xymatrix{
			\eta^{\gamma(J^{1}:H^{1})^{1/2}}|_{J^{1}\cap J^{1\gamma}} \ar[r]^{\Phi} & \eta^{(J^{1}:H^{1})^{1/2}}|_{J^{1}\cap J^{1\gamma}} \ar[d]^{\mathrm{Pr}} \\
			\eta^{\gamma}|_{J^{1}\cap J^{1\gamma}} \ar[u]^{s} \ar[r]^{\varphi'} & \eta|_{J^{1}\cap J^{1\gamma}}   }$$
		By definition we have $\lambda_{0}'\circ\varphi'=\lambda_{0}'\circ\mathrm{Pr}\circ\Phi\circ s=\lambda_{0}''\neq 0$, which means that $\varphi'\neq 0$. Since $\mathrm{Hom}_{J^{1}\cap J^{1\gamma}}(\eta^{\gamma},\eta)$ is of dimension 1, we deduce that $\varphi$ equals $\varphi'$ multiplying with a non-zero scalar, which means that $\lambda_{0}'\circ\varphi\neq 0$. Since $\mathrm{Hom}_{J^{1}\cap G^{\delta}}(\eta,1)$ and $\mathrm{Hom}_{J^{1}\cap G^{\delta}}(\eta^{\gamma},1)$ are of dimension 1, we know that $f_{\varphi}$ is an $R$-vector space isomorphism which proves Proposition \ref{Propvarphibijec}.
		
		\subsection{Existence of a $\tau$-selfdual extension of $\eta$}
		
		Now our aim is to choose a simple $\boldsymbol{\kappa}$ as an extension of $\eta$. Specifically, under the condition of Remark \ref{rempicond}, we show that we may assume $\boldsymbol{\kappa}$ to be $\tau$-selfdual, which means that $\boldsymbol{\kappa}^{\tau}\cong\boldsymbol{\kappa}^{\vee}$. First of all, we have the following lemma whose proof is the same as that in \cite{secherre2019supercuspidal}, Lemma 5.21:
		
		\begin{lemma}\label{lemmu}
			
			There exists a unique character $\mu$ of $\boldsymbol{J}$ trivial on $J^{1}$ such that $\boldsymbol{\kappa}^{\tau\vee}\cong\boldsymbol{\kappa}\mu$. It satisfies the identity $\mu\circ\tau=\mu$.
			
		\end{lemma}
		
		\begin{proposition}\label{propselfdual}
			
			When $\mathrm{char}(R)=0$, there exists a character $\phi$ of $\boldsymbol{J}$ trivial on $J^{1}$ such that $\mu=\phi(\phi\circ\tau)$. Moreover for any $R$, we may choose $\boldsymbol{\kappa}$ to be an extension of $\eta$ such that $\boldsymbol{\kappa}^{\tau\vee}\cong\boldsymbol{\kappa}$.
			
		\end{proposition}
		
		\begin{proof}
			
			First we consider the case where char$(R)=0$. we need the following elementary lemma:
			
			\begin{lemma}\label{lem2sroot}
				
				Assume $\mathrm{char}(R)=0$. For $N$ odd and $A\in\mathrm{GL}_{N}(R)$ such that $A^{2^{s}}=cI_{N}$ for $s\in\mathbb{N}$ and $c\in R^{\times}$, we have $\mathrm{Tr}(A)\neq 0$.
				
			\end{lemma}
			
			\begin{proof}
				
				$s=0$ is trivial, so from now on we assume $s\geq 1$. Let $\zeta_{2^{s}}$ be a primitive $2^{s}$-th root of $1$ in $R$ and let $c^{1/2^{s}}$ be a $2^{s}$-th root of $c$ in $R$, then we get $\mathrm{Tr}(A)=c^{1/2^{s}}\sum_{i=1}^{N}\zeta_{2^{s}}^{n_{i}}$ with $n_{i}\in\{0,1,2,...,2^{s}-1\}$. We know that $P(x)=x^{2^{s-1}}+1$ is the minimal polynomial of $\zeta_{2^{s}}$ in $\mathbb{Q}[x]$. If $\mathrm{Tr}(A)=0$, then for $Q(x)=\sum_{i=1}^{N}x^{n_{i}}$, we have $Q(\zeta_{2^{s}})=0$. As a result, $P(x)|Q(x)$ in $\mathbb{Q}[x]$ and thus in $\mathbb{Z}[x]$ by the Gauss lemma. However, the sum of all the coefficients of $P(x)$ is even and the sum of all the coefficients of $Q(x)$ equals $N$ which is odd. We get a contradiction. So $\mathrm{Tr}(A)\neq 0$.
				
			\end{proof}
			
			Come back to our proof. We choose $\boldsymbol{\kappa}$ to be an extension of $\eta$, thus as in Lemma \ref{lemmu}, there exists a character $\mu$ of $\boldsymbol{J}$ such that $\boldsymbol{\kappa}^{\tau\vee}\cong\boldsymbol{\kappa}\mu$. If $E/E_{0}$ is unramified, we let $$\overline{\mu}:\mathrm{GL}_{m}(\boldsymbol{l})\cong J/J^{1}\rightarrow R^{\times}$$ be the character whose inflation is $\mu|_{J}$. There exists a character $\varphi:\boldsymbol{l}^{\times}\rightarrow R^{\times}$ such that $\overline{\mu}=\varphi\circ\mathrm{det}$. Since $\overline{\mu}\circ\tau=\overline{\mu}$, we get $(\varphi\circ\sigma)\varphi=1$, or equivalently $\varphi|_{\boldsymbol{l}_{0}^{\times}}=1$, where $\boldsymbol{l}_{0}$ is the residue field of $E_{0}$, and $\sigma$ acts on $\boldsymbol{l}$ as the Frobenius map corresponding to $\boldsymbol{l}_{0}$. Let $Q$ be the cardinality of $\boldsymbol{l}_{0}$, then the cardinality of $\boldsymbol{l}$ is $Q^{2}$. If we fix $\zeta_{\boldsymbol{l}}$ a generator of $\boldsymbol{l}^{\times}$, then $\zeta_{\boldsymbol{l}}^{Q+1}$ is a generator of $\boldsymbol{l}_{0}^{\times}$. So we have $\varphi(\zeta_{l})^{Q+1}=1$. Choose $\alpha:\boldsymbol{l}^{\times}\rightarrow R^{\times}$ a character such that $$\alpha(\zeta_{\boldsymbol{l}}^{m})^{Q-1}=\varphi(\zeta_{\boldsymbol{l}})^{-m} \ \text{for}\ m\in\mathbb{Z}.$$
			Since $$\alpha(\zeta_{\boldsymbol{l}})^{Q^{2}-1}=\varphi(\zeta_{\boldsymbol{l}})^{-Q-1}=1,$$ we know that $\alpha$ is well-defined as a character of $\boldsymbol{l}^{\times}$. Moreover, we get $\varphi=\alpha(\alpha\circ\sigma)^{-1}$. Choose $\phi^{0}:J\rightarrow R^{\times}$ as the inflation of $\alpha\circ\mathrm{det}$, we get $\mu|_{J}=\phi^{0}(\phi^{0}\circ\tau).$
			
			Since $\varpi_{E}$ and $J$ generate $\boldsymbol{J}$, to choose $\phi$ as a character of $\boldsymbol{J}$ extending $\phi^{0}$, it suffices to show that $\mu(\varpi_{E})=1$. Since $\mu=\mu\circ\tau$, we get
			$$\mu(\varpi_{E})=\mu(\tau(\varpi_{E}))=\mu(\varpi_{E})^{-1}, \ \text{thus}\ \mu(\varpi_{E})\in\{1,-1\}.$$
			Let $e$ be the ramification index of $E/F$, and let $\varpi_{E}^{e}=a_{0}\varpi_{F}$ for a certain $a_{0}\in\mathfrak{o}_{E}^{\times}$. We have $$\varpi_{E}^{e(Q-1)}=a_{0}^{Q-1}\varpi_{F}^{Q-1}\ \text{with}\ a_{0}^{Q-1}\in 1+\mathfrak{p}_{E}\subset H^{1}(\mathfrak{a},\beta).$$ We write $e(Q-1)=2^{s}u$ for $2\nmid u$ and $s\in\mathbb{N}$. For $A=\boldsymbol{\kappa}(\varpi_{E}^{u})$, we have $$A^{2^{s}}=\boldsymbol{\kappa}(a_{0}^{Q-1}\varpi_{F}^{Q-1})=\theta(a_{0}^{Q-1})\omega_{\boldsymbol{\kappa}}(\varpi_{F}^{Q-1})I_{N},$$ where we use the fact that the restriction of $\boldsymbol{\kappa}$ to $H^{1}(\mathfrak{a},\beta)$ equals $N$-copies of $\theta$ with $N=(J^{1}:H^{1})^{1/2}$, and $\omega_{\boldsymbol{\kappa}}$ is the central character of $\boldsymbol{\kappa}$. Using Lemma \ref{lem2sroot} with $A$ and $c=\theta(a_{0}^{Q-1})\omega_{\boldsymbol{\kappa}}(\varpi_{F}^{Q-1})$, we get $\mathrm{Tr}(\boldsymbol{\kappa}(\varpi_{E}^{u}))\neq 0.$ Since $\boldsymbol{\kappa}^{\tau\vee}\cong\boldsymbol{\kappa}\mu$, considering the trace of both sides at $\varpi_{E}^{u}$, we get $$\mathrm{Tr}(\boldsymbol{\kappa}(\varpi_{E}^{u}))=\mathrm{Tr}(\boldsymbol{\kappa}(\varpi_{E}^{u}))\mu(\varpi_{E}^{u}),$$ thus $\mu(\varpi_{E}^{u})=1$. Since $u$ is odd and $\mu(\varpi_{E})$ equals either $1$ or $-1$, we get $\mu(\varpi_{E})=1$ which finishes the proof of this case.
			
			If $E/E_{0}$ is ramified, first we show that $\mu|_{\boldsymbol{l}^{\times}}=1$, where we consider the embedding $\boldsymbol{l}^{\times}\hookrightarrow E^{\times}$. Let $Q$ be the cardinality of $\boldsymbol{l}=\boldsymbol{l}_{0}$ and let $\zeta_{\boldsymbol{l}}$ be a generator of $\boldsymbol{l}^{\times}$, then we want to show that $\mu(\zeta_{\boldsymbol{l}})=1$. Write $Q-1=2^{s}u$ with $2\nmid u$ and use Lemma \ref{lem2sroot} with $A=\boldsymbol{\kappa}(\zeta_{\boldsymbol{l}}^{u})$ and $c=1$, we get $\mathrm{Tr}(\boldsymbol{\kappa}(\zeta_{\boldsymbol{l}}^{u}))\neq 0.$ Since $\boldsymbol{\kappa}^{\tau\vee}\cong\boldsymbol{\kappa}\mu$, we get $$\mathrm{Tr}(\boldsymbol{\kappa}(\zeta_{\boldsymbol{l}}^{u}))=\mathrm{Tr}(\boldsymbol{\kappa}(\zeta_{\boldsymbol{l}}^{u}))\mu(\zeta_{\boldsymbol{l}}^{u})$$ after considering the trace. Thus $\mu(\zeta_{\boldsymbol{l}}^{u})=1$. Since $\mu(\zeta_{\boldsymbol{l}})$ equals either $1$ or $-1$ which can be proved as the former case and $u$ is odd, we get $\mu(\zeta_{\boldsymbol{l}})=1$. Thus $\mu|_{J}=1$.
			
			To finish the definition of $\phi:\boldsymbol{J}\rightarrow R^{\times}$ such that $\mu=\phi(\phi\circ\tau)$, we only need to verify the equation $$\mu(\varpi_{E})=\phi(\varpi_{E})\phi(\tau(\varpi_{E}))=\phi(\varpi_{E})\phi(-\varpi_{E})^{-1}=\phi(-1)^{-1}.$$ Since we have already showed that $\mu(-1)=1$, using the relation $\mu=\mu\circ\tau$, we get $\mu(\varpi_{E}^{2})=\mu(-\varpi_{E}^{2})=\mu(\varpi_{E})\mu(\tau(\varpi_{E}))^{-1}=1$, so we deduce that $\mu(\varpi_{E})$ equals either $1$ or $-1$. Choose $\phi(-1)=\mu(\varpi_{E})$ which is well-defined, we finish the definition of $\phi$ such that $\mu=\phi(\phi\circ\tau)$. Let $\boldsymbol{\kappa}'=\boldsymbol{\kappa}\phi$, then $\boldsymbol{\kappa}'$ is $\tau$-selfdual.
			
			Now we suppose $R=\overline{\mathbb{F}_{l}}$. Let $\widetilde{\theta}$ be the lift of $\theta$ to $\overline{\mathbb{Q}_{l}}$ given by the canonical embedding $\overline{\mathbb{F}_{l}}^{\times}\hookrightarrow\overline{\mathbb{Q}_{l}}^{\times}$, then $\widetilde{\theta}$ is a simple character and $\widetilde{\theta}\circ\tau=\widetilde{\theta}^{-1}$. There is a $\tau$-selfdual representation $\widetilde{\boldsymbol{\kappa}}$ of $\boldsymbol{J}$ extending the Heisenberg representation $\widetilde{\eta}$ of $J^{1}$ corresponding to $\widetilde{\theta}$. Moreover we can further choose $\widetilde{\boldsymbol{\kappa}}$ such that the central character of $\widetilde{\boldsymbol{\kappa}}$ is integral. To do this, first we choose $\widetilde{\boldsymbol{\kappa}^{0}}$ to be a representation of $J$ extending $\eta$. We extend $\widetilde{\boldsymbol{\kappa}^{0}}$ to a representation of $F^{\times}J$. This requires us to choose a quasi-character $\widetilde{\omega}:F^{\times}\rightarrow \overline{\mathbb{Q}_{l}}^{\times}$ extending $\omega_{\widetilde{\boldsymbol{\kappa}^{0}}}$. We choose $\widetilde{\omega}$ such that it is integral. If we further extend this representation to $\widetilde{\boldsymbol{\kappa}}$ as a representation of $J=E^{\times}J$, then $\widetilde{\boldsymbol{\kappa}}$ is also integral. From the proof of the characteristic $0$ case, we may further assume $\widetilde{\boldsymbol{\kappa}}^{\tau\vee}\cong\widetilde{\boldsymbol{\kappa}}$ without losing the property that $\widetilde{\boldsymbol{\kappa}}$ is integral. By \cite{minguez2014types}, \S 2.11, the reduction of $\widetilde{\boldsymbol{\kappa}}$ to $R$, denoted by $\boldsymbol{\kappa}$, is thus a $\tau$-selfdual representation of $\boldsymbol{J}$ extending $\eta$.
			
			For $\mathrm{char}(R)=l>0$ in general, we fix $\iota:\overline{\mathbb{F}_{l}}\hookrightarrow R$ an embedding. For $\theta$ a simple character over $R$ as before which is of finite image, there exists a simple character $\theta_{0}$ over $\overline{\mathbb{F}_{l}}$ corresponding to the same simple stratum $[\mathfrak{a},\beta]$, such that $\theta=\iota\circ\theta_{0}$ and $\theta_{0}\circ\tau=\theta_{0}^{-1}$. Let $\eta_{0}$ be the Heisenberg representation of $\theta_{0}$ and choose $\boldsymbol{\kappa}_{0}$ to be a $\tau$-selfdual extension of $\eta_{0}$ by the former case. Then $\boldsymbol{\kappa}=\boldsymbol{\kappa}_{0}\otimes_{\overline{\mathbb{F}_{l}}} R$ is what we want.
			
		\end{proof}
		
		
		
		
		\subsection{Proof of Theorem \ref{thmtype}}\label{subsectionproofofthmdisctype}
		
		Using Proposition \ref{propselfdual}, we may assume that $\boldsymbol{\kappa}$ is $\tau$-selfdual, which means that $\boldsymbol{\kappa}^{\tau\vee}\cong\boldsymbol{\kappa}$. From its proof, when $R=\overline{\mathbb{F}_{l}}$, we assume further that $\boldsymbol{\kappa}$ is the reduction of a $\tau$-selfdual representation $\widetilde{\boldsymbol{\kappa}}$ of $\boldsymbol{J}$ over $\overline{\mathbb{Q}_{l}}$, and when $\mathrm{char}(R)=l>0$ in general, we assume $\boldsymbol{\kappa}$ to be realized as a $\overline{\mathbb{F}_{l}}$-representation via a certain fields embedding $\overline{\mathbb{F}_{l}}\hookrightarrow R$.
		
		\begin{proposition}\label{propchiquad}
			
			The character $\chi$ defined by Lemma \ref{lemkappa}.(1) is quadratic over $J^{g}\cap G^{\tau}$, that is, $\chi^{2}|_{J^{g}\cap G^{\tau}}=1$.
			
		\end{proposition}
		
		\begin{proof}
			
			First we assume that $\mathrm{char}(R)=0$. We have the following isomorphisms
			\begin{align}
				\mathrm{Hom}_{J^{1\tau(g)}\cap G^{\tau}}(\eta^{\tau(g)},1)&\cong\mathrm{Hom}_{J^{1g}\cap G^{\tau}}(\eta^{g},1) \nonumber \\
				&\cong\mathrm{Hom}_{\boldsymbol{J}^{g}\cap  G^{\tau}}(\boldsymbol{\kappa}^{g},\chi^{-1}) \nonumber  \\
				&\cong\mathrm{Hom}_{\boldsymbol{J}^{g}\cap G^{\tau}}(\chi,\boldsymbol{\kappa}^{g\vee})\quad (\text{by the duality of contragredient}) \nonumber  \\
				&\cong\mathrm{Hom}_{\boldsymbol{J}^{g}\cap G^{\tau}}(\boldsymbol{\kappa}^{g\vee},\chi)\quad (\text{since}\ \mathrm{char}(R)=0) \nonumber  \\
				&\cong\mathrm{Hom}_{\boldsymbol{J}^{g}\cap G^{\tau}}(\boldsymbol{\kappa}^{g\vee}\circ\tau,\chi\circ\tau) \nonumber  \\
				&\cong\mathrm{Hom}_{\boldsymbol{J}^{g}\cap G^{\tau}}((\boldsymbol{\kappa}^{\tau\vee})^{\tau(g)},\chi\circ\tau) \nonumber  \\
				&\cong\mathrm{Hom}_{\boldsymbol{J}^{\tau(g)}\cap G^{\tau}}(\boldsymbol{\kappa}^{\tau(g)},\chi\circ\tau)\quad(\text{since}\ \boldsymbol{\kappa}\ \text{is}\  \tau\text{-selfdual})  \label{eqhometakappaiso} .
			\end{align}
			Using Proposition \ref{Propchichi'} and the uniqueness of $\chi'$ in the \emph{loc. cit.}, we have $\chi\circ\tau=\chi^{-1}$. Since $\chi$ is defined on $\boldsymbol{J}^{g}\cap G^{\tau}$ which is $\tau$-invariant, we have $\chi\circ\tau=\chi$. Thus $\chi^{2}=\chi(\chi\circ\tau)=1$.
			
			If $R=\overline{\mathbb{F}_{l}}$, we denote by $\widetilde{\boldsymbol{\kappa}}$ a $\tau$-selfdual $\overline{\mathbb{Q}_{l}}$-lift of $\boldsymbol{\kappa}$ and we denote by $\widetilde{\chi}$ the character defined by Lemma \ref{lemkappa}.(1) with respect to $\widetilde{\boldsymbol{\kappa}}$ and $\widetilde{\eta}$, where $\widetilde{\eta}$ is a $J^{1}\cap G^{\tau}$-distinguished $\overline{\mathbb{Q}_{l}}$-lift of $\eta$. Using this proposition for $\overline{\mathbb{Q}_{l}}$-representations, we get $\widetilde{\chi}^{2}=1$. From the uniqueness of $\chi$, we know that $\widetilde{\chi}$ is a $\overline{\mathbb{Q}_{l}}$-lift of $\chi$. As a result, we get $\chi^{2}=1$.
			
			If $\mathrm{char}(R)=l>0$ in general, from the assumption of $\boldsymbol{\kappa}$ mentioned at the beginning of this subsection, via a fields embedding $\overline{\mathbb{F}_{l}}\hookrightarrow R$ we may realize all the representations mentioned in this proposition as representations over $\overline{\mathbb{F}_{l}}$, so we finish the proof by using the former case.
			
		\end{proof}
		
		As in the proof of Lemma \ref{lemdoucos}, we assume $g\in B^{\times}$ and
		\begin{equation}\label{eqgammabxcond}
			\gamma=bx,\quad b\tau(b)=1, \quad x\in K^{1},\quad b\in\varpi_{E}^{a_{1}}\mathrm{GL}_{m_{1}}(\mathfrak{o}_{E})\times...\times\varpi_{E}^{a_{r}}\mathrm{GL}_{m_{r}}(\mathfrak{o}_{E}).
		\end{equation}
		There exists a unique standard hereditary order $\mathfrak{b}_{m}\subseteq\mathfrak{b}$ such that
		$$U^{1}(\mathfrak{b}_{m})=(U\cap\delta(U^{1}))U^{1}=(U\cap U^{1\gamma})U^{1},$$
		where we define $\delta(y)=\gamma^{-1}\tau(y)\gamma$ for any $y\in G$ as an involution on $G$.
		We have the following lemma whose proof is the same as that in \cite{secherre2019supercuspidal}, Lemma 6.22, inspired by \cite{hakim2008distinguished}, Proposition 5.20:
		
		\begin{lemma}\label{lemU1bm}
			
			We have $U^{1}(\mathfrak{b}_{m})=(U^{1}(\mathfrak{b}_{m})\cap G^{\delta})U^{1}$.
			
		\end{lemma}
		
		\begin{theorem}\label{thmtaugg-1}
			
			Let $g\in G$ and suppose $\mathrm{Hom}_{\boldsymbol{J}^{g}\cap G^{\tau}}(\Lambda^{g},1)$ is non-zero. Then $\tau(g)g^{-1}\in\boldsymbol{J}$.
			
		\end{theorem}
		
		\begin{proof}
			
			It is enough to show that $r=1$ in (\ref{eqgammabxcond}). If not, $\mathfrak{b}_{m}$ by definition is a proper suborder of $\mathfrak{b}$. Furthermore, $\overline{U^{1}(\mathfrak{b}_{m})}:=U^{1}(\mathfrak{b}_{m})/U^{1}$ is a non-trivial unipotent subgroup of $U/U^{1}\cong\mathrm{GL}_{m}(\boldsymbol{l})$. Using Lemma \ref{lemkappa}.(2), we have
			$$\mathrm{Hom}_{\boldsymbol{J}\cap G^{\delta}}(\boldsymbol{\rho},\chi^{g^{-1}})\cong\mathrm{Hom}_{\boldsymbol{J}^{g}\cap G^{\tau}}(\boldsymbol{\rho}^{g},\chi)\neq 0.$$
			Restricting to $U^{1}(\mathfrak{b}_{m})\cap G^{\delta}$, we have
			\begin{equation}\label{eqrho}
				\mathrm{Hom}_{U^{1}(\mathfrak{b}_{m})\cap G^{\delta}}(\boldsymbol{\rho},\chi^{g^{-1}})\neq 0.
			\end{equation}
			Using Lemma \ref{lemU1bm}, we have the isomorphism
			$$(U^{1}(\mathfrak{b}_{m})\cap G^{\delta})U^{1}/U^{1}\cong U^{1}(\mathfrak{b}_{m})/U^{1}.$$
			We denote by $\overline{\rho}$ the cuspidal representation of $U^{0}/U^{1}\cong\mathrm{GL}_{m}(\boldsymbol{l})$ whose inflation is $\boldsymbol{\rho}|_{U^{0}}$, and by $\overline{\chi^{g^{-1}}}$ the character of $\overline{U^{1}(\mathfrak{b}_{m})}$ whose inflation is $\chi^{g^{-1}}$. So if we consider the equation (\ref{eqrho}) modulo $U^{1}$, then we get
			$$\mathrm{Hom}_{\overline{U^{1}(\mathfrak{b}_{m})}}(\overline{\rho},\overline{\chi^{g^{-1}}})\neq 0.$$
			Since $\chi^{g^{-1}}|_{J\cap G^{\delta}}$ is quadratic and $\overline{U^{1}(\mathfrak{b}_{m})}$ is a $p$-group with $p\neq 2$, we get $\overline{\chi^{g^{-1}}}=1$, thus
			$$\mathrm{Hom}_{\overline{U^{1}(\mathfrak{b}_{m})}}(\overline{\rho},1)\neq 0,$$
			which contradicts to the fact that $\overline{\rho}$ is cuspidal.
			
		\end{proof}
		
		\begin{proof}[Proof of Theorem \ref{thmtype}]
			
			If there exists a $\tau$-selfdual simple type $(\boldsymbol{J},\Lambda)$ in $\pi$ such that $\mathrm{Hom}_{\boldsymbol{J}\cap G^{\tau}}(\Lambda,1)$ is non-zero, then $\pi$ is $G^{\tau}$-distinguished. Conversely, there exists $g\in G$ such that $\mathrm{Hom}_{\boldsymbol{J}^{g}\cap G^{\tau}}(\Lambda^{g},1)$ is non-zero. Using Theorem \ref{thmtaugg-1}, we conclude that $(\boldsymbol{J}^{g},\Lambda^{g})$ is a $\tau$-selfdual simple type.
			
		\end{proof}
		
		Finally we state the following corollary of Theorem \ref{thmtaugg-1} as the end of this section:
		
		\begin{corollary}\label{corJGtau}
			
			Under the assumption of Theorem \ref{thmtaugg-1}, we have $g\in\boldsymbol{J}G^{\tau}$ or $g\in\boldsymbol{J}g_{1}G^{\tau}$, where the latter case exists only if $m$ is even, and $g_{1}\in B^{\times}$ is fixed such that
			$$\tau(g_{1})g_{1}^{-1}=\begin{cases}\varpi_{E}I_{m} &\quad \text{if}\ E/E_{0}\ \text{is unramified.}\\
				\varpi_{E}J_{m/2} &\quad \text{if}\ E/E_{0}\ \text{is ramified.}\end{cases}$$
			As a result, $$\mathrm{Hom}_{G^{\tau}}(\pi,1)\cong\mathrm{Hom}_{\boldsymbol{J}\cap G^{\tau}}(\Lambda,1)\oplus\mathrm{Hom}_{\boldsymbol{J}^{g_{1}}\cap G^{\tau}}(\Lambda^{g_{1}},1).$$
			
		\end{corollary}
		
		\begin{proof}
			
			Recall that we have already assumed that $g\in B^{\times}$. Since $\tau(g)g^{-1}\in \boldsymbol{J}\cap B^{\times}=E^{\times}\mathfrak{b}^{\times}$, changing $g$ up to multiplying by an element in $E^{\times}$ which doesn't change the double coset it represents, we may assume $(g^{*})^{-1}\varepsilon^{-1}g^{-1}\in\mathfrak{b}^{\times}$ or $\varpi_{E}\mathfrak{b}^{\times}$, where $\varepsilon$ equals $I_{m}$ for $E/E_{0}$ unramified\footnote{It is also possible in the unramified case that $\varepsilon=\mathrm{diag}(\varpi_{E},...,\varpi_{E})$. However in this case $\varepsilon\in E^{\times}$ which commutes with $B^{\times}$, thus this case can be combined into the case where $\varepsilon=I_{m}$.} and $\varepsilon$ equals $I_{m}$ or $\mathrm{diag}(1,...,1,\epsilon)$ with $\epsilon\in \mathfrak{o}_{E_{0}}^{\times}-\mathrm{N}_{E/E_{0}}(\mathfrak{o}_{E}^{\times})$ for $E/E_{0}$ ramified. Using Proposition \ref{PropJOH}, we may change $g^{-1}$ up to multiplying by an element in $\mathfrak{b}^{\times}$ on the right, thus we may write $(g^{*})^{-1}\varepsilon^{-1}g^{-1}=\varpi_{E}^{\alpha}$, where $\varpi_{E}^{\alpha}$ is defined as in \S \ref{subsectionunitary}. Thus we get $\mathrm{det}_{B}(\varpi_{E}^{\alpha})/\mathrm{det}_{B}(\varepsilon^{-1})\in\mathrm{N}_{E/E_{0}}(E^{\times})$.
			
			If $(g^{*})^{-1}\varepsilon^{-1}g^{-1}\in\mathfrak{b}^{\times}$, from the definition and the uniqueness of $\varpi_{E}^{\alpha}$ in Proposition \ref{PropJOH}, we get $\varpi_{E}^{\alpha}=\varepsilon$. We may further change $g^{-1}$ up to multiplying by an element in $\mathfrak{b}^{\times}$ on the right, such that $(g^{*})^{-1}\varepsilon^{-1}g^{-1}=\varepsilon^{-1}$. Thus we get $\tau(g)=\varepsilon(g^{*})^{-1}\varepsilon^{-1}=g$, which means that $g\in G^{\tau}$.
			
			If $(g^{*})^{-1}\varepsilon^{-1}g^{-1}\in \varpi_{E}\mathfrak{b}^{\times}$. Considering the determinant we deduce that $\mathrm{det}_{B}((g^{*})^{-1}\varepsilon^{-1}g^{-1})\in E^{\times}$ is of even order with respect to the discrete valuation of $E$. Since the determinant of elements in $\varpi_{E}\mathfrak{b}^{\times}$ is of order $m$, we know that $m$ is even. Thus from the definition and the uniqueness of $\varpi_{E}^{\alpha}$ in Proposition \ref{PropJOH}, we get $\varpi_{E}^{\alpha}=\varpi_{E}\varepsilon$ when $E/E_{0}$ is unramified and $\varpi_{E}^{\alpha}=\varpi_{E}J_{m/2}$ when $E/E_{0}$ is ramified. For the former case, we have $\varepsilon=I_{m}$. Using Proposition \ref{PropGOH}, we may choose $g_{1}\in B^{\times}$ such that $(g_{1}^{*})^{-1}g_{1}^{-1}=\varpi_{E}I_{m}=(g^{*})^{-1}g^{-1}$. Thus $g\in g_{1}G^{\tau}$. For the latter case, considering the determinant we must have $\mathrm{det}_{B}(\varepsilon)\in\mathrm{N}_{E/E_{0}}(E^{\times})$, thus $\varepsilon=I_{m}$. Using Proposition \ref{PropGOH}, we may choose $g_{1}\in B^{\times}$ such that $(g_{1}^{*})^{-1}g_{1}^{-1}=\varpi_{E}J_{m/2}=(g^{*})^{-1}g^{-1}$, thus $g\in g_{1}G^{\tau}$.
			
		\end{proof}

		\section{The supercuspidal unramified case}\label{sectionunram}
		
		In this section, we study the distinction of $\sigma$-invariant supercuspidal representations of $G$ in the case where $E/E_{0}$ is unramified.
		
		\subsection{The finite field case}
		
		In this subsection, we assume $\boldsymbol{l}/\boldsymbol{l}_{0}$ to be a quadratic extension of finite fields with characteristic $p\neq 2$. Let $|\boldsymbol{l}_{0}|=Q$, then $|\boldsymbol{l}|=Q^{2}$. Let $\sigma$ be the non-trivial involution in $\mathrm{Gal}(\boldsymbol{l}/\boldsymbol{l}_{0})$.
		
		Let $m$ be a positive integer and let $\boldsymbol{t}$ be an extension of degree $m$ over $\boldsymbol{l}$. We identify $\boldsymbol{t}^{\times}$ with a maximal torus of $\mathrm{GL}_{m}(\boldsymbol{l})$. We call a character $\xi:\boldsymbol{t}^{\times}\rightarrow R^{\times}$ $\boldsymbol{l}$-regular (or regular for short) if for any $i=1,...,m-1$, we have $\xi^{|\boldsymbol{l}|^{i}}\neq\xi$. By Green \cite{green1955characters} when $\mathrm{char}(R)=0$ and James \cite{james1986irreducible} when $\mathrm{char}(R)=l>0$ prime to $p$, there is a surjective map
		$$\xi\mapsto\overline{\rho_{\xi}}$$
		between $\boldsymbol{l}$-regular characters of $\boldsymbol{t}^{\times}$ and isomorphism classes of supercuspidal representations of $\mathrm{GL}_{m}(\boldsymbol{l})$, whose fibers are $\mathrm{Gal}(\boldsymbol{t}/\boldsymbol{l})$-orbits. 
		
		\begin{lemma}\label{Lemmafinsiginv}
			
			(1) If there exists a $\sigma$-invariant supercuspidal representation of $\mathrm{GL}_{m}(\boldsymbol{l})$, then $m$ is odd.
			
			(2) When $\mathrm{char}(R)=0$, the converse of (1) is true.
			
		\end{lemma}
		
		\begin{proof}
			
			We may follow the same proof of \cite{secherre2019supercuspidal}, Lemma 2.3, with the concept $\sigma$-selfdual in \emph{loc. cit.} replaced by $\sigma$-invariant and the corresponding contragredient (or inverse) replaced by the identity.
			
			
			
		\end{proof}
		
		Let $H=\mathrm{U}_{m}(\boldsymbol{l}/\boldsymbol{l}_{0}):=\mathrm{U}_{m}(I_{m})$ be the unitary subgroup of $\mathrm{GL}_{m}(\boldsymbol{l})$ corresponding to the hermitian matrix $I_{m}$ with respect to $\bs{l}/\bs{l_{0}}$. Note that there is only one conjugacy class of unitary subgroup of $\overline{G}$, which is isomorphic to $H$.
		
		\begin{lemma}\label{Lemmafinunidist}
			
			Suppose $m$ to be odd and let $\overline{\rho}$ be a supercuspidal representation of $\mathrm{GL}_{m}(\boldsymbol{l})$. The following assertions are equivalent:
			
			(1) The representation $\overline{\rho}$ is $\sigma$-invariant;
			
			(2) The representation $\overline{\rho}$ is $H$-distinguished;
			
			(3) The $R$-vector space $\mathrm{Hom}_{H}(\overline{\rho},1)$ has dimension 1.
			
		\end{lemma}
		
		\begin{proof}
			
			When $R$ has characteristic 0, this is \cite{gow1984two}, Theorem 2.1, Theorem 2.4. Suppose now that $R=\overline{\mathbb{F}_{l}}$. First we prove that (1) is equivalent to (2).
			
			For $\overline{\rho}$ a supercuspidal representation of $\mathrm{GL}_{m}(\boldsymbol{l})$, we denote by $P_{\overline{\rho}}$ the projective envelope of $\overline{\rho}$ as a $\overline{\mathbb{Z}_{l}}[\mathrm{GL}_{m}(\boldsymbol{l})]$-module, where $\overline{\mathbb{Z}_{l}}$ is the ring of integers of $\overline{\mathbb{Q}_{l}}$. Using \cite{vigneras1996representations}, Chapitre III, Th\'eor\`eme 2.9 and \cite{serre2012linear}, Proposition 42, we have:
			
			(1) $P_{\overline{\rho}}\otimes_{\overline{\mathbb{Z}_{l}}}\overline{\mathbb{F}_{l}}$ is the projective envelope of $\overline{\rho}$ as a $\overline{\mathbb{F}_{l}}[\mathrm{GL}_{m}(\boldsymbol{l})]$-module, which is indecomposable of finite length with each irreducible component isomorphic to $\overline{\rho}$;

			(2) For $\widetilde{P_{\overline{\rho}}}=P_{\overline{\rho}}\otimes_{\overline{\mathbb{Z}_{l}}}\overline{\mathbb{Q}_{l}}$, we have $\widetilde{P_{\overline{\rho}}}\cong\bigoplus\widetilde{\overline{\rho}}$, where $\widetilde{\overline{\rho}}$ in the direct sum ranges over all the supercuspidal $\overline{\mathbb{Q}_{l}}$-lifts of $\overline{\rho}$ and appears with multiplicity 1. We have
			\begin{align*}
				\mathrm{Hom}_{H}(\overline{\rho},1)\neq 0; &\Longleftrightarrow \mathrm{Hom}_{\overline{\mathbb{F}_{l}}[\mathrm{GL}_{m}(\boldsymbol{l})]}(\overline{\rho},\overline{\mathbb{F}_{l}}[H\backslash\mathrm{GL}_{m}(\boldsymbol{l})])\neq 0;\\
				&\Longleftrightarrow \mathrm{Hom}_{\overline{\mathbb{F}_{l}}[\mathrm{GL}_{m}(\boldsymbol{l})]}(P_{\overline{\rho}}\otimes_{\overline{\mathbb{Z}_{l}}}\overline{\mathbb{F}_{l}},\overline{\mathbb{F}_{l}}[H\backslash\mathrm{GL}_{m}(\boldsymbol{l})])\neq 0;\\
				&\Longleftrightarrow \mathrm{Hom}_{\overline{\mathbb{Z}_{l}}[\mathrm{GL}_{m}(\boldsymbol{l})]}(P_{\overline{\rho}},\overline{\mathbb{Z}_{l}}[H\backslash\mathrm{GL}_{m}(\boldsymbol{l})])\neq 0;\\
				&\Longleftrightarrow \mathrm{Hom}_{\overline{\mathbb{Q}_{l}}[\mathrm{GL}_{m}(\boldsymbol{l})]}(\widetilde{P_{\overline{\rho}}},\overline{\mathbb{Q}_{l}}[H\backslash\mathrm{GL}_{m}(\boldsymbol{l})])\neq 0;\\
				&\Longleftrightarrow \text{There exists}\ \widetilde{\overline{\rho}}\ \text{as above such that}\ \mathrm{Hom}_{\overline{\mathbb{Q}_{l}}[\mathrm{GL}_{m}(\boldsymbol{l})]}(\widetilde{\overline{\rho}},\overline{\mathbb{Q}_{l}}[H\backslash\mathrm{GL}_{m}(\boldsymbol{l})])\neq 0;\\
				&\Longleftrightarrow \text{There exists}\ \widetilde{\overline{\rho}}\ \text{as above such that}\ \widetilde{\overline{\rho}}^{\sigma}=\widetilde{\overline{\rho}};\\
				&\Longleftrightarrow \overline{\rho}^{\sigma}=\overline{\rho}.
			\end{align*}
			The former five equivalences are direct, by noting that a projective $\bzl[\mrgl_{m}(\bs{l})]$-module is a free $\bzl$-module. For the second last equivalence we use the result for the characteristic 0 case. For the last equivalence from the construction of supercuspidal representation given by Green and James, since it is always possible to lift a $\sigma$-invariant regular character over $\overline{\mathbb{F}_{l}}$ to a $\sigma$-invariant regular character over $\overline{\mathbb{Q}_{l}}$, it is always possible to find a $\sigma$-invariant $\overline{\mathbb{Q}_{l}}$-lift $\widetilde{\overline{\rho}}$ for a $\sigma$-invariant supercuspidal representation $\overline{\rho}$.
			
			Since (3) implies (2) by definition, we only need to prove (2) implies (3). We sum up the proof occurring in \cite{secherre2019supercuspidal}, Lemma 2.19. We have the following $\overline{\mathbb{F}_{l}}[\mathrm{GL}_{m}(\boldsymbol{l})]$-module decomposition $$\overline{\mathbb{F}_{l}}[H\backslash\mathrm{GL}_{m}(\boldsymbol{l})]=V_{\overline{\rho}}\oplus V',$$ where $V_{\overline{\rho}}$ is composed of irreducible components isomorphic to $\overline{\rho}$, and $V'$ has no irreducible component isomorphic to $\overline{\rho}$. First we verify that $\mathrm{End}_{\overline{\mathbb{F}_{l}}[\mathrm{GL}_{m}(\boldsymbol{l})]}(V_{\overline{\rho}})$ is commutative.
			By \cite{gow1984two}, Theorem 2.1, the convolution algebra $\overline{\mathbb{Z}_{l}}[H\backslash\mathrm{GL}_{m}(\boldsymbol{l})/H]$ is commutative. Modulo $l$ we deduce that $$\overline{\mathbb{F}_{l}}[H\backslash\mathrm{GL}_{m}(\boldsymbol{l})/H]\cong\mathrm{End}_{\overline{\mathbb{F}_{l}}[\mathrm{GL}_{m}(\boldsymbol{l})]}(\overline{\mathbb{F}_{l}}[H\backslash\mathrm{GL}_{m}(\boldsymbol{l})])\cong\mathrm{End}_{\overline{\mathbb{F}_{l}}[\mathrm{GL}_{m}(\boldsymbol{l})]}(V_{\overline{\rho}})\oplus\mathrm{End}_{\overline{\mathbb{F}_{l}}[\mathrm{GL}_{m}(\boldsymbol{l})]}(V')$$ is commutative, thus $\mathrm{End}_{\overline{\mathbb{F}_{l}}[\mathrm{GL}_{m}(\boldsymbol{l})]}(V_{\overline{\rho}})$ is commutative.
			
			By \cite{vigneras1996representations}, Chapitre III, Th\'eor\`eme 2.9, $P=P_{\overline{\rho}}\otimes_{\bzl}\bfl$ is indecomposable with each irreducible subquotient isomorphic to $\overline{\rho}$. By \cite{dat2012theorie}, Proposition B.1.2, there exists a nilpotent endomorphism $N\in\mathrm{End}_{\overline{\mathbb{F}_{l}}[\mathrm{GL}_{m}(\boldsymbol{l})]}[P]$ such that $\mathrm{End}_{\overline{\mathbb{F}_{l}}[\mathrm{GL}_{m}(\boldsymbol{l})]}[P]=\overline{\mathbb{F}_{l}}[N]$, and there exist $r\geq 1$ and $n_{1},...,n_{r}$ positive integers such that
			$$V_{\overline{\rho}}\cong\bigoplus_{i=1}^{r}P/N^{n_{i}}P.$$
			Since $\mathrm{End}_{\overline{\mathbb{F}_{l}}[\mathrm{GL}_{m}(\boldsymbol{l})]}(V_{\overline{\rho}})$ is commutative, we have $r=1$ and $V_{\overline{\rho}}=P/N^{n_{1}}P$. Thus
			$$\mathrm{Hom}_{H}(\overline{\rho},1)\cong\mathrm{Hom}_{\mathrm{GL}_{m}(\boldsymbol{l})}(\overline{\rho},V_{\overline{\rho}})=\mathrm{Hom}_{\mathrm{GL}_{m}(\boldsymbol{l})}(\overline{\rho},P/N^{n_{1}}P)\cong\overline{\mathbb{F}_{l}}.$$
			
			For $\mathrm{char}(R)=l>0$ in general, we fix an embedding $\overline{\mathbb{F}_{l}}\hookrightarrow R$ and write $\overline{\rho}=\overline{\rho}_{0}\otimes_{\overline{\mathbb{F}_{l}}}R$, where $\overline{\rho}_{0}$ is a supercuspidal representation of $\mathrm{GL}_{m}(\boldsymbol{l})$ over $\overline{\mathbb{F}_{l}}$. By considering the Brauer characters, we have
			$$\overline{\rho}^{\sigma}\cong\overline{\rho}\quad\text{if and only if}\quad \overline{\rho}_{0}^{\sigma}\cong\overline{\rho}_{0}.$$
			Moreover
			$$\mathrm{Hom}_{R[H]}(\overline{\rho},R)\cong\mathrm{Hom}_{\overline{\mathbb{F}_{l}}[H]}(\overline{\rho}_{0},\bfl)\otimes_{\bfl}R.$$
			Thus we come back to the former case.
			
		\end{proof}
		
		\begin{remark}
			
			We give an example of a $\sigma$-invariant cuspidal non supercuspidal representation of $\mathrm{GL}_{m}(\boldsymbol{l})$ over $\overline{\mathbb{F}_{l}}$ which is not distinguished by $H$. Assume $m=2$ and $l\neq 2$ such that $l|Q^{2}+1$. Let $B$ be the subgroup of $\mathrm{GL}_{2}(\boldsymbol{l})$ consisting of upper triangular matrices. For $\mathrm{Ind}_{B}^{\mathrm{GL}_{2}(\boldsymbol{l})}\bfl$, it is a representation of length 3 with irreducible components of dimension $1,Q^{2}-1,1$ respectively. Denote by $\overline{\rho}$ the irreducible subquotient of $\mathrm{Ind}_{B}^{\mathrm{GL}_{2}(\boldsymbol{l})}\bfl$ of dimension $Q^{2}-1$ . It is thus cuspidal (not supercuspidal) and $\sigma$-invariant. Let $\widetilde{\overline{\rho}}$ be a $\overline{\mathbb{Q}_{l}}$-lift of $\overline{\rho}$ which is an irreducible cuspidal representation. We write $\widetilde{\overline{\rho}}|_{H}=V_{1}\oplus...\oplus V_{r}$ its decomposition of irreducible components. Since $|H|=Q(Q+1)(Q^{2}-1)$ is prime to $l$, reduction modulo $l$ preserves irreducibility. So $\overline{\rho}|_{H}$ decomposes as $W_{1}\oplus...\oplus W_{r}$, where the irreducible representation $W_{i}$ is the reduction of $V_{i}$ modulo $l$ for each $i=1,...,r$. Suppose that $\overline{\rho}$ is distinguished. Then $W_{i}=\overline{\mathbb{F}_{l}}$ for some $i$. Thus $V_{i}$ is a character which must be trivial, which implies that $\widetilde{\overline{\rho}}$ is distinguished. It is impossible by Lemma \ref{Lemmafinsiginv} and Lemma \ref{Lemmafinunidist}, since $m=2$ is even. See \cite{secherre2019supercuspidal}, Remark 2.8. for the Galois selfdual case.
			
		\end{remark}
		
		Finally, we need the following finite group version of Proposition \ref{PropGelKa} which is well-known:
		
		\begin{proposition}\label{PropGelKafin}
			
			For $\overline{\rho}$ an irreducible representation of $\mathrm{GL}_{m}(\boldsymbol{l})$, we have $\overline{\rho}^{\vee}\cong\overline{\rho}(\,^{t}\cdot^{-1})$, where $\overline{\rho}(\,^{t}\cdot^{-1}):x\mapsto\overline{\rho}(\,^{t}x^{-1})$ for any $x\in\mathrm{GL}_{m}(\boldsymbol{l})$.
			
		\end{proposition}
		
		\begin{proof}
			
			By definition, the Brauer characters of $\overline{\rho}^{\vee}$ and $\overline{\rho}(\,^{t}\cdot^{-1})$ are the same.
			
		\end{proof}
		
		\subsection{Distinction criterion in the unramified case}
		
		Let $\pi$ be a $\sigma$-invariant supercuspidal representation of $G$. In this subsection we prove Theorem \ref{Thmmain} and Theorem \ref{Thmmult1} in the case where $E/E_{0}$ is unramified. To prove Theorem \ref{Thmmain}, it remains to show that $\pi$ is distinguished by any unitary subgroup $G^{\tau}$ with the aid of Theorem \ref{Thmdistgalinv}. Since changing $\tau$ up to a $G$-action doesn't change the content of the theorem, we only need to consider two special unitary involutions mentioned in Remark \ref{rempicond}.(4). To justify the assumption in Remark \ref{rempicond}.(3), first we prove the following lemma:
		
		\begin{lemma}\label{Lemmamodd}
			
			For any $\sigma$-invariant supercuspidal representation $\pi$ with $E/E_{0}$ unramified, $m$ is odd.
			
		\end{lemma}
		
		\begin{proof}
			
			We consider $\tau=\tau_{1}$, where $\tau_{1}(x)=\sigma(\,^{t}x^{-1})$ for any $x\in G$. We follow the settings of Remark \ref{rempicond}. For $(\boldsymbol{J},\Lambda)$ a simple type as in Remark \ref{rempicond}.(2), we may write $\Lambda\cong\boldsymbol{\kappa}\otimes\boldsymbol{\rho}$ as before. Using Proposition \ref{propselfdual}, we may further assume $\boldsymbol{\kappa}^{\tau\vee}\cong\boldsymbol{\kappa}$.
			Since $\Lambda$ and $\boldsymbol{\kappa}$ are $\tau$-selfdual,  $\boldsymbol{\rho}$ is $\tau$-selfdual. Let $\overline{\rho}$ be the supercuspidal representation of $\mathrm{GL}_{m}(\boldsymbol{l})\cong J/J^{1}$ whose inflation equals $\boldsymbol{\rho}|_{J}$, then $\overline{\rho}^{\tau\vee}\cong\overline{\rho}$ when regarding $\tau$ as a unitary involution on $\mathrm{GL}_{m}(\boldsymbol{l})$. Using Proposition \ref{PropGelKafin}, we have $\overline{\rho}\circ\sigma\cong\overline{\rho}$. Using Lemma \ref{Lemmafinsiginv}, we conclude that $m$ is odd.
			
		\end{proof}
		
		With the aid of Lemma \ref{Lemmamodd}, we may assume as in Remark \ref{rempicond}.(4) that $\tau(x)=\varepsilon\sigma(\,^{t}x^{-1})\varepsilon^{-1}$ for any $x\in G$ with $\varepsilon$ equal to $I_{n}$ or $\mathrm{diag}(\varpi_{E},...,\varpi_{E})$, representing the two classes of unitary involutions. For $(\boldsymbol{J},\Lambda)$ a simple type as in Remark \ref{rempicond}.(2), we may write $\Lambda\cong\boldsymbol{\kappa}\otimes\boldsymbol{\rho}$ as before. Using Proposition \ref{propselfdual}, we may further assume $\boldsymbol{\kappa}^{\tau\vee}\cong\boldsymbol{\kappa}$. Using Lemma \ref{lemkappa} with $g=1$, there exists a quadratic character $\chi:\boldsymbol{J}\cap G^{\tau}\rightarrow R^{\times}$ such that $$\mathrm{dim}_{R}\mathrm{Hom}_{\boldsymbol{J}\cap G^{\tau}}(\boldsymbol{\kappa},\chi^{-1})=1$$
		and
		$$\mathrm{Hom}_{\boldsymbol{J}\cap G^{\tau}}(\Lambda,1)\cong\mathrm{Hom}_{\boldsymbol{J}\cap G^{\tau}}(\boldsymbol{\kappa},\chi^{-1})\otimes_{R}\mathrm{Hom}_{\boldsymbol{J}\cap G^{\tau}}(\boldsymbol{\rho},\chi).$$
		We want to show that $\chi=1$. First we need the following lemma:
		
		\begin{lemma}\label{Lemmachiext}
			
			The character $\chi$ can be extended to a character $\chi'$ of $\boldsymbol{J}$.
			
		\end{lemma}
		
		\begin{proof}
			
			Using Lemma \ref{LemmaunitaryJJ0}, we have $\boldsymbol{J}\cap G^{\tau}=J\cap G^{\tau}$. Write $\overline{\chi}$ the character of $\mathrm{U}_{m}(\boldsymbol{l}/\boldsymbol{l}_{0})\cong J\cap G^{\tau}/J^{1}\cap G^{\tau}$ whose inflation equals $\chi$.  Since it is well-known that the derived subgroup of $\mathrm{U}_{m}(\boldsymbol{l}/\boldsymbol{l}_{0})$ is $\mathrm{SU}_{m}(\boldsymbol{l}/\boldsymbol{l}_{0}):=\{g\in\mathrm{U}_{m}(\boldsymbol{l}/\boldsymbol{l}_{0})|\mathrm{det}(g)=1\}$ (see \cite{dieudonne1963geometrie}, II. \S 5), there exists $\overline{\phi}$ as a quadratic character of $\mathrm{det}(\mathrm{U}_{m}(\boldsymbol{l}/\boldsymbol{l}_{0}))=\{x\in\boldsymbol{l}^{\times}|x\sigma(x)=x^{Q+1}=1\}$, such that $\overline{\chi}=\overline{\phi}\circ\mathrm{det}|_{\mathrm{U}_{m}(\boldsymbol{l}/\boldsymbol{l}_{0})}$. We extend $\overline{\phi}$ to a character of $\boldsymbol{l}^{\times}$ and we write $\overline{\chi}'=\overline{\phi}\circ\mathrm{det}$ which is a character of $\mathrm{GL}_{m}(\boldsymbol{l})$ extending $\overline{\chi}$. Write $\chi'^{0}$ the inflation of $\overline{\chi}'$ with respect to the isomorphism $\mathrm{GL}_{m}(\boldsymbol{l})\cong J/J^{1}$. Finally we choose $\chi'$ to be a character of $\boldsymbol{J}$ extending $\chi'^{0}$ by choosing $\chi'(\varpi_{E})\neq 0$ randomly. By construction, $\chi'|_{\boldsymbol{J}\cap G^{\tau}}=\chi$.
			
		\end{proof}
		
		\begin{proposition}\label{Propchi=1}
			
			(1) When $\mathrm{char}(R)=0$, for any $\chi'$ extending $\chi$ to $\bs{J}$ we have $\chi'(\chi'\circ\tau)=1$.
			
			(2) Furthermore, for any $R$ we have $\chi=1$.
			
		\end{proposition}
		
		\begin{proof}
			
			First we consider $\mathrm{char}(R)=0$. Since $m$ is odd, Lemma \ref{Lemmafinsiginv} implies that $\mathrm{GL}_{m}(\boldsymbol{l})$ possesses a $\sigma$-invariant supercuspidal representation $\overline{\rho'}$. Using Proposition \ref{PropGelKafin}, we get $\overline{\rho'}^{\tau\vee}\cong\overline{\rho'}$. We denote by $\boldsymbol{\rho}'$ a representation of $\boldsymbol{J}$ trivial on $J^{1}$ such that its restriction to $J$ is the inflation of $\overline{\rho'}$. Since $\sigma(\varpi_{E})=\varpi_{E}$, we have $\boldsymbol{\rho}'(\tau(\varpi_{E}))=\boldsymbol{\rho}'(\varpi_{E})^{-1}$ which means that $\boldsymbol{\rho}'$ is $\tau$-selfdual. By Lemma \ref{Lemmafinunidist}, it is also distinguished.
			
			Let $\Lambda'$ denote the $\tau$-selfdual simple type $\boldsymbol{\kappa}\otimes\boldsymbol{\rho}'$. The natural isomorphism
			$$\mathrm{Hom}_{\boldsymbol{J}\cap G^{\tau}}(\Lambda',\chi^{-1})\cong\mathrm{Hom}_{\boldsymbol{J}\cap G^{\tau}}(\boldsymbol{\kappa},\chi^{-1})\otimes_{R}\mathrm{Hom}_{\boldsymbol{J}\cap G^{\tau}}(\boldsymbol{\rho}',1)$$
			shows that $\Lambda'$ is $\chi^{-1}$-distinguished.
			
			By Lemma \ref{Lemmachiext}, there exists a character $\chi'$ extending $\chi$. The representation $\Lambda''=\Lambda'\chi'$ is thus a distinguished simple type. Let $\pi''$ be the supercuspidal representation of $G$ compactly induced by $(\boldsymbol{J},\Lambda'')$. It is distinguished, thus $\tau$-selfdual by Theorem \ref{Thmdistgalinv} and Proposition \ref{PropGelKa}. Since $\Lambda''$ and $\Lambda''^{\tau\vee}\cong\Lambda''\chi'^{-1}(\chi'^{-1}\circ\tau)$ are both contained in $\pi''$, it follows that $\chi'(\chi'\circ\tau)$ is trivial.
			
			We write $\overline{\chi}=\overline{\phi}\circ\mathrm{det}$ as in the proof of Lemma \ref{Lemmachiext}. Since $\chi'(\chi'\circ\tau)=1$, we get $\overline{\phi}(\overline{\phi\circ\sigma})^{-1}=\overline{\phi}^{1-Q}=1$. Choose $\zeta_{\boldsymbol{l}}$ to be a primitive root of $\boldsymbol{l}^{\times}$, then $\zeta_{\boldsymbol{l}}^{Q-1}$ generates the group $\mathrm{det}(\mathrm{U}_{m}(\boldsymbol{l}/\boldsymbol{l}_{0}))=\{x\in\boldsymbol{l}^{\times}|x\sigma(x)=x^{Q+1}=1\}$. Since $\overline{\phi}(\zeta_{\boldsymbol{l}}^{1-Q})=1$, we deduce that $\overline{\phi}|_{\mathrm{det}(\mathrm{U}_{m}(\boldsymbol{l}/\boldsymbol{l}_{0}))}$ is trivial, which means that $\overline{\chi}$ is trivial. Thus $\chi$ as the inflation of $\overline{\chi}$ is also trivial.
			
			Now we consider $R=\overline{\mathbb{F}_{l}}$. As already mentioned in the proof of Proposition \ref{propchiquad}, if we denote by $\widetilde{\boldsymbol{\kappa}}$ the $\overline{\mathbb{Q}_{l}}$-lift of $\boldsymbol{\kappa}$ and if we denote by $\widetilde{\chi}$ the character defined by Lemma \ref{lemkappa}.(1) with respect to $\widetilde{\boldsymbol{\kappa}}$ and $\widetilde{\eta}$, then $\widetilde{\chi}$ is a $\overline{\mathbb{Q}_{l}}$-lift of $\chi$. Using the characteristic 0 case we already proved, we get $\widetilde{\chi}=1$ which implies that $\chi=1$.
			
			When $\mathrm{R}=l>0$ in general, we follow the same logic as in the proof of Proposition \ref{propchiquad}.
			
		\end{proof}
		
		\begin{remark}
			
			In fact in Proposition \ref{Propchi=1}, we proved that when $m$ is odd and $E/E_{0}$ is unramified, any $\tau$-selfdual $\boldsymbol{\kappa}$ constructed in Proposition \ref{propselfdual} as an extension of a $J^{1}\cap G^{\tau}$-distinguished Heisenberg representation $\eta$ is $\boldsymbol{J}\cap G^{\tau}$-distinguished.
			
		\end{remark}
		
		Now we come back to the proof of our main theorem. We have
		$$\mathrm{Hom}_{\boldsymbol{J}\cap G^{\tau}}(\Lambda,1)\cong\mathrm{Hom}_{\boldsymbol{J}\cap G^{\tau}}(\boldsymbol{\kappa},1)\otimes_{R}\mathrm{Hom}_{\boldsymbol{J}\cap G^{\tau}}(\boldsymbol{\rho},1),$$
		where $\mathrm{Hom}_{\boldsymbol{J}\cap G^{\tau}}(\boldsymbol{\kappa},1)$ is of dimension 1, and $\mathrm{Hom}_{\boldsymbol{J}\cap G^{\tau}}(\boldsymbol{\rho},1)\cong\mathrm{Hom}_{\mathrm{U}_{m}(\boldsymbol{l}/\boldsymbol{l}_{0})}(\overline{\rho},1)$ is also of dimension 1 by Lemma \ref{LemmaunitaryJJ0}, Lemma \ref{Lemmafinunidist} and Proposition \ref{PropGelKafin}. So $\mathrm{Hom}_{\boldsymbol{J}\cap G^{\tau}}(\Lambda,1)$ is of dimension 1, which implies that $\pi$ is $G^{\tau}$-distinguished. Thus we finish the proof of Theorem \ref{Thmmain} when $E/E_{0}$ is unramified. Using Corollary \ref{corJGtau} and the fact that $m$ is odd, we deduce that $\mathrm{Hom}_{G^{\tau}}(\pi,1)$ is of dimension 1, which finishes the proof of Theorem \ref{Thmmult1} when $E/E_{0}$ is unramified.

		\section{The supercuspidal ramified case}
		
		In this section, we study the distinction of $\sigma$-invariant supercuspidal representations of $G$ in the case where $E/E_{0}$ is ramified. This finishes the proof of our main theorem.
		
		\subsection{The finite field case}\label{subsectionfiniteram}
		
		Let $\boldsymbol{l}$ be a finite field of characteristic $p\neq 2$ and let $|\boldsymbol{l}|=Q$. For $m$ a positive integer, we denote by $\mathbf{G}$ the reductive group $\mathrm{GL}_{m}$ over $\boldsymbol{l}$. Thus by definition, $\mathbf{G}(\boldsymbol{l})=\mathrm{GL}_{m}(\boldsymbol{l})$. For $\overline{\varepsilon}$ a matrix in $\mathbf{G}(\boldsymbol{l})$ such that $\,^{t}\overline{\varepsilon}=\overline{\varepsilon}$, the automorphism defined by $\tau(x)=\overline{\varepsilon}\,^{t}x^{-1}\overline{\varepsilon}^{-1}$ for any $x\in\mathrm{GL}_{m}(\boldsymbol{l})$ gives an involution on $\mathrm{GL}_{m}(\boldsymbol{l})$, which induces an involution on $\mathbf{G}$. Thus $\mathbf{G}^{\tau}$ is the orthogonal group corresponding to $\tau$, which is a reductive group over $\boldsymbol{l}$, and  $\mathbf{G}^{\tau}(\boldsymbol{l})=\mathrm{GL}_{m}(\boldsymbol{l})^{\tau}$ which is a subgroup of $\mathrm{GL}_{m}(\boldsymbol{l})$. In this subsection, for $\overline{\rho}$ a supercuspidal representation of $\mathrm{GL}_{m}(\boldsymbol{l})$ and $\overline{\chi}$ a character of $\mathrm{GL}_{m}(\boldsymbol{l})^{\tau}$, we state the result mentioned in \cite{hakim2012distinguished} which gives a criterion for $\overline{\rho}$ distinguished by $\overline{\chi}$.
		
		First of all, we assume $R=\overline{\mathbb{Q}_{l}}$. We recall a little bit of Deligne-Lusztig theory (see \cite{deligne1976representations}). Let $\mathbf{T}$ be an elliptic maximal $\boldsymbol{l}$-torus in $\mathbf{G}$, where ellipticity means that $\mathbf{T}(\boldsymbol{l})=\boldsymbol{t}^{\times}$ and $\boldsymbol{t}/\boldsymbol{l}$ is the field extension of degree $m$. Let $\xi$ be a regular character of $\mathbf{T}(\boldsymbol{l})$, where regularity means the same as in the construction of Green and James in \S 7.1. Using \cite{deligne1976representations}, Theorem 8.3, there is a virtual character $R_{\mathbf{T},\xi}$ as the character of a cuspidal representation of $\mathrm{GL}_{m}(\boldsymbol{l})$. Moreover if we fix $\mathbf{T}$, we know that $\xi\mapsto R_{\mathbf{T},\xi}$ gives a bijection from the set of Galois orbits of regular characters of $\mathbf{T}$ to the set of cuspidal representations of $\mathrm{GL}_{m}(\boldsymbol{l})$. So we may choose $\xi$ such that $\mathrm{Trace}(\overline{\rho})=R_{\mathbf{T},\xi}$. Moreover, using \cite{deligne1976representations}, Theorem 4.2, we get $R_{\mathbf{T},\xi}(-1)=\mathrm{dim}(\overline{\rho})\xi(-1)$ with $\mathrm{dim}(\overline{\rho})=(Q-1)(Q^{2}-1)...(Q^{m-1}-1)$. So if we denote by $\omega_{\overline{\rho}}$ the central character of $\overline{\rho}$, we get $\omega_{\overline{\rho}}(-1)=\xi(-1)$. 
		
		\begin{proposition}[\cite{hakim2012distinguished}, Proposition 6.7]\label{Proporthdist}
			
			For $\tau$, $\overline{\rho}$, $\boldsymbol{T}$ and $\xi$ above, we have:
			
			$$\mathrm{dim}_{R}(\mathrm{Hom}_{\mathbf{G}^{\tau}(\boldsymbol{l})}(\overline{\rho},\overline{\chi}))=
			\begin{cases} 1 &\text{if}\ \omega_{\overline{\rho}}(-1)=\xi(-1)=\overline{\chi}(-1),\\
				0&\text{otherwise}.
			\end{cases}$$
			
		\end{proposition}
		
		Now we consider the $l$-modular case and assume $\mathrm{char}(R)=l>0$.
		
		\begin{proposition}\label{Proporthdistl}
			
			For $\tau$ above and $\overline{\rho}$ a supercuspidal representation of $\mathrm{GL}_{m}(\boldsymbol{l})$ over $R$, the space $\mathrm{Hom}_{\mathrm{GL}_{m}(\boldsymbol{l})^{\tau}}(\overline{\rho},\overline{\chi})\neq 0$ if and only if $\omega_{\overline{\rho}}(-1)=\overline{\chi}(-1)$. Moreover if the condition is satisfied, then we have $\mathrm{dim}_{R}(\mathrm{Hom}_{\mathrm{GL}_{m}(\boldsymbol{l})^{\tau}}(\overline{\rho},\overline{\chi}))=1.$
			
		\end{proposition}
		
		\begin{proof}
			
			First we assume $R=\overline{\mathbb{F}_{l}}$. We use a similar proof to that in Lemma \ref{Lemmafinunidist}. Let $H=\mathrm{GL}_{m}(\boldsymbol{l})^{\tau}$. We choose $\widetilde{\overline{\chi}}$ to be a character of $H$ lifting $\overline{\chi}$, which is defined over $\overline{\mathbb{Z}_{l}}$ or $\overline{\mathbb{Q}_{l}}$ by abuse of notation. For $S=\overline{\mathbb{Z}_{l}}, \overline{\mathbb{Q}_{l}}$, we define $$S[H\backslash\mathrm{GL}_{m}(\boldsymbol{l})]_{\widetilde{\overline{\chi}}}:=\{f|f:\mathrm{GL}_{m}(\boldsymbol{l})\rightarrow S,\ f(hg)=\widetilde{\overline{\chi}}(h)f(g)\ \text{for any}\ h\in H, g\in \mathrm{GL}_{m}(\boldsymbol{l})\}.$$
			Especially $$\overline{\mathbb{Q}_{l}}[H\backslash\mathrm{GL}_{m}(\boldsymbol{l})]_{\widetilde{\overline{\chi}}}=\mathrm{Ind}^{\mathrm{GL}_{m}(\boldsymbol{l})}_{H}\widetilde{\overline{\chi}}$$ as a representation of $\mathrm{GL}_{m}(\boldsymbol{l})$ over $\overline{\mathbb{Q}_{l}}$, and $\overline{\mathbb{Z}_{l}}[H\backslash\mathrm{GL}_{m}(\boldsymbol{l})]_{\widetilde{\overline{\chi}}}$ is a free $\overline{\mathbb{Z}_{l}}$-module. If we further define $$\overline{\mathbb{F}_{l}}[H\backslash\mathrm{GL}_{m}(\boldsymbol{l})]_{\overline{\chi}}=\mathrm{Ind}^{\mathrm{GL}_{m}(\boldsymbol{l})}_{H}\overline{\chi},$$ then we have
			$$\overline{\mathbb{Z}_{l}}[H\backslash\mathrm{GL}_{m}(\boldsymbol{l})]_{\widetilde{\overline{\chi}}}\otimes_{\overline{\mathbb{Z}_{l}}}\overline{\mathbb{F}_{l}}=\overline{\mathbb{F}_{l}}[H\backslash\mathrm{GL}_{m}(\boldsymbol{l})]_{\overline{\chi}}$$
			and
			$$\overline{\mathbb{Z}_{l}}[H\backslash\mathrm{GL}_{m}(\boldsymbol{l})]_{\widetilde{\overline{\chi}}}\otimes_{\overline{\mathbb{Z}_{l}}}\overline{\mathbb{Q}_{l}}=\overline{\mathbb{Q}_{l}}[H\backslash\mathrm{GL}_{m}(\boldsymbol{l})]_{\widetilde{\overline{\chi}}}.$$
			
			We deduce that
			\begin{align*}
				\mathrm{Hom}_{H}(\overline{\rho},\overline{\chi})\neq 0; 
				&\Longleftrightarrow \text{There exists}\ \widetilde{\overline{\rho}}\ \text{lifting}\ \overline{\rho}\  \text{such that}\ \mathrm{Hom}_{\overline{\mathbb{Q}_{l}}[\mathrm{GL}_{m}(\boldsymbol{l})]}(\widetilde{\overline{\rho}},\overline{\mathbb{Q}_{l}}[H\backslash\mathrm{GL}_{m}(\boldsymbol{l})]_{\widetilde{\overline{\chi}}})\neq 0;\\
				&\Longleftrightarrow \text{There exists}\ \widetilde{\overline{\rho}}\ \text{lifting}\ \overline{\rho}\  \text{such that}\ \omega_{\widetilde{\overline{\rho}}}(-1)=\widetilde{\overline{\chi}}(-1);\\
				&\Longleftrightarrow \omega_{\overline{\rho}}(-1)=\overline{\chi}(-1).
			\end{align*}
			The first equivalence is of the same reason as in the proof of Lemma \ref{Lemmafinunidist}, and we use Proposition \ref{Proporthdist} for the second equivalence. For the last equivalence, the ``$\Rightarrow$" direction is trivial. For the other direction, when $l\neq 2$, we choose $\widetilde{\overline{\rho}}$ to be any supercuspidal $\overline{\mathbb{Q}_{l}}$-lift of $\overline{\rho}$. Thus we have $\omega_{\widetilde{\overline{\rho}}}(-1)=\omega_{\overline{\rho}}(-1)=\overline{\chi}(-1)=\widetilde{\overline{\chi}}(-1).$ When $l=2$, using the construction of Green and James, for $\xi$ a regular character over $\overline{\mathbb{F}_{l}}$ corresponding to $\overline{\rho}$, we may always find a $\overline{\mathbb{Q}_{l}}$-lift $\widetilde{\xi}$ which is regular and satisfies $\widetilde{\xi}(-1)=\widetilde{\overline{\chi}}(-1)$. Thus the supercuspidal representation $\widetilde{\overline{\rho}}$ corresponding to $\widetilde{\xi}$ as a lift of $\overline{\rho}$ satisfies $\omega_{\widetilde{\overline{\rho}}}(-1)=\widetilde{\overline{\chi}}(-1)$. So we finish the proof of the first part.
			
			To calculate the dimension, as in the proof of Lemma \ref{Lemmafinunidist} if we write $$\overline{\mathbb{F}_{l}}[H\backslash\mathrm{GL}_{m}(\boldsymbol{l})]_{\overline{\chi}}=V_{\overline{\rho}}\oplus V',$$
			where $V_{\overline{\rho}}$ is composed of irreducible components isomorphic to $\overline{\rho}$, and $V'$ has no irreducible component isomorphic to $\overline{\rho}$, then we only need to show that $\mathrm{End}_{\overline{\mathbb{F}_{l}}[\mathrm{GL}_{m}(\boldsymbol{l})]}(V_{\overline{\rho}})$ is commutative. We consider the following $\overline{\mathbb{Z}_{l}}[\mathrm{GL}_{m}(\boldsymbol{l})]$-module decomposition
			$$\overline{\mathbb{Z}_{l}}[H\backslash\mathrm{GL}_{m}(\boldsymbol{l})]_{\widetilde{\overline{\chi}}}=\widetilde{V_{\overline{\rho}}}\oplus \widetilde{V'},$$
			where $\widetilde{V_{\overline{\rho}}}\otimes_{\overline{\mathbb{Z}_{l}}}\overline{\mathbb{Q}_{l}}=\bigoplus_{\widetilde{\overline{\rho}}}\widetilde{\overline{\rho}}$ with the direct sum ranges over all the irreducible representations $\widetilde{\overline{\rho}}$ over $\overline{\mathbb{Q}_{l}}$ occurring in $\widetilde{P_{\overline{\rho}}}$ counting the multiplicity, and $\widetilde{V'}$ denotes a $\overline{\mathbb{Z}_{l}}[\mathrm{GL}_{m}(\boldsymbol{l})]$-complement of $\widetilde{V_{\overline{\rho}}}$, such that $\widetilde{V'}\otimes_{\overline{\mathbb{Z}_{l}}}\overline{\mathbb{Q}_{l}}$ contains no irreducible component of $\widetilde{\overline{\rho}}$. Using Proposition \ref{Proporthdist}, $\widetilde{V_{\overline{\rho}}}\otimes_{\overline{\mathbb{Z}_{l}}}\overline{\mathbb{Q}_{l}}$ is multiplicity free, which means that $\mathrm{End}_{\overline{\mathbb{Q}_{l}}[\mathrm{GL}_{m}(\boldsymbol{l})]}(\widetilde{V_{\overline{\rho}}}\otimes_{\overline{\mathbb{Z}_{l}}}\overline{\mathbb{Q}_{l}})$ is commutative. The canonical embedding from $\overline{\mathbb{Z}_{l}}[H\backslash\mathrm{GL}_{m}(\boldsymbol{l})]_{\widetilde{\overline{\chi}}}$ to $\overline{\mathbb{Q}_{l}}[H\backslash\mathrm{GL}_{m}(\boldsymbol{l})]_{\widetilde{\overline{\chi}}}$ induces the following ring monomorphism
			$$\mathrm{End}_{\overline{\mathbb{Z}_{l}}[\mathrm{GL}_{m}(\boldsymbol{l})]}(\overline{\mathbb{Z}_{l}}[H\backslash\mathrm{GL}_{m}(\boldsymbol{l})]_{\widetilde{\overline{\chi}}})\hookrightarrow\mathrm{End}_{\overline{\mathbb{Q}_{l}}[\mathrm{GL}_{m}(\boldsymbol{l})]}(\overline{\mathbb{Q}_{l}}[H\backslash\mathrm{GL}_{m}(\boldsymbol{l})]_{\overline{\chi}})$$
			given by tensoring $\overline{\mathbb{Q}_{l}}$, which leads to the ring monomorphism
			$$\mathrm{End}_{\overline{\mathbb{Z}_{l}}[\mathrm{GL}_{m}(\boldsymbol{l})]}(\widetilde{V_{\overline{\rho}}})\hookrightarrow\mathrm{End}_{\overline{\mathbb{Q}_{l}}[\mathrm{GL}_{m}(\boldsymbol{l})]}(\widetilde{V_{\overline{\rho}}}\otimes_{\overline{\mathbb{Z}_{l}}}\overline{\mathbb{Q}_{l}}).$$
			Thus $\mathrm{End}_{\overline{\mathbb{Z}_{l}}[\mathrm{GL}_{m}(\boldsymbol{l})]}(\widetilde{V_{\overline{\rho}}})$ is also commutative.
			
			The modulo $l$ map from $\overline{\mathbb{Z}_{l}}[H\backslash\mathrm{GL}_{m}(\boldsymbol{l})]_{\widetilde{\overline{\chi}}}$ to $\overline{\mathbb{F}_{l}}[H\backslash\mathrm{GL}_{m}(\boldsymbol{l})]_{\overline{\chi}}$
			induces the following ring epimorphism
			$$\mathrm{End}_{\overline{\mathbb{Z}_{l}}[\mathrm{GL}_{m}(\boldsymbol{l})]}(\overline{\mathbb{Z}_{l}}[H\backslash\mathrm{GL}_{m}(\boldsymbol{l})]_{\widetilde{\overline{\chi}}})\twoheadrightarrow\mathrm{End}_{\overline{\mathbb{F}_{l}}[\mathrm{GL}_{m}(\boldsymbol{l})]}(\overline{\mathbb{F}_{l}}[H\backslash\mathrm{GL}_{m}(\boldsymbol{l})]_{\overline{\chi}}),$$ which leads to the ring epimorphism
			$$\mathrm{End}_{\overline{\mathbb{Z}_{l}}[\mathrm{GL}_{m}(\boldsymbol{l})]}(\widetilde{V_{\overline{\rho}}})\twoheadrightarrow\mathrm{End}_{\overline{\mathbb{F}_{l}}[\mathrm{GL}_{m}(\boldsymbol{l})]}(V_{\overline{\rho}}).$$
			Since $\mathrm{End}_{\overline{\mathbb{Z}_{l}}[\mathrm{GL}_{m}(\boldsymbol{l})]}(\widetilde{V_{\overline{\rho}}})$ is commutative, $\mathrm{End}_{\overline{\mathbb{F}_{l}}[\mathrm{GL}_{m}(\boldsymbol{l})]}(V_{\overline{\rho}})$ is also commutative. Thus we may use the same proof as in Lemma \ref{Lemmafinunidist} to show that
			$$\mathrm{dim}_{\overline{\mathbb{F}_{l}}}(\mathrm{Hom}_{\mathrm{GL}_{m}(\boldsymbol{l})^{\tau}}(\overline{\rho},\overline{\chi}))=1.$$
			
			Finally for $\mathrm{char}(R)=l>0$ in general, we follow the corresponding proof in Lemma \ref{Lemmafinunidist}.
			
		\end{proof}
		
		\begin{remark}
			
			For $\mathbf{G}^{\tau}(\boldsymbol{l})$ an orthogonal group with $m\geq 2$, it is  well-known that its derived group is always a subgroup of $\mathbf{G}^{\tau0}(\boldsymbol{l})$ of index 2 (see \cite{dieudonne1963geometrie}, II. \S 8)
			, which means that there exists a character of $\mathbf{G}^{\tau}(\boldsymbol{l})$ which isn't trivial on $\mathbf{G}^{\tau0}(\boldsymbol{l})$. It means that we cannot expect $\overline{\chi}$ to be trivial on $\mathbf{G}^{\tau0}(\boldsymbol{l})$ in general. However, for those $\overline{\chi}$ occurring in the next subsection, it is highly possible that $\overline{\chi}$ is trivial on $\mathbf{G}^{\tau0}(\boldsymbol{l})$. For example, \cite{hakim2012distinguished}, Proposition 6.4 gives evidence for this in the case where $\pi$ is tame supercuspidal. However, I don't know how to prove it.
			
		\end{remark}
		
		Now we assume that $m$ is even. We write $J_{m/2}=\bigg(\begin{matrix}0 & I_{m/2} \\ -I_{m/2} & 0\end{matrix}\bigg)$ and we denote by $$\mathrm{Sp}_{m}(\boldsymbol{l})=\{x\in\mathrm{GL}_{m}(\boldsymbol{l})|\,^{t}xJ_{m/2}x=J_{m/2}\}$$ the symplectic subgroup of $\mathrm{GL}_{m}(\boldsymbol{l})$.
		
		\begin{proposition}\label{Propsympdisc}
			
			For $\overline{\rho}$ a cuspidal representation of $\mathrm{GL}_{m}(\boldsymbol{l})$, we have $\mathrm{Hom}_{\mathrm{Sp}_{m}(\boldsymbol{l})}(\overline{\rho},1)=0$.
			
		\end{proposition}
		
		\begin{proof}

			Using \cite{klyachko1984models}, Corollary 1.4. whose proof also works for the $l$-modular case, we know that an irreducible generic representation cannot be distinguished by a symplectic subgroup. Since a cuspidal representation is generic, we finish the proof.
			
		\end{proof}
		
		\subsection{Distinction criterion in the ramified case}\label{subsectiondisccriram}
		
		Still let $\pi$ be a $\sigma$-invariant supercuspidal representation of $G$. In this subsection we prove Theorem \ref{Thmmain} and Theorem \ref{Thmmult1} in the case where $E/E_{0}$ is ramified. Using Theorem \ref{Thmdistgalinv}, we only need to show that $\pi$ is distinguished by any unitary subgroup $G^{\tau}$ to finish the proof of Theorem \ref{Thmmain}. We may change $\tau$ up to a $G$-action which doesn't change the property of being distinguished. Thus using Remark \ref{rempicond}.(4), we may assume $\tau(x)=\varepsilon\sigma(\,^{t}x^{-1})\varepsilon^{-1}$ for any $x\in G$, where $\varepsilon$ equals $I_{n}$ or $\mathrm{diag}(I_{d},...,I_{d},\epsilon)$ with $\epsilon\in \mathfrak{o}_{E_{0}}^{\times}-\mathrm{N}_{E/E_{0}}(\mathfrak{o}_{E}^{\times})$, representing the two classes of unitary involutions. We denote by $\overline{\varepsilon}$ the image of $\varepsilon$ in $\mathrm{GL}_{m}(\boldsymbol{l})$.
		
		For $(\boldsymbol{J},\Lambda)$ a simple type in Remark \ref{rempicond}.(2), we write $\Lambda\cong\boldsymbol{\kappa}\otimes\boldsymbol{\rho}$. Using Proposition \ref{propselfdual}, we may further assume $\boldsymbol{\kappa}^{\tau\vee}\cong\boldsymbol{\kappa}$. Using Lemma \ref{lemkappa} with $g=1$, there exists a quadratic character $\chi:\boldsymbol{J}\cap G^{\tau}\rightarrow R^{\times}$ such that
		\begin{equation}\label{eqkappa1}
			\mathrm{dim}_{R}\mathrm{Hom}_{\boldsymbol{J}\cap G^{\tau}}(\boldsymbol{\kappa},\chi^{-1})=1
		\end{equation}
		and
		\begin{equation}\label{eqlamkaprho}
			\mathrm{Hom}_{\boldsymbol{J}\cap G^{\tau}}(\Lambda,1)\cong\mathrm{Hom}_{\boldsymbol{J}\cap G^{\tau}}(\boldsymbol{\kappa},\chi^{-1})\otimes_{R}\mathrm{Hom}_{\boldsymbol{J}\cap G^{\tau}}(\boldsymbol{\rho},\chi).
		\end{equation}
		If we denote by $\omega_{\boldsymbol{\kappa}}$ the central character of $\boldsymbol{\kappa}$ defined on $F^{\times}$, using (\ref{eqkappa1}), we get $\omega_{\boldsymbol{\kappa}}=\chi^{-1}$ as characters of $F^{\times}\cap(\boldsymbol{J}\cap G^{\tau})$. In particular, $\omega_{\boldsymbol{\kappa}}(-1)=\chi^{-1}(-1)$. Since $\boldsymbol{\kappa}^{\tau\vee}\cong\boldsymbol{\kappa}$, we get $\omega_{\boldsymbol{\kappa}}\circ\tau=\omega_{\boldsymbol{\kappa}}^{-1}$. In particular we have $$\omega_{\boldsymbol{\kappa}}(\varpi_{F})^{-1}=\omega_{\boldsymbol{\kappa}}(\tau(\varpi_{F}))=\omega_{\boldsymbol{\kappa}}(\varpi_{F})^{-1}\omega_{\boldsymbol{\kappa}}(-1)^{-1},$$ where we use the fact that $\sigma(\varpi_{F})=-\varpi_{F}$. Thus we get $\omega_{\boldsymbol{\kappa}}(-1)=\chi(-1)=1$.
		
		Since $\Lambda$ and $\boldsymbol{\kappa}$ are $\tau$-selfdual, $\boldsymbol{\rho}$ is $\tau$-selfdual. Using the same proof as that for $\boldsymbol{\kappa}$, we get $\omega_{\boldsymbol{\rho}}(-1)=1$. Let $\overline{\rho}$ be the supercuspidal representation of $\mathrm{GL_{m}}(\boldsymbol{l})\cong J/J^{1}$ whose inflation equals $\boldsymbol{\rho}|_{J}$ and let $\overline{\chi}$ be the character of
		$$\mathbf{G}^{\tau}(\boldsymbol{l})\cong J\cap G^{\tau}/J^{1}\cap G^{\tau}$$
		whose inflation equals $\chi$, where $\tau$ naturally induces an orthogonal involution on $\mathbf{G}$ with respect to a symmetric matrix $\overline{\varepsilon}\in\mrgl_{m}(\bs{l})$. By definition and Lemma \ref{LemmaunitaryJJ0} we get $$\mathrm{Hom}_{\boldsymbol{J}\cap G^{\tau}}(\boldsymbol{\rho},\chi)\cong\mathrm{Hom}_{\mathbf{G}^{\tau}(\boldsymbol{l})}(\overline{\rho},\overline{\chi}).$$
		Since $\omega_{\overline{\rho}}(-1)=\overline{\chi}(-1)=1$, using Proposition \ref{Proporthdist} and Proposition \ref{Proporthdistl} the space above is non-zero. Thus by (\ref{eqlamkaprho}) we have $$\mathrm{Hom}_{\boldsymbol{J}\cap G^{\tau}}(\Lambda,1)\neq 0$$ which means that $\pi$ is distinguished by $G^{\tau}$,
		finishing the proof of Theorem \ref{Thmmain}. Moreover using Proposition \ref{Proporthdist}, Proposition \ref{Proporthdistl}, (\ref{eqkappa1}) and (\ref{eqlamkaprho}), we get $$\mathrm{dim}_{R}\mathrm{Hom}_{\boldsymbol{J}\cap G^{\tau}}(\Lambda,1)=1.$$
		
		Now if $m$ is even and $\varepsilon=I_{m}$, we also need to study the space $\mathrm{Hom}_{\boldsymbol{J}^{g_{1}}\cap G^{\tau}}(\Lambda^{g_{1}},1)$, where $g_{1}$ is defined in Corollary \ref{corJGtau} such that $\tau(g_{1})g_{1}^{-1}=\varpi_{E}J_{m/2}\in B^{\times}$. Using Lemma \ref{lemkappa}, there exists a quadratic character $\chi_{1}:\boldsymbol{J}^{g_{1}}\cap G^{\tau}\rightarrow R^{\times}$ such that
		\begin{equation}\label{eqkappag1}
			\mathrm{dim}_{R}\mathrm{Hom}_{\boldsymbol{J}^{g_{1}}\cap G^{\tau}}(\boldsymbol{\kappa}^{g_{1}},\chi_{1}^{-1})=1.
		\end{equation}
		and
		\begin{equation}\label{eqlamkaprhog1}
			\mathrm{Hom}_{\boldsymbol{J}^{g_{1}}\cap G^{\tau}}(\Lambda^{g_{1}},1)\cong\mathrm{Hom}_{\boldsymbol{J}^{g_{1}}\cap G^{\tau}}(\boldsymbol{\kappa}^{g_{1}},\chi_{1}^{-1})\otimes_{R}\mathrm{Hom}_{\boldsymbol{J}^{g_{1}}\cap G^{\tau}}(\boldsymbol{\rho}^{g_{1}},\chi_{1}).
		\end{equation}
		So we only need to study the space $\mathrm{Hom}_{\boldsymbol{J}^{g_{1}}\cap G^{\tau}}(\boldsymbol{\rho}^{g_{1}},\chi_{1})\cong\mathrm{Hom}_{\boldsymbol{J}\cap G^{\delta_{g_{1}}}}(\boldsymbol{\rho},\chi_{1}^{g_{1}^{-1}})$, where $$\delta_{g_{1}}(x):=(\tau(g_{1})g_{1}^{-1})^{-1}\tau(x)(\tau(g_{1})g_{1}^{-1})=(\varpi_{E}J_{m/2})^{-1}\tau(x)\varpi_{E}J_{m/2}$$ for any $x\in G$ as an involution on $G$.
		
		Let $\overline{\rho}$ be the supercuspidal representation of $\mathrm{GL_{m}}(\boldsymbol{l})\cong J/J^{1}$ whose inflation equals $\boldsymbol{\rho}|_{J}$ and let $\overline{\chi_{1}^{g_{1}^{-1}}}$ be the character of $$\mathrm{Sp}_{m}(\boldsymbol{l})\cong J\cap G^{\delta_{g_{1}}}/J^{1}\cap G^{\delta_{g_{1}}}$$
		whose inflation equals $\chi_{1}^{g_{1}^{-1}}$, then we get
		$$\mathrm{Hom}_{\boldsymbol{J}\cap G^{\delta_{g_{1}}}}(\boldsymbol{\rho},\chi_{1}^{g_{1}^{-1}})\cong\mathrm{Hom}_{\mathrm{Sp}_{m}(\boldsymbol{l})}(\overline{\rho},\overline{\chi_{1}^{g_{1}^{-1}}})=\mathrm{Hom}_{\mathrm{Sp}_{m}(\boldsymbol{l})}(\overline{\rho},1),$$
		where the last equation is because of the well-known fact that $\mathrm{Sp}_{m}(\boldsymbol{l})$ equals its derived group (\cite{dieudonne1963geometrie}, II. \S 8), thus $\overline{\chi_{1}^{g_{1}^{-1}}}|_{\mathrm{Sp}_{m}(\boldsymbol{l})}$ is trivial. Using Proposition \ref{Propsympdisc}, we get $\mathrm{Hom}_{\mathrm{Sp}_{m}(\boldsymbol{l})}(\overline{\rho},1)=0$. Thus $\mathrm{Hom}_{\boldsymbol{J}^{g_{1}}\cap G^{\tau}}(\Lambda^{g_{1}},1)=0$.
		
		Using Corollary \ref{corJGtau}, we get $$\mathrm{dim}_{R}\mathrm{Hom}_{G^{\tau}}(\pi,1)=\mathrm{dim}_{R}\mathrm{Hom}_{\boldsymbol{J}\cap G^{\tau}}(\Lambda,1)+\mathrm{dim}_{R}\mathrm{Hom}_{\boldsymbol{J}^{g_{1}}\cap G^{\tau}}(\Lambda^{g_{1}},1)=1,$$ which finishes the proof of Theorem \ref{Thmmult1} when $E/E_{0}$ is ramified.
		
		
		
		
		\subsection{Proof of Theorem \ref{Thmlift}}
		
		We finish the proof of Theorem \ref{Thmlift}. Let $\pi$ be a $\sigma$-invariant supercuspidal representation of $G$ over $\overline{\mathbb{F}_{l}}$. For $\tau$ a unitary involution, by Theorem \ref{Thmmain}, $\pi$ is distinguished by $G^{\tau}$. From the proof of Theorem \ref{Thmdistgalinv}, there exists a distinguished integral $\sigma$-invariant supercuspidal representation $\widetilde{\pi}$ of $G$ over $\overline{\mathbb{Q}_{l}}$ which lifts $\pi$.
		
		\section{A purely local proof of Theorem \ref{Thmdistgalinv}}
		
		In this section, we generalize Theorem \ref{Thmdistgalinv} to irreducible cuspidal representations, meanwhile we also give another proof of the original theorem which is purely local. Precisely, we prove the following theorem:
		
		\begin{theorem}\label{Thmdistgalinvgen}
			
			Let $\pi$ be an irreducible cuspidal representation of $G$ over $R$. If $\pi$ is distinguished by $G^{\tau}$, then $\pi$ is $\sigma$-invariant.
			
		\end{theorem}
		
		\subsection{The finite analogue}
		
		\begin{proposition}\label{propfindiscinv}
			
			Let $\bs{l}/\bs{l}_{0}$ be a quadratic extension of finite fields of characteristic $p$ and let $\overline{\rho}$ be an irreducible generic representation of $\mrgl_{m}(\bs{l})$ over $R$. If $\overline{\rho}$ is distinguished by the unitary subgroup $H$ of $\mrgl_{m}(\bs{l})$ with respect to $\bs{l}/\bs{l_{0}}$, then it is $\sigma$-invariant.
			
		\end{proposition}
		
		\begin{proof}
			
			When $\mathrm{char}(R)=0$, the proposition was proved by Gow \cite{gow1984two} for any irreducible representations. So we only consider the $l$-modular case and without loss of generality we assume $R=\overline{\mathbb{F}_{l}}$. We write $P_{\overline{\rho}}$ for the projective envelope of $\overline{\rho}$ as a $\overline{\mathbb{Z}_{l}}[\mrgl_{m}(\bs{l})]$-module. Thus $P_{\overline{\rho}}\otimes_{\overline{\mathbb{Z}_{l}}}\overline{\mathbb{F}_{l}}$ is a projective $\overline{\mathbb{F}_{l}}[\mrgl_{m}(\bs{l})]$-module, and moreover
			\begin{align*}
				\mathrm{Hom}_{\overline{\mathbb{F}_{l}}[H]}(\overline{\rho},\bfl)\cong&\mathrm{Hom}_{\overline{\mathbb{F}_{l}}[\mrgl_{m}(\bs{l})]}(\overline{\rho},\bfl[H\backslash\mrgl_{m}(\bs{l})])\neq0
			\end{align*}
			implies that
			\begin{align*}
				\mathrm{Hom}_{\overline{\mathbb{F}_{l}}[\mrgl_{m}(\bs{l})]}(P_{\overline{\rho}}\otimes_{\overline{\mathbb{Z}_{l}}}\overline{\mathbb{F}_{l}},\bfl[H\backslash\mrgl_{m}(\bs{l})])\neq0. \end{align*}
			Using the same argument as that in Lemma \ref{Lemmafinunidist}, we have
			\begin{align*}
				\mathrm{Hom}_{\overline{\mathbb{Q}_{l}}[\mrgl_{m}(\bs{l})]}(P_{\overline{\rho}}\otimes_{\overline{\mathbb{Z}_{l}}}\overline{\mathbb{Q}_{l}},\bql[H\backslash\mrgl_{m}(\bs{l})])\neq0, \end{align*}
			and thus there exists an irreducible constituent $\widetilde{\overline{\rho}}$ of $P_{\overline{\rho}}\otimes_{\overline{\mathbb{Z}_{l}}}\overline{\mathbb{Q}_{l}}$ such that
			\begin{align*}
				\mathrm{Hom}_{\overline{\mathbb{Q}_{l}}[\mrgl_{m}(\bs{l})]}(\widetilde{\overline{\rho}},\bql[H\backslash\mrgl_{m}(\bs{l})])\neq0. 
			\end{align*}
			By \cite{serre2012linear}, \S 14.5, \S 15.4, $\overline{\rho}$ is a constituent of $r_{l}(\widetilde{\overline{\rho}})$. Since $\widetilde{\overline{\rho}}$ is $H$-distinguished, it is $\sigma$-invariant and so is $r_{l}(\widetilde{\overline{\rho}})$. For $i=1,...,k$, we choose $\widetilde{\overline{\rho}}_{i}$ to be a cuspidal representations of $\mrgl_{m_{i}}(\bs{l})$ over $\overline{\mathbb{Q}_{l}}$, such that $\widetilde{\overline{\rho}}$ is a subrepresentation of the parabolic induction $\widetilde{\overline{\rho}}_{1}\times...\times\widetilde{\overline{\rho}}_{k}$, where $m_{1}+...+m_{k}=m$. For each $i$ we write $\overline{\rho}_{i}=r_{l}(\widetilde{\overline{\rho}}_{i})$ which is a cuspidal representation of $\mrgl_{m_{i}}(\bs{l})$ over $\overline{\mathbb{F}_{l}}$, then all the irreducible constituents of $r_{l}(\widetilde{\overline{\rho}})$ are subquotients of $\overline{\rho}_{1}\times...\times\overline{\rho}_{k}$, and in particular so is $\overline{\rho}$. Since $\overline{\rho}$ is generic (or non-degenerate), by \cite{vigneras1996representations}, Chapitre III, 1.10, it is the unique non-degenerate subquotient contained in $\overline{\rho}_{1}\times...\times\overline{\rho}_{k}$, thus it is the unique non-degenerate constituent in $r_{l}(\widetilde{\overline{\rho}})$. Thus it is $\sigma$-invariant.
			
		\end{proof}
		
		\subsection{The cuspidal case}
		
		In this subsection we prove Theorem \ref{Thmdistgalinvgen}. We choose $(\bs{J},\Lambda)$ to be a simple type of $\pi$, then by Frobenius reciprocity and the Mackey formula, there exists $g\in G$ such that
		\begin{equation}\label{eqtypedisc}
			\mathrm{Hom}_{\bs{J}^{g}\cap G^{\tau}}(\Lambda^{g},1)\neq0.
		\end{equation}
		Let $H^{1}$ be the corresponding subgroup of $\bs{J}$, let $\theta$ be the simple character of $H^{1}$ contained in $\Lambda$ and let $\eta$ be the Heisenberg representation of $\theta$.  Restricting (\ref{eqtypedisc}) to $H^{1g}\cap G^{\tau}$ we get $\theta^{g}|_{H^{1g}\cap G^{\tau}}=1$. Following the proof of \cite{secherre2019supercuspidal}, Lemma 6.5, we have
		\begin{equation}\label{eqthetatauint}
			(\theta\circ\tau)^{\tau(g)}|_{\tau(H^{1g})\cap H^{1g}}=\theta^{g}\circ\tau|_{\tau(H^{1g})\cap H^{1g}}=(\theta^{g})^{-1}|_{\tau(H^{1g})\cap H^{1g}},
		\end{equation}
		or in other words, $\theta\circ\tau$ intertwines with $\theta^{-1}$. Using the Intertwining Theorem (\emph{cf.} \cite{bushnell2013intertwining}), $\theta\circ\tau$ and $\theta^{-1}$ are endo-equivalent, which, from the argument of Lemma \ref{Lemmasigmaendo}, is equivalent to $\Theta^{\sigma}=\Theta$, where $\Theta$ denotes the endo-class of $\theta$.
		
		We let $\tau_{1}$ be the unitary involution corresponding to $I_{n}$, which in particular satisfies the condition of Theorem \ref{Thmendotautheta}. Since $\Theta^{\sigma}=\Theta$, by \emph{loc. cit.}, we may choose a simple stratum $[\mfa,\beta]$ and $\theta'\in \mcc(\mfa,\beta)$ with $\theta'\in \Theta$, such that
		$$\tau_{1}(\beta)=\beta^{-1},\quad \tau_{1}(\mfa)=\mfa\quad \text{and}\quad \theta'\circ\tau_{1}=\theta'^{-1}.$$
		Up to $G$-conjugacy, we may and will assume that $\bs{J}=\bs{J}(\mfa,\beta)$ and $\theta'=\theta$. We write $E=F[\beta]$ and $B\cong\mrm_{m}(E)$ for the centralizer of $E$ in $\mrm_{n}(F)$. Using Proposition \ref{propselfdual}, we write $\Lambda=\bs{\kappa}\otimes\bs{\rho}$ with $\bs{\kappa}$ an extension of the Heisenberg representation $\eta$ such that $\bs{\kappa}^{\tau_{1}}\cong\bs{\kappa}^{\vee}$. Let $\varepsilon$ be an hermitian matrix such that $\tau(x)=\varepsilon\sigma(\,^{t}x^{-1})\varepsilon^{-1}=\tau_{1}(x)^{\varepsilon^{-1}}$ for any $x\in G$. For a fixed $g\in G$, we define $\gamma=\varepsilon^{-1}\tau(g)g^{-1}=\tau_{1}(g)\varepsilon^{-1}g^{-1}$ and by direct calculation we have $\tau_{1}(\gamma)=\gamma$.
		
		\begin{proposition}
			
			Let $g\in G$ such that $\mathrm{Hom}_{\bs{J}^{g}\cap G^{\tau}}(\Lambda^{g},1)\neq 0$.
			
			(1) Changing $g$ by another representative in the same $\bs{J}$-$G^{\tau}$ double coset, we may assume $\gamma\in B^{\times}$;
			
			(2) The dimension $\mathrm{dim}_{R}\mrhom_{J^{1g}\cap G^{\tau}}(\eta^{g},1)=1$;
			
			(3) There is a unique quadratic character $\chi$ of $\bs{J}^{g}\cap G^{\tau}$ trivial on $J^{1g}\cap G^{\tau}$, such that
			$$\mathrm{Hom}_{J^{1g}\cap G^{\tau}}(\eta^{g},1)\cong\mathrm{Hom}_{\boldsymbol{J}^{g}\cap G^{\tau}}(\boldsymbol{\kappa}^{g},\chi^{-1})\cong R.$$
			Moreover 
			$$\mathrm{Hom}_{\bs{J}^{g}\cap G^{\tau}}(\Lambda^{g},1)\cong\mathrm{Hom}_{\bs{J}^{g}\cap G^{\tau}}(\bs{\kappa}^{g},\chi^{-1})\otimes\mathrm{Hom}_{\bs{J}^{g}\cap G^{\tau}}(\bs{\rho}^{g},\chi);$$
			
			(4) The element $\gamma\in\bs{J}$, thus under the assumption of (1), $\gamma\in B^{\times}\cap\bs{J}=E^{\times}\mfb^{\times}$.
			
		\end{proposition}
		
		\begin{proof}
			
			We sketch the proof which follows from that of Theorem \ref{thmtype} (actually we have the same theorem if $\tau=\tau_{1}$). Using (\ref{eqthetatauint}) and the fact that $\tau(H^{1g})=\tau_{1}(H^{1})^{\varepsilon^{-1}\tau(g)}=H^{1\varepsilon^{-1}\tau(g)}$ and $(\theta\circ\tau)^{\tau(g)}=(\theta\circ\tau_{1})^{\varepsilon^{-1}\tau(g)}=(\theta^{-1})^{\varepsilon^{-1}\tau(g)}$, we have $$(\theta^{\varepsilon^{-1}\tau(g)})^{-1}|_{H^{1\varepsilon^{-1}\tau(g)}\cap H^{1g}}=(\theta\circ\tau)^{\tau(g)}|_{\tau(H^{1g})\cap H^{1g}}=\theta^{g}\circ\tau|_{\tau(H^{1g})\cap H^{1g}}=(\theta^{g})^{-1}|_{H^{1\varepsilon^{-1}\tau(g)}\cap H^{1g}},$$
			which means that $\gamma$ intertwines $\theta$, or in other words $\gamma\in JB^{\times}J$. The following lemma follows from the same proof of Lemma \ref{lemdoucos}, once we replace $\gamma$ there with our $\gamma$ here and $\tau$ there with $\tau_{1}$.
			
			\begin{lemma}
				
				There exist $y\in J=J(\mfa,\beta)$ and $b\in B^{\times}$, such that $\gamma=\tau_{1}(y)by$.
				
			\end{lemma}
			
			Thus we change $g$ by $y^{-1}g$ and then the corresponding $\gamma=b\in B^{\times}$, which proves (1). For (2), we write
			$$\delta(x):=(\tau(g)g^{-1})^{-1}\tau(x)\tau(g)g^{-1}=\gamma^{-1}\tau_{1}(x)\gamma\quad \text{for any}\ x\in G$$
			an involution on $G$, then by definition we have
			$$\mrhom_{J^{1g}\cap G^{\tau}}(\eta^{g},1)\cong\mrhom_{J^{1}\cap G^{\delta}}(\eta,1),$$
			and
			\begin{equation*}
				\gamma\delta(\gamma)=\gamma\gamma^{-1}\tau_{1}(\gamma)\gamma=1.
			\end{equation*}
			Moreover, by direct calculation we have
			\begin{equation*}
				\delta(H^{1})=(\tau(g)g^{-1})^{-1}H^{1\varepsilon^{-1}}\tau(g)g^{-1}=H^{1\gamma}\quad\text{and}\quad\theta\circ\delta=(\theta^{-1})^{\varepsilon^{-1}\tau(g)g^{-1}}=(\theta^{-1})^{\gamma}.
			\end{equation*}
			So using Proposition \ref{Propheisdelta}, we finish the proof of (2).
			
			Using (2) and the same argument of Proposition \ref{lemkappa} we get the statement (3), except the part $\chi$ being quadratic. To finish that part, since
			$$\tau_{1}(\tau_{1}(g)\varepsilon^{-1})\varepsilon^{-1}(\tau_{1}(g)\varepsilon^{-1})^{-1}=g\varepsilon\tau_{1}(g)^{-1}=(\tau_{1}(g)\varepsilon^{-1}g^{-1})^{-1}=\gamma^{-1}\in B^{\times},$$
			we may replace $g$ with $\varepsilon^{-1}\tau(g)=\tau_{1}(g)\varepsilon^{-1}$ in the statement (3) to get a unique character $\chi'$ of $\boldsymbol{J}^{\varepsilon^{-1}\tau(g)}\cap G^{\tau}$ trivial on $J^{1\varepsilon^{-1}\tau(g)}\cap G^{\tau}$.
			Moreover, using the facts $\tau(\boldsymbol{J})=\boldsymbol{J}^{\varepsilon^{-1}}$, $\tau(J)=J^{\varepsilon^{-1}}$, $\tau(J^{1})=J^{1\varepsilon^{-1}}$ and $\tau(H^{1})=H^{1\varepsilon^{-1}}$ and Lemma \ref{LemmaunitaryJJ0}, it is easy to show that \begin{equation}\label{eqjtau=j0tau1}
				\boldsymbol{J}^{g}\cap G^{\tau}=\boldsymbol{J}^{\varepsilon^{-1}\tau(g)}\cap G^{\tau}=J^{g}\cap G^{\tau}=J^{\varepsilon^{-1}\tau(g)}\cap G^{\tau}
			\end{equation}
				As a result, $\chi$ and $\chi'$ are characters defined on the same group $\boldsymbol{J}^{g}\cap G^{\tau}=\boldsymbol{J}^{\varepsilon^{-1}\tau(g)}\cap G^{\tau}$. We have the following lemma similar to Proposition \ref{Propchichi'}:
				
				\begin{lemma}\label{lemmachichi'}
					
					We have $\chi=\chi'$.
					
				\end{lemma}
				
				\begin{proof}
					
					We write $\delta$ for the involution defined as above. By \S \ref{subsectiontypes}, we have $\gamma\in I_{G}(\eta)=I_{G}(\kappa^{0})$ and
					$$\mathrm{dim}_{R}(\mathrm{Hom}_{J\cap J^{\gamma}}(\kappa^{0\gamma},\kappa^{0}))=\mathrm{dim}_{R}(\mathrm{Hom}_{J^{1}\cap J^{1\gamma}}(\eta^{\gamma},\eta))=1,$$
					where $\kappa^{0}=\bs{\kappa}|_{J}$.
					By direct calculation, we have $J^{1}\cap G^{\delta}=J^{1\gamma}\cap G^{\delta}$ as a subgroup of $J^{1}\cap J^{1\gamma}$ and $H^{1}\cap G^{\delta}=H^{1\gamma}\cap G^{\delta}$.
					Using statement (2) for $g$ and $\varepsilon^{-1}\tau(g)$ respectively, we get $$\mathrm{dim}_{R}\mathrm{Hom}_{J^{1g}\cap G^{\tau}}(\eta^{g},1)=\mathrm{dim}_{R}\mathrm{Hom}_{J^{1\varepsilon^{-1}\tau(g)}\cap G^{\tau}}(\eta^{\varepsilon^{-1}\tau(g)},1)=1.$$
					By Proposition \ref{Propvarphibijec}, for $0\neq\varphi\in\mathrm{Hom}_{J^{1}\cap J^{1\gamma}}(\eta^{\gamma},\eta)=\mathrm{Hom}_{J^{1g}\cap J^{1\varepsilon^{-1}\tau(g)}}(\eta^{\varepsilon^{-1}\tau(g)},\eta^{g})$, the map
					\begin{align*}
						f_{\varphi}:\mathrm{Hom}_{J^{1g}\cap G^{\tau}}(\eta^{g},1)&\longrightarrow\mathrm{Hom}_{J^{1\varepsilon^{-1}\tau(g)}\cap G^{\tau}}(\eta^{\varepsilon^{-1}\tau(g)},1),\\
						\lambda\quad&\longmapsto\quad \lambda\circ\varphi,
					\end{align*}
					is bijective\footnote{Noting that $J^{1g}\cap G^{\tau}=(J^{1}\cap G^{\delta})^{g}$ and $J^{1\varepsilon^{-1}\tau(g)}\cap G^{\tau}=(J^{1\gamma}\cap G^{\delta})^{g}$, thus $\mathrm{Hom}_{J^{1g}\cap G^{\tau}}(\eta^{g},1)=\mathrm{Hom}_{J^{1}\cap G^{\delta}}(\eta,1)$ and $\mathrm{Hom}_{J^{1\varepsilon^{-1}\tau(g)}\cap G^{\tau}}(\eta^{\varepsilon^{-1}\tau(g)},1)=\mathrm{Hom}_{J^{1\gamma}\cap G^{\delta}}(\eta^{\gamma},1)$.}. 
					If we choose
					$$0\neq \lambda\in\mathrm{Hom}_{J^{1g}\cap G^{\tau}}(\eta^{g},1)\quad \text{and}\quad 0\neq \lambda':=f_{\varphi}(\lambda)=\lambda\circ\varphi\in\mathrm{Hom}_{J^{1\varepsilon^{-1}\tau(g)}\cap G^{\tau}}(\eta^{\varepsilon^{-1}\tau(g)},1),$$
					then for any $v$ in the representation space of $\eta$ and any $x\in J^{g}\cap G^{\tau}=J^{\varepsilon^{-1}\tau(g)}\cap G^{\tau}$, using the similar argument to (\ref{eqchi'lambda'chilambda}) we have
					\begin{align*}
						\chi'(x)^{-1}\lambda'(v)
						=\chi(x)^{-1}\lambda'(v).
					\end{align*}
					Since $v$ and $x\in J^{g}\cap G^{\tau}=J^{\varepsilon^{-1}\tau(g)}\cap G^{\tau}$ are arbitrary, we have $\chi'|_{J^{\varepsilon^{-1}\tau(g)}\cap G^{\tau}}=\chi|_{J^{g}\cap G^{\tau}}$. Combining with (\ref{eqjtau=j0tau1}) we finish the proof of the lemma.
					
				\end{proof}
				
				To prove that $\chi$ is quadratic, we first assume that $\mathrm{char}(R)=0$. Using the similar argument to (\ref{eqhometakappaiso}) we have the following isomorphism
				\begin{align*}
					\mathrm{Hom}_{J^{1\varepsilon^{-1}\tau(g)}\cap G^{\tau}}(\eta^{\varepsilon^{-1}\tau(g)},1)
					\cong\mathrm{Hom}_{\boldsymbol{J}^{\varepsilon^{-1}\tau(g)}\cap G^{\tau}}(\boldsymbol{\kappa}^{\varepsilon^{-1}\tau(g)},\chi\circ\tau).
				\end{align*}
				Using the above lemma and the uniqueness of $\chi'$, we have $\chi\circ\tau=\chi^{-1}$. Since $\chi$ is defined on $\boldsymbol{J}^{g}\cap G^{\tau}=J^{g}\cap G^{\tau}$ which is $\tau$-invariant, we have $\chi\circ\tau=\chi$, thus $\chi^{2}=\chi(\chi\circ\tau)=1$. When $\mathrm{char}(R)=l>0$ the same argument in Proposition \ref{propchiquad} can be used directly.
				
				Finally using (3) and the distinction of the simple type, we have $\mathrm{Hom}_{\bs{J}^{g}\cap G^{\tau}}(\bs{\rho}^{g},\chi)\neq 0$. Then the proof of (4) is the same of that in \S \ref{subsectionproofofthmdisctype}, once we replace $\gamma$ there with our $\gamma$ here.
				
			\end{proof}
			
			\begin{corollary}\label{corJGtau'}
				
				For $g\in G$ such that $\mathrm{Hom}_{\bs{J}^{g}\cap G^{\tau}}(\Lambda^{g},1)\neq 0$, we may change $g$ by anothor representative in the same $\bs{J}$-$G^{\tau}$ double coset, such that
				$$\gamma=\begin{cases}I_{m}\quad \text{or}\quad \varpi_{E}I_{m} &\quad \text{if}\ E/E_{0}\ \text{is unramified;}\\
					I_{m}\quad \text{or}\quad \mathrm{diag}(1,...,1,\epsilon)\quad \text{or}\quad \varpi_{E}J_{m/2} &\quad \text{if}\ E/E_{0}\ \text{is ramified,}\end{cases}$$
				as an element in $\mrgl_{m}(E)\cong B^{\times}\hookrightarrow G$, where $\epsilon\in\mfo_{E_{0}}^{\times}-\mrn_{E/E_{0}}(\mfo_{E}^{\times})$
				
			\end{corollary}
			
			\begin{proof}
				
				We have proved that $\gamma=\tau_{1}(g)\varepsilon^{-1} g^{-1}\in B^{\times}\cap\bs{J}=E^{\times}\mathfrak{b}^{\times}$. Changing $g$ up to multiplying by an element in $E^{\times}$ which doesn't change the double coset
				it represents, we may assume $\gamma\in\mfb^{\times}$ or $\varpi_{E}\mfb^{\times}$. Using Proposition \ref{PropJOH} and changing $g$ up to multiplying by an element in $\mfb^{\times}$ on the left, we may assume that $\gamma=\varpi_{E}^{\alpha}$, and from the uniqueness we must have $\varpi_{E}^{\alpha}=I_{m}$ or $\varpi_{E}I_{m}$ when $E/E_{0}$ is unramified, and $\varpi_{E}^{\alpha}=I_{m}$ or $\mathrm{diag}(1,...,1,\epsilon)$ or $\varpi_{E}J_{m/2}$ when $E/E_{0}$ is totally ramified.
				
			\end{proof}
			
			Thus for $g\in G$ as above, we get
			$$\mathrm{Hom}_{J\cap G^{\delta}}(\bs{\rho}|_{J},\chi^{g^{-1}})\cong\mathrm{Hom}_{\bs{J}\cap G^{\delta}}(\bs{\rho},\chi^{g^{-1}})\cong\mathrm{Hom}_{\bs{J}^{g}\cap G^{\tau}}(\bs{\rho}^{g},\chi)\neq 0.$$
			Write $H=J\cap G^{\delta}/J^{1}\cap G^{\delta}$ for the subgroup of $\mrgl_{m}(\bs{l})\cong J/J^{1}$, which, from the expression of $\gamma$ in Corollary \ref{corJGtau'}, is either a unitary subgroup, or an orthogonal subgroup, or a symplectic subgroup of $\mrgl_{m}(\bs{l})$. Moreover we have
			$$\mathrm{Hom}_{H}(\overline{\rho},\overline{\chi'})\neq 0,$$
			where $\overline{\rho}$ is a cuspidal representation of $\mrgl_{m}(\bs{l})$ whose inflation is $\bs{\rho}|_{J}$ and $\overline{\chi'}$ is a quadratic character of $H$ whose inflation is $\chi^{g^{-1}}|_{J\cap G^{\delta}}$.
			
			When $H$ is unitary which also means that $E/E_{0}$ is unramified, by Lemma \ref{Lemmachiext} (or more precisely its argument) $\overline{\chi'}$ can be extended to a quadratic character of $\mrgl_{m}(\bs{l})$. Thus $\overline{\rho}\overline{\chi'}^{-1}$ as a cuspidal representation of $\mrgl_{m}(\bs{l})$ is distinguished by $H$, and thus it is $\sigma$-invariant by Proposition \ref{propfindiscinv}. The quadratic character $\overline{\chi'}$ must be $\sigma$-invariant, thus $\overline{\rho}$ is also $\sigma$-invariant, or by Proposition \ref{PropGelKafin}, $\overline{\rho}^{\tau_{1}}\cong\overline{\rho}^{\vee}$. Thus both $\bs{\kappa}$ and $\bs{\rho}$ are $\tau_{1}$-selfdual, which means that $\Lambda$ and $\pi$ are $\tau_{1}$-selfdual. By Proposition \ref{PropGelKa}, $\pi$ is $\sigma$-invariant.
			
			When $H$ is orthogonal which also means that $E/E_{0}$ is totally ramified, comparing the central character as in \S \ref{subsectiondisccriram} we have $\overline{\rho}(-I_{m})=\mathrm{id}$. Thus $\bs{\rho}^{\tau_{1}}|_{J}=\bs{\rho}(\,^{t}\cdot^{-1})|_{J}\cong\bs{\rho}|_{J}$ by Proposition \ref{PropGelKafin} and $\bs{\rho}(\tau_{1}(\varpi_{E}))=\bs{\rho}(-\varpi_{E})=\bs{\rho}(\varpi_{E})$, which means that $\bs{\rho}$ is $\tau_{1}$-selfdual, finishing the proof as above.
			
			Finally by Proposition \ref{Propsympdisc} and the fact that $\mathrm{Sp}_{m}(\bs{l})$ equals its derived subgroup, the case where $H$ is symplectic never occurs, which ends the proof of Theorem \ref{Thmdistgalinvgen}.
			
			\begin{remark}
				
				Combining Theorem \ref{Thmdistgalinvgen} with the argument in \cite{feigon2012representations}, section 6, we may further prove that an irreducible \emph{generic} representation $\pi$ of $G$ distinguished by a unitary subgroup $G^{\tau}$ is $\sigma$-invariant.
				
			\end{remark}

			\newcommand{\etalchar}[1]{$^{#1}$}

		\end{document}